\documentclass[a4paper,12pt,reqno]{amsart}

\usepackage[utf8]{inputenc}
\usepackage{latexsym}
\usepackage[normalem]{ulem}
\usepackage{amssymb, xcolor}
\usepackage{amsmath,hyperref}
\usepackage{enumerate, enumitem, verbatim}
\usepackage{tikz,url,pgf,amscd,enumerate,multicol,tikz-cd, soul}%
\usetikzlibrary{decorations.pathmorphing}
\usepackage[margin=1.9cm]{geometry}

\newtheorem{theorem}{Theorem}[section]
\newtheorem{lemma}[theorem]{Lemma}
\newtheorem{proposition}[theorem]{Proposition}

\newtheorem{example}[theorem]{Example}

\theoremstyle{definition}
\newtheorem{remark}[theorem]{Remark}
\newtheorem*{remark*}{Remark}

\newtheorem{definition}[theorem]{Definition}
\newtheorem{definitions}[theorem]{Definitions}

\newtheorem{remarks}[theorem]{Remarks}

\def\la{\lambda}
\def\La{\Lambda}
\def\N{\mathbb N}

\def\E{\mathcal E}

\newcommand{\NN}{\mathbb{N}}
\newcommand{\ZZ}{\mathbb{Z}}
\newcommand{\RR}{\mathbb{R}}
\newcommand{\Cc}{\mathcal{C}}
\newcommand{\F}{\mathcal{F}}
\newcommand{\G}{\mathcal{G}}

\newcommand{\inv}{^{-1}}
\newcommand{\sse}{\subseteq}
\newcommand{\mc}{\mathcal}

\newcommand{\wt}[1]{\widetilde{#1}}

\DeclareMathOperator{\ev}{\text{ev}}

\def\r{{\textsf r}}
\def\l{{\textsf l}}
\def\t{{\textsf t}}
\def\b{{\textsf b}}

\title[Insplitting for 2-graphs and textile systems]{Relating insplittings of 2-graphs and of textile systems}

\author[Brooker]{Samantha Brooker}

\address{Samantha Brooker, Department of Mathematics, Virginia Tech, Blacksburg, VA 24061-1026, USA}\email{bsamantha24@vt.edu}

\author[Ganesan]{Priyanga Ganesan}

\address{Priyanga Ganesan, Department of Mathematical Sciences, University of California San Diego, La Jolla, CA 92093, USA.}\email{pganesan@ucsd.edu}

\author[Gillaspy]{Elizabeth Gillaspy}
\address{Elizabeth Gillaspy,  Department of Mathematical Sciences, University of Montana, Missoula, MT 59812-0864, USA} \email{elizabeth.gillaspy@mso.umt.edu} 

\author[Lin]{Ying-Fen Lin}

\address{Ying-Fen Lin, Mathematical Sciences Research Centre, Queen's University Belfast, Belfast BT7 1NN, United Kingdom}\email{y.lin@qub.ac.uk}

\author[Pask]{David Pask} 

\address{David Pask \\ School of Mathematics and
Applied Statistics  \\
The University of Wollongong\\
NSW  2522\\
AUSTRALIA} \email{david.a.pask@gmail.com}

\author[Plavnik]{Julia Plavnik}

\address{Julia Plavnik, Department of Mathematics, Indiana University, Bloomington, IN 47405, USA \\
and Department of Mathematics and Data Science,  Vrije Universiteit Brussel, 1050 Brussels, BELGIUM}
\email{jplavnik@iu.edu} 

\date{\today}



\begin{document}

\tableofcontents

\begin{abstract}

    The graphical operation of insplitting is key to understanding conjugacy of shifts of finite type (SFTs) in both one and two dimensions. In this paper, we consider two approaches to studying 2-dimensional SFTs: textile systems and rank-2 graphs.      Nasu's textile systems describe all two-sided 2D SFTs up to conjugacy \cite{johnson-madden, Aso}, whereas the 2-graphs  (higher-rank graphs of rank 2) introduced by Kumjian and Pask yield associated $C^*$-algebras \cite{Yuxiang-Tang}. Both models have a naturally-associated notion of insplitting (introduced for textile systems in \cite{johnson-madden} and for 2-graphs in \cite{EFGGGP}).  We show   that these notions do not coincide, raising the question of whether insplitting a 2-graph induces a conjugacy of the associated one-sided 2-dimensional SFTs.

    Our first main result shows how to reconstruct 2-graph insplitting using textile-system insplits and inversions,   and consequently proves that  2-graph insplitting induces a conjugacy of dynamical systems. We also present several other facets of the relationship between 2-graph insplitting and textile-system insplitting.  Incorporating an insplit of the ``bottom'' graph of the textile system turns out to be key to this relationship.     By articulating the connection between operator-algebraic and dynamical notions of insplitting in two dimensions, this article lays the groundwork for a $C^*$-algebraic framework for classifying one-sided conjugacy in higher-dimensional SFTs.
\end{abstract}

\newpage
\maketitle
 
\section{Introduction}

The interplay between symbolic dynamics, directed graphs, and the $C^*$-algebras of directed graphs has long been a fruitful
area of research, revealing deep structural connections across these domains.  In this paper, we
explore the relationship between the dynamical notion of insplitting and the structure of the resulting $C^*$-algebras in two dimensions. This lays the groundwork for a $C^*$-algebraic framework for studying one-sided conjugacy
for higher-dimensional dynamical systems.  

While the research presented in 
this article was motivated by $C^*$-algebraic concerns, our main
theorems are phrased purely in dynamical terms; we relegate the $C^*$-algebraic considerations to
Appendix \ref{sec:$C^*$} and future papers.

In both one and two dimensions, shifts of finite type
(SFTs) are precisely those shifts that possess a graphical description. The representation of
1-dimensional SFTs via directed graphs has been known since the early development of the
field (cf.~\cite[Theorem 2.3.2]{lind-marcus}), and this correspondence has been highly successful in describing
conjugacy, flow equivalence and related invariants of dynamical systems. In
the 2-dimensional setting, textile systems (consisting of a pair of directed graphs ``woven'' together) serve a similar purpose. 
Introduced by Nasu in \cite{nasu}, textile systems were shown to model all 2-dimensional SFTs by Johnson and Madden in \cite[Proposition 2.3]{johnson-madden} (see also \cite[Theorem 4.1]{Aso}).
They also characterized conjugacy for SFTs via textile systems in \cite[Corollary 3.10]{johnson-madden} and  \cite[Theorem 3.1]{Aso}.  As in the 1-dimensional case, two 2-dimensional SFTs are conjugate---i.e., dynamically isomorphic---if and only if their textile systems can be transformed into one another via the graphical operations of insplitting, outsplitting and inversion (and their inverses).  (The operation of inversion interchanges the horizontal and vertical directions in the 2-dimensional SFT, and has no analogue in the 1-dimensional setting.)

Directed graphs and their higher-dimensional generalizations have also led to significant advances in the theory of $C^*$-algebras, in which $C^*$-algebraic versions of in- and out-splitting have played a major role. 
Directed graphs (or, equivalently, 1-dimensional SFTs)
yield key examples of $C^*$-algebras, which are tractable because their internal algebraic structure mirrors the structure
of the underlying graph (cf.~\cite{kpr, bhrs,raeburn-szyman}). As $C^*$-algebras are analytic objects with very few
{\em a priori} structural restrictions, the tight structural link between a directed graph and its $C^*$-algebra makes graph $C^*$-algebras a key source of examples for $C^*$-algebraists (cf.~\cite{hong-syman, mann-raeburn-sutherland}).
Indeed, the program of classifying $C^*$-algebras up to isomorphism (or the weaker invariant of Morita equivalence) has seen spectacular success for graph $C^*$-algebras \cite{sorensen-first, errs-cuntz, errs, aer-2}: two graph $C^*$-algebras are isomorphic if and only if their underlying directed graphs can be converted into each other via a finite list of moves, including insplitting (and outsplitting in the case of Morita equivalence).  Graph $C^*$-algebras and related constructions have also led to new insights in symbolic dynamics, by providing new characterizations of dynamical invariants like flow equivalence,
continuous orbit equivalence, one- and two-sided conjugacy, and eventual (one-sided) conjugacy, see
 \cite{cuntz-krieger, giordano-putnam-skau, matsumoto-matui, matsumoto, CEOR, aer-1, carlsen-rout, brix-carlsen, brix-carlsen-2, brix-balanced, ABCE, carlsen-doron-eilers}.

The tight structural link between a directed graph and its $C^*$-algebra also means that many interesting $C^*$-algebras, such as the rotation
algebras \cite{evans-sims} or any $C^*$-algebra with torsion $K_1$ group \cite{cuntz-K-thy, raeburn-szyman} cannot arise as graph
$C^*$-algebras. Higher-rank graphs, or $k$-graphs, were introduced by Kumjian and the fifth-named author in \cite{kp} to provide combinatorial models
for such $C^*$-algebras. 
Like their 1-dimensional cousins the directed graphs, $k$-graphs also have a strong link with dynamical systems, see \cite{KumjianPask3}. In \cite{Yuxiang-Tang}, Tang
established a bijection between 2-graphs (or 2-dimensional higher-rank graphs) and left-resolving (LR) textile systems. 
More generally, Carlsen and Rout initiated in \cite{carlsen-rout} the $C^*$-algebraic analysis of two-sided
conjugacy, eventual one-sided conjugacy, and continuous orbit equivalence for the dynamical
systems arising from $k$-graphs.

Unlike textile systems, 
$k$-graphs are categories, by definition \cite{kp}. That is, in a $k$-graph, concatenation of ``paths'' is really composition of morphisms in a category. For consistency, therefore, we think of paths in a $k$-graph as pointing right-to-left: two paths $p,q$ are composable with product $pq$ if $s(p) = r(q)$.  This convention unfortunately leaves us at odds with the literature on textile systems; we discuss the translation in Remark \ref{rmk:square-convention} below.

Both  graph $C^*$-algebras and $k$-graph $C^*$-algebras are built from the underlying one-sided SFT \cite{kprr, kp}.  This underlies the fact (discovered by Bates and the fifth-named author in \cite{bates-pask}) that insplitting\footnote{In \cite{bates-pask}, as in much of the dynamics literature, the roles of in- and out-splitting are reversed from the discussion in this paper. This is due to our view of $k$-graphs as categories, which leads to a different choice of convention between \cite{bates-pask} and the current paper (which follows \cite{kp}) in the definition of the graph $C^*$-algebra.} yields an isomorphism of graph $C^*$-algebras, while outsplitting yields only a Morita equivalence. Indeed, the asymmetry between the $C^*$-algebraic implications of in- and out-splitting persists for $k$-graph $C^*$-algebras \cite{EFGGGP, listhartke}. This is the main reason why we concern ourselves with one-sided shifts.

However, for 2-dimensional dynamical systems, Theorem \ref{thm:JM-insplit-not-LR} shows that these natural $C^*$-algebraic formulations of in- and out-splitting
do not agree with the textile-system in- and out-splitting of \cite{johnson-madden}. On the other hand, in the 1-dimensional setting, a dynamically-constructed isomorphism of $C^*$-algebras frequently implies that the underlying one-sided shift spaces are conjugate (cf.~\cite{cuntz-krieger, brix-carlsen, ABCE}). This raises the natural question of whether, in the 2-dimensional setting as well, the $C^*$-algebraic
definition of insplitting should yield a conjugacy of the associated one-sided SFTs. Equivalently, are the 2-dimensional versions of insplitting from \cite{EFGGGP}  and \cite{johnson-madden} related by more than mere nomenclature? 

Theorem \ref{thm:6.1} below gives an affirmative answer to this question. More
precisely, suppose $T$ is an LR textile system (Definition \ref{def:textile}); then Theorem \ref{thm:PST} shows that $T$
determines a 2-graph $\Lambda_T$ (Definition \ref{def:2graph}), and hence a $C^*$-algebra by \cite{kp}. Theorem \ref{thm:6.1}
says that a 2-graph insplit on $\Lambda_T$ (as in \cite{EFGGGP}) can be recreated via textile-system insplits
and inversions on $T$. In particular, the 2-graph insplitting of \cite{EFGGGP} yields a conjugacy of
one-sided, 2-dimensional shifts of finite type. Additionally, we present various constructions in section \ref{sec:JM-ghinsp} that recover the textile-system in-splitting \cite{johnson-madden} from a combination of directed graph in-splitting and 2-graph in-splitting.

\subsection*{Structure of the paper}
As one-sided 2-dimensional dynamics have not received as much attention in the literature as their two-sided 1-dimensional counterparts, we begin in Section \ref{sec:prelim} with definitions and basic results about these shifts.  In particular, we define conjugacy and higher block presentations for one-sided 2-dimensional SFTs and show that any conjugacy is induced by a factor map.

Section \ref{sec:textile} recalls the definition of a textile system, and in Section \ref{sec:catmodels} 2-graphs
are briefly introduced, as well as their associated dynamical systems. The link between 2-graphs and textile systems is established in Theorem \ref{thm:PST}.
We review the definitions of insplitting for textile systems and for 2-graphs in Section \ref{sec:2ginsplit}.

Our main results lie in Sections \ref{sec:ghinsp-JM} and \ref{sec:JM-ghinsp}. 
In addition to establishing, in Theorem \ref{thm:6.1}, that 2-graph insplitting can be understood via textile system insplitting, 
we provide several perspectives on the relationship between 2-graph insplitting and textile system insplitting.  
Theorem \ref{thm:lr-insplit=2g-insplit} identifies when a single textile-system insplitting could alternately be interpreted as a 2-graph insplitting, and Theorem \ref{thm:Priyanga} shows how to interpret 2-graph insplitting as a textile insplit together with an extra insplit of the ``bottom'' graph $E$ from the textile system. We also demonstrate that certain insplittings of the base graph simultaneously yield both types of insplitting. Indeed, these different perspectives are equivalent, as we establish in Section \ref{subsec: unification}. We conclude the paper with a brief discussion of the $C^*$-algebraic consequences of  2-graph in- and out-splitting, as well as the dynamical consequences of 2-graph outsplitting, in Appendix \ref{sec:$C^*$}.

\subsection*{Acknowledgments}
This research was supported by the Australian Research Council, Banff Research Station, Alexander von Humboldt Foundation (to J.P.), International Centre for Mathematical Sciences, the US National Science Foundation (grant DMS-1800749 to E.G. and grant DMS-2146392 to J.P.), and the Simons Foundation (Award 889000 as part of the Simons Collaboration on Global Categorical Symmetries to J.P.)

The fifth author would like to thank Aidan Sims, Yuxiang Tang and Samuel Webster for interesting conversations about textile systems. %

This research was initiated at the Women in Operator Algebras II workshop at BIRS, and continued during a Research in Groups stay at ICMS.  We thank both of these organizations for the congenial research atmosphere they provided.

\section{Two-dimensional non-invertible symbolic dynamics}
\label{sec:prelim}

We use the convention that $0 \in \NN$ (see Remark~\ref{rmk:niscat} below).
The standard generators of $\NN^2$ (resp.\ $\ZZ^2$) are denoted $\varepsilon_1,\varepsilon_2$, and we write $n_i$ for the $i^{\textrm{th}}$ coordinate of $n \in \NN^2$. We define a partial order on $\NN^2$ by $m\leq n$ if $m_i \leq n_i$ for $i=1, 2$.
For $m,n \in \mathbb{N}^2$, we write $m \vee n$ for their coordinatewise maximum and $m \wedge n$ for their coordinatewise minimum. 

For completeness, we include a brief overview of two-dimensional one-sided shifts of finite type, as systematic treatments of this topic are scarce in the literature. We assume familiarity with the basic theory of one-dimensional shifts of finite type, as found in \cite{lind-marcus} and \cite{Kitchens}.

To fix notation, we begin with a discussion of one-dimensional dynamical systems.

\subsection{One-dimensional preliminaries}

\begin{definitions}
A \emph{directed graph} is a quadruple $E= (E^0,E^1,r,s)$, where $E^0, E^1$ are sets of vertices and edges, respectively, and $r, s: E^1 \to E^0$ are maps giving the range and source of each edge respectively. We say that $E$ is {\em source-free} or {\em has no sources} if $r$ is onto; it is \emph{essential} if both $r,s$ are onto.
\label{defs:directed}
\end{definitions} 

\begin{example}
\label{ex:bouquet}
From a finite set $X$, we construct an essential (directed) graph $B_X$, the ``bouquet of $X$ loops''. This graph has one vertex and $B_X^1 = \{ e_x: x \in X\}$; the range and source maps are uniquely defined.
\end{example}

For a directed graph $E$ and $n \in \NN$, $E^n$ denotes the directed paths of length $n$ in $E$:
\begin{equation}
    \label{eq:def-En}
    E^n = \{ e_1 \cdots e_n: e_i \in E^1, \ r(e_i) = s(e_{i-1}) \ \text{for all }i\}.
\end{equation}
We denote the collection of finite directed paths in $E$ by $E^* := \bigcup_{n \ge 0} E^n$. The range and source maps $r, s: E^1 \to E^0$ extend naturally to $E^*$ 
(where $r(v)=s(v)=v$ for all $v \in E^0$). 
Given $v,w \in E^0$ and $F \subseteq E^*$ define $vF := r^{-1}(v) \cap F$, $Fw := s^{-1}(w) \cap F$, and $vFw := vF \cap Fw$. 

\begin{remark} \label{rmk:1graph2b}
The set $E^*$ becomes a category where the objects are the vertices $E^0$ and the domain of a morphism (path) is its source 
and the range of a morphism (path) is its categorical range. 
In \cite{maclane} this is called the \emph{free (or path) category generated by $E$} since there are no relations required.
\end{remark}

\begin{remark} \label{rmk:niscat}
One can identify $\NN$ with the path category of the directed graph $B_X$ where $X$ has one element.
\end{remark}

Let $E=(E^0,E^1,r,s)$ be an essential directed graph. Then,
as in \cite[Definition 2.2.5 and Section 13.8]{lind-marcus}, we define
\begin{equation}\label{eq:inf-path-space-defn}
\textsf{X}_E^+ = \{ x=(x_n) \in (E^1)^{\NN} : r(x_n) = s(x_{n-1}) \text{ for all } n \in \NN \}
\end{equation}
to be the one-sided, one-dimensional shift of finite type associated to the directed graph $E$.

\subsection{One-sided, two-dimensional shifts of finite type}

The following definition is adapted from \cite[\S 2]{Schmidt}.

Let $\mathcal{A}$ be a finite set (alphabet).  Define
$\mathcal{A}^{\NN^2} = \{ x : \NN^2 \to \mathcal{A}\}$. A typical point $x \in \mathcal{A}^{\NN^2}$ is written as $x= ( x_n : n \in \NN^2)$, where $x_n \in \mathcal{A}$ denotes the value of $x$ at $n \in \NN^2$. 

We equip $\mathcal A^{\NN^2}$ with the product topology.  That is, the basic open sets of $\mathcal A^{\NN^2}$ are the {\em cylinder sets} $Z(x[0,m]),$ for $x \in \mathcal A^{\NN^2}$ and $m \in \NN^2$, defined by
\[ Z(x[0,m]) = \{ y \in \mathcal A^{\NN^2} : x_n = y _n \text{ for all } 0 \leq n \leq m\}.\]
There is an action $\sigma$ of $\NN^2$ on $\mathcal{A}^{\NN^2}$ defined by
\begin{equation} \label{eq:shift}
( \sigma^m (x))_n = x_{n+m}
\end{equation}
for every $x=(x_n) \in \mathcal{A}^{\NN^2}$ and $m \in \NN^2$. 

\begin{definitions}
A subset $X \subset \mathcal{A}^{\NN^2}$ is called \emph{shift-invariant} if $\sigma^m(X) =X$ for every $m \in \NN^2$, and a closed shift-invariant subset $X \subset \mathcal{A}^{\NN^2}$ is called a \emph{(one-sided) subshift}. If $X \subset \mathcal{A}^{\NN^2}$ is a subshift then we write $\sigma = \sigma_X$ for the restriction of the shift action \eqref{eq:shift} to $X$, and we say that $(X,\sigma)$ is a \emph{shift space}. 

For any subset $S \subset \NN^2$, we denote by $\pi_S : \mathcal{A}^{\NN^2} \to \mathcal{A}^S$ the projection map which restricts every $x \in \mathcal{A}^{\NN^2}$ to $S$. A subshift $X \subset \mathcal{A}^{\NN^2}$ is a {\em one-sided subshift of finite type (SFT)} if there exist a finite set $Z \subset \NN^2$ and a finite collection $\mathcal F$ of functions $f: Z \to \mathcal A$, such that
\begin{equation} \label{eq:sftdef}
X = \{ x \in \mathcal{A}^{\NN^2} : \pi_Z(\sigma^m(x)) \in \mathcal F
\text{ for every } m \in \NN^2 \} .
\end{equation}
\end{definitions}

\noindent
This definition resembles the ``allowed blocks" model of shift spaces described in \cite[Section 2.1]{lind-marcus}: A configuration $x \in \mathcal{A}^{\NN^2}$ belongs to the shift space $X$ if whenever $x$ is restricted to some translate of $Z$ the values of $\mathcal{A}$ found there match those allowed in $Z$.

\subsection{Sliding-block codes, conjugacy, inversion} \label{sec:sbc}

\begin{definition}[Blocks, block map]
Let $m \in \NN^2$ and 
$X \subset \mathcal{A}^{\NN^2}$ be a one-sided SFT. The collection of $m$-\emph{blocks} in $X$ is defined by
\[
B_m (X) = \{ x[a,a+m] : 
x \in X , 
a \in \NN^2 \} .
\]
\end{definition}

\noindent
Without loss of generality $B_0(X)= \mathcal{A}$.

The following is adapted from Kitchens \cite[page 25]{Kitchens}. Let $X$ be a SFT over the alphabet $\mathcal A$, and $Y$ be a SFT over the alphabet $\mathcal B$. Fix $\ell \in \NN^2$. An $\ell$-\emph{block map} $\phi: (X, \sigma_X) \to (Y, \sigma_Y)$ is determined by a map $\Phi : B_\ell (X) \to B_0 (Y)$ that is consistent with $X$ and $Y$ in the following sense: For every $x \in X, a \in \NN^2$, and $i=1,2$, we have
\begin{equation}
    \label{eq:l-block-map}
\Phi ( x [a,a+\ell] ) \Phi ( x[a+\varepsilon_i , a+\ell+\varepsilon_i ] ) \in B_{\varepsilon_i} (Y), 
\end{equation}

\noindent
then 
the associated $\ell$-block map $\phi  : (X,\sigma_X) \to (Y,\sigma_Y)$
is given by
\[
(\phi (x))_n = \Phi(x[ n, n+ \ell]) .
\]

\begin{lemma}
    \label{lem:block-map-cts}
    Every block map $\phi: (X, \sigma_X) \to (Y, \sigma_Y)$ is continuous and shift commuting, that is, $\sigma_Y \circ \phi = \phi \circ \sigma_X.$
\end{lemma}
\begin{proof}
Suppose $\phi: X \to Y$ is an $\ell$-block map.  Fix $n \in \NN^2$ and consider a basic open set $Z(y[0, n])$ in $Y$. 
Note that 
\begin{align*} 
\phi^{-1}(Z(y[0,n]))&  = \{ x \in X : \Phi(x[j, j+\ell]) = y_j \ \text{for all } 0\leq j\leq n\}\\
&= \bigcup \{ Z(x[0, n+\ell]): \Phi(x[j, j+\ell]) = y_j \ \text{for all } 0\leq j\leq n\},
\end{align*}
being a union of cylinder sets, is open. Hence $\phi$ is continuous.

To see that $\sigma_Y \circ \phi = \phi \circ \sigma_X,$ let $x\in X$, $n \in \NN^2$ and  $i= 1,2.$  Then 
\[ \sigma_Y^{\varepsilon_i}(\phi(x))_n = \phi(x)_{n+\varepsilon_i} = \Phi(x[n+ \varepsilon_i, n + \varepsilon_i + \ell]), \]
while $\phi(\sigma^{\varepsilon_i}_X(x))_n = \phi((x_{k+\varepsilon_i})_k)_n = \Phi(x[n+\varepsilon_i, n+\varepsilon_i +\ell]).$
\end{proof}

The argument of \cite[Theorem 1.4.9]{Kitchens} generalizes to show that any continuous shift commuting map between shift spaces is a block map.
In particular, every conjugacy 
can be viewed as a block map.

\noindent
\textbf{Standing assumption}
In the same way that Lind and Marcus reduce arguments to using $1$-block maps for shifts of finite maps in \cite[\S 2]{lind-marcus}, we may do the same here:
 By considering the sizes of all the forbidden blocks and then applying a suitable higher block presentation (cf. \cite[p.27]{Kitchens}), we may assume that 
\begin{equation} \label{eq:bestF}
Z = \{ 0,1 \}^2 \subset \NN^2 
\end{equation}
in \eqref{eq:sftdef}.

\begin{definition}[Conjugacy]
\label{def:conjugacy}
 The shift spaces $(X,\sigma_X)$ and $(Y,\sigma_Y)$ are \emph{conjugate} if there is a homeomorphism $h : X \to Y$ satisfying 
 \[ h \circ \sigma_X = \sigma_Y \circ h.\]
 
\end{definition}

In two dimensions, the choice of horizontal and vertical directions is somewhat arbitrary. 
Hence (following \cite{johnson-madden}) we define an operation on a shift space which swaps these directions.

\begin{definition}[Inversion] \label{dfn:inversion}
Fix an alphabet $\mathcal A$.
The {\em inversion} of the shift space $(X, \sigma)$ is $(\theta(X), \sigma')$, where 
\begin{align}\label{eq:inversiondef}
    \theta(X)= \{x: \NN^2 \to \mathcal A \ |  \text{ there is }  y \in X 
   & \text{ such that } \\ 
   &\text{ for all } n= (n_1, n_2) \in \NN^2,
    x_{(n_1, n_2)} = y_{(n_2, n_1)}\}.
    \nonumber
\end{align} 
We write $x = \theta(y)$.  The shift map is also inverted in $\theta(X)$; that is, $\sigma'$ is given by 
 $$(\sigma')^{(m_1, m_2)}(\theta(y))_{(n_1, n_2)} := \theta(y)_{(n_1 + m_2, n_2 + m_1)} = y_{(n_2 + m_1, n_1 + m_2)}.$$

\end{definition}

For $x\in \theta(X)$, we write $x= \theta(y)$ for some $y\in X$, and observe that $X$ is a SFT if and only if $\theta(X)$ is.

\begin{lemma}
Let $(X, \sigma)$ be a shift space and $( \theta(X), \sigma' )$ be the inversion of $X$.  Then $\theta : X \to \theta(X)$ is a conjugacy.
\end{lemma}

\begin{proof}
Evidently, $\theta^{-1} $ and $\theta$ are given by the same formula; in particular, $\theta$ is a bijection. Moreover, $\theta \circ \sigma = \sigma' \circ \theta,$ since for any $m, n \in \NN^2,$
\begin{align*} 
(\sigma') ^{(m_1, m_2)}(\theta(y))_{(n_1, n_2)} &  =  \theta(y)_{(n_1 + m_2, n_2 + m_1)} = y_{(n_2 + m_1, n_1 + m_2)}, \text{ while }\\
  \theta (\sigma^{(m_1, m_2)}(y))_{(n_1, n_2)} & = (\sigma^{(m_1, m_2)}(y))_{(n_2, n_1)} = y_{(n_2 + m_1, n_1 + m_2)}.
  \end{align*}
  It is straightforward to see that $\theta$ is continuous; for instance, the inverse image of the cylinder set 
  $Z(x[(0,(m_1, m_2))])$ is 
  \begin{eqnarray*}
  \{ y \in X: y_{(n_2, n_1)} = x_{(n_1, n_2)} \text{ for all } n = (n_1, n_2) \text{ with } (n_1, n_2) \leq (m_1, m_2)\},
  \end{eqnarray*}
which is $Z(\theta(x)[(0, (m_2, m_1))])$, another cylinder set. We conclude that $\theta$ is a homeomorphism, which completes the proof.
\end{proof}

\section{Textile systems and their dynamical systems} \label{sec:textile}

Just as every one-dimensional SFT can be realized as the edge shift of a directed graph \cite[Theorem 2.3.2]{lind-marcus}, every 2-dimensional SFT can be realized using a pair of graphs woven together into the form of a {\em textile system}. Johnson and Madden establish in \cite[Proposition 2.3]{johnson-madden} that every (two-sided) $2$-dimensional shift of finite type is conjugate to a textile system SFT, by using a higher block argument similar to that given in Section~\ref{sec:sbc}. 

The definition that follows is an adaptation of the original definition given by \cite{nasu}.

\begin{definition}\label{def:textile}
A {\em textile system} $T=(p,q:F \to E)$ consists of two directed graphs $F, E$ and two graph homomorphisms $p, q: F \to E$ such that the  function $A : F^1 \to F^0 \times E^1 \times F^0 \times E^1$ given by
\[
 A(f) = (r(f), p(f), s(f), q(f)) 
\]
is injective. 
\end{definition}

\begin{remark}
In \cite{johnson-madden}, Johnson and Madden restrict their attention to textile systems where $E=B_X$ for some set $X$. We find this condition too restrictive for our aims and shall not enforce this restriction in this paper.   
\end{remark}

A \emph{square} in the textile system $T=(p,q :F \to E)$ is a four-sided object, whose edges are labeled using the edges of $E$  and vertices of $F$ as shown below. 
\begin{equation} \label{eq:square-labels} 
\begin{array}{l}
\begin{tikzpicture}[scale=1.3]
\node at (-3.3,0.7) {\text{The square $T_f$}};
\node at (-3.35,0.3) {\text{labeled by $f \in F^1$:}};
\node at (0.5,0.5) {$T_f$};
\node[inner sep=0.8pt] (00) at (0,0) {$.$};
\node[inner sep=0.8pt] (10) at (1,0) {$.$}
	edge[->] node[auto,black] {$q(f)$} (00);
\node[inner sep=0.8pt] (01) at (0,1) {$.$}
	edge[->] node[auto,black,swap] {$r(f)$} (00);
\node[inner sep=0.8pt] (11) at (1,1) {$.$}
	edge[->] node[auto,black,swap] {$p(f)$} (01)
	edge[->] node[auto,black] {$s(f)$} (10);

\end{tikzpicture}
\end{array}
\end{equation}

\noindent
Let $W_T = \{ T_f : f \in F^1 \}$ denote the collection of squares associated to $T$.

\begin{remark} \label{rmk:square-convention}
Here our conventions for labeling the squares associated to a textile system are different from those in \cite{nasu} and \cite{johnson-madden}. 
This is the first visible sign of our categorical preferences mentioned in the introduction.
To be precise, we use the convention that $r(f)$ lies on the left of the square denoting the edge $f$ and $s(f)$ on the right so that (with concatenation of paths as composition of morphisms) $F^*$ is a category. To maintain consistency with Nasu's definition of LR textile systems (in particular, so that we measure the path lifting properties of $p$ with respect to $r$ and not $s$ in Definitions~\ref{dfn:resolving}) this choice forces us to place $p(f)$ on the top of the square $T_f$ and $q(f)$ on the bottom.
\end{remark}

\begin{definitions} \label{dfn:resolving}
A graph homomorphism $p: F \to E$ is said to have {\em (unique) $r$-path lifting} if, for all $v \in F^0,$ whenever $p(v) = r(e) \in E^0$ {for some $e\in E^1$}, there exists (a unique) $f \in F^1$ with $p(f) = e$ and $r(f) = v.$  Having unique $r$-path lifting is sometimes referred to being {\em left resolving} in the dynamics literature.
Similarly, $p$ has {\em (unique) $s$-path lifting} if, for all $v\in F^0$, whenever $p(v) = s(e) \in E^0$ for some $e\in E^1$, there exists (a unique) $f \in F^1$ with $p(f) = e$ and $s(f) = v$. Having unique $s$-path lifting is sometimes referred to as being {\em right resolving} in the dynamics literature

A textile system $T= (p, q: F\to E)$ is called {\em LR} if $p$ has unique $r$-path lifting and $q$ has unique $s$-path lifting. 
\end{definitions}

\begin{example}\label{ex:tex}
Consider the directed graphs shown below:
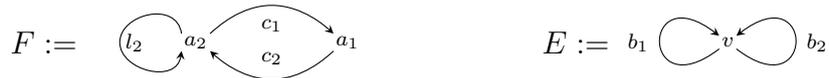
\[
\begin{tikzpicture}[yscale=0.75, >=stealth]
\node[inner sep=0.8pt] (v) at (0,0) {$\scriptstyle a_2$};
\node[inner sep=0.8pt] (w) at (2,0) {$\scriptstyle a_1$};

\node[inner sep=2pt, anchor = west] at (-1,0 ){$\scriptstyle l_2$};
\draw[->] (v.north west)
        .. controls (-0.25,0.75) and (-1,0.5) ..
        (-1,0)
        .. controls (-1,-0.5) and (-0.25,-0.75) ..
        (v.south west);
\draw[->] (v.north east)
    parabola[parabola height=0.5cm] (w.north west);

\node[inner sep=2pt, anchor = north] at (1,0.5) {$\scriptstyle c_1$};
\draw[->] (w.south west)
    parabola[parabola height=-0.5cm] (v.south east);

\node[inner sep=2pt, anchor = south] at (1,-0.5) {$\scriptstyle c_2$};
\node at (-2,0) {$F:=$};
\node at (5,0) {$E:=$};

\node[inner sep=0.8pt] (u) at (7,0) {$\scriptstyle v$}
	edge[loop,->,in=45,out=-45,looseness=15] node[auto,black,swap] {$\scriptstyle b_2$} (u)
	edge[loop,->,in=135,out=-135,looseness=15] node[auto,black] {$\scriptstyle b_1$} (u);
\end{tikzpicture}
\]

\noindent Define $p,q : F \to E$ by $p(c_1)= b_1, p(c_2)= p(l_2) = b_2$, and $q(c_1)= q(l_2) = b_1, q(c_2)=b_2$.
Then $T = (p,q:F \to E)$ is a non-LR textile system since $p(l_2)=p(c_2)=b_2$ and $r(l_2)=r(c_2)=a_2$. 
Worse yet, $p$ fails to have $s$-path lifting, since there is no edge in $F^1$ with $s(f)= a_1$ and $p(f)= b_1$, and $q$ fails to have $r$-path lifting since there is no edge $f \in F^1$ with $r(f) = a_1$ and $q(f) = b_2$.
\end{example}

\begin{lemma}\label{lem:upl gives texsys}
Let $E$ and $F$ be directed graphs and $p,q:F\to E$ be graph morphisms. If $p$ has unique $r$- or unique $s$-path lifting, then $(p,q:F \to E)$ is a textile system for any $q$. Similarly, if $q$ has unique $r$- or unique $s$-path lifting, then $(p,q:F \to E)$ is a textile system for any $p$.
\end{lemma}

\begin{proof}
Suppose that $p$ has unique $r$-path lifting. We check that $f\mapsto ((r(f),p(f),s(f),q(f))$ is injective. Let $f , g \in F^1$ be such that $((r(f),p(f),s(f),q(f)) = ((r(g),p(g),s(g),q(g))$. Then $p(f) = p(g)$ and $r(f)=r(g)$, hence $f=g$. The other cases are proved \emph{mutatis mutandis}.
\end{proof}

\noindent
However, Example~\ref{ex:tex} shows that the converse of Lemma~\ref{lem:upl gives texsys} does not hold.

The following is adapted from \cite[\S 4]{Schmidt}.

\begin{definition}[Textile tiling]\label{dfn:textiletiling}
Let $T = ( p,q : F \to E)$ be a 
textile system, with the associated collection $W_T = \{ T_f : f \in F^1 \}$ of squares.  A (one-sided) \emph{textile 
weaved by $T$} 
is a covering of $\RR_{\ge 0}^2$ by translating copies of elements of $W_T$ with non-overlapping interiors such that the following conditions are satisfied:
\begin{enumerate}[label = (\roman*)]
\item every corner of each square lies in $\NN^2 \subset \RR_{\ge 0}^2 $; 
\item two squares are only allowed to touch along edges 
in the following sense:
\begin{itemize}
\item[(T1)] $s(f) = r(f')$ whenever $T_f, T_{f'}$ are horizontally adjacent squares with $T_f$ to the left of $T_{f'}$, or 
\item[(T2)] $p(f) = q(f')$ if $T_f, T_{f'}'$ are vertically adjacent with $T_{f'}$ above $T_f$.
\end{itemize}
\end{enumerate}
We write $\textsf{X}_T^+ $ for the subshift of $W_T^{\N^2}$ (with the usual horizontal and vertical shift maps) consisting of all textiles weaved by $T$. 
That is, 
\begin{equation}
\label{eq:textiles}
\textsf{X}_T^+ = \{ x \in W_T^{\NN^2} : s(x_{v})= r(x_{v+\varepsilon_1}), p(x_w)=q(x_{w+\varepsilon_2}), \;\; \forall v,w \in \mathbb{N}^2\} .
\end{equation}
\end{definition}

\begin{remark} \label{rmk:textilesysissft}
It is straightforward to show that $\textsf{X}_{T}^+ \subset W_T^{\NN^2}$ is closed and shift invariant. 
When the graph $F$ is finite, $\textsf{X}_T^+$ is a shift of finite type as it is of the form \eqref{eq:sftdef} for $Z=\{0,1\}^2$:
the restrictions of \eqref{eq:textiles} specify which $2\times 2$ blocks of tiles are allowed.

For simplicity we shall identify $W_T$ with $F^1$, that is, the alphabet of the shift space $\textsf{X}_T^+$ is $F^1$.
\end{remark}

We now proceed to give sufficient conditions for $\textsf{X}_T^+$ to be nonempty.

\begin{theorem}\label{lem:basic}
    [cf.~\cite[Proposition 4.6]{Yuxiang-Tang}  ]\label{thm:cantile}
Let $T = (p,q : F \to E)$ be a textile system in which 
$F$ is source-free, $q|_{F^0}$ is onto, 
and $q$ has $r$-path lifting. Then $\textsf X_T^+$ is nonempty. 
\end{theorem} 

\begin{minipage}[l]{0.4\textwidth}
\[
\begin{tikzpicture}[scale=1.5]
        
\node[inner sep=1pt, circle] (10) at (1,0) {$.$};
\node[inner sep=1pt, circle] (11) at (1,1) {\tiny $u_0$};

\node[inner sep=1pt, circle] (22) at (2,2) {$.$};
\node[inner sep=1pt, circle] (21) at (2,1) {\tiny $u_1$};
        
\node[inner sep=1pt, circle] (30) at (3,0) {$.$};
\node[inner sep=1pt, circle] (31) at (3,1) {$.$};
    
\node[inner sep=1pt, circle] (40) at (4,0) {$.$};
\node[inner sep=1pt, circle] (41) at (4,1) {$.$};

\node[inner sep=1pt, circle] (12) at (1,2) {$.$};
\node[inner sep=1pt, circle] (20) at (2,0) {$.$};
\node[inner sep=1pt, circle] (32) at (3,2) {$.$};
\node[inner sep=1pt, circle] (42) at (4,2) {$.$};

    
\draw[->] (11)--(10) node[pos=0.5, left] {\tiny\color{black}$ $};
    
    
\draw[->] (20)--(10) node[pos=0.5, below] {\tiny\color{black}$ $};
\draw[->] (21)--(11) node[pos=0.5, below] {\tiny\color{black}$p(x_0^0)$};
\draw[->] (21)--(20) node[pos=0.5, left] {\tiny\color{black}$ $};
    
    
\draw[->] (30)--(20) node[pos=0.5, below] {\tiny\color{black}$ $};
\draw[->] (31)--(21) node[pos=0.5, below]{\tiny\color{black}$p(x_1^0)$};
\draw[->] (31)--(30) node[pos=0.5, left] {\tiny\color{black}$ $};
    
    
\draw[->] (40)--(30) node[pos=0.5, below] {\tiny\color{black}$ $};
\draw[->] (41)--(31) node[pos=0.5, above] {\tiny\color{black}$ $};
\draw[->] (41)--(40) node[pos=0.5, left] {\tiny\color{black}$ $};

\node at (1.5,1.2) {\tiny $q(x_0^1)$};
\node at (1.5, 0.4) {\tiny $x_0^0$};
\node at (2.5, 0.5) {\tiny $x_1^0$};
\node at (3.5, 0.5) {\tiny $x^0_2$};
\node at (4.4,0.5) {$\ldots$};

\node at (2.5,1.2) {\tiny $q(x_1^1)$};
\node at (1.5,1.6) {\tiny $x^1_0$};
\node at (2.5,1.5) {\tiny $x^1_1$};
\node at (3.5,1.5) {\tiny $x^1_2$};
\node at (4.4,1.5) {$\ldots$};

\node at (1.5,2.5) {$\vdots$};
\node at (2.5,2.5) {$\vdots$};
\node at (3.5,2.5) {$\vdots$};

\draw[->] (12)--(11) node[pos=0.5, right] {\tiny\color{black}$ $};
    
    
\draw[->] (22)--(12) node[pos=0.5, below] {\tiny\color{black}$ $};

\draw[->] (22)--(21) node[pos=0.5, left] {\tiny\color{black}$  $};
  
  
\draw[->] (42)--(32) node[pos=0.5, below] {\tiny\color{black}$ $};
\draw[->] (42)--(41) node[pos=0.5, left] {\tiny\color{black}$ $};
  
    
\draw[->] (32)--(22) node[pos=0.5, below] {\tiny\color{black}$ $};
  \draw[->] (32)--(31) node[pos=0.5, left] {\tiny\color{black}$ $};

\end{tikzpicture}
\]
\end{minipage}%
\begin{minipage}[l]{0.6\textwidth}

{\em Proof. }
We first show that for any infinite path $x^0 \in \textsf X^+_F$ in $F$, there is $x^1 \in \textsf X_F^+$ such that $q(x^1) = p(x^0).$

As $F$ is source-free, $\textsf{X}^+_F$ defined in \eqref{eq:inf-path-space-defn} is nonempty so we may choose $x^0 = (x_i^0 )_{i \in \NN}  \in \textsf{X}^+_F$. Consider $x_0^0 \in F^1$ and let $u_0 = r( p(x_0^0))$. Since $q$ has $r$-path lifting and $q|_{F^0}$ is onto,
we may choose $w_0 \in F^0$ with $q (w_0) =u_0$ and $x^1_0 \in F^1$ such that $r ( x^1_0) = w_0$ and $q (x^1_0 ) =p ( x_0^0 )$. Let $w_1 = s(x^1_0) $; as $q$ is a graph homomorphism, 
\[ u_1 := q(w_1)  = s(q(x^1_0)) = s(p(x_0^0)) = r(p(x_1^0)).\]
The fact that $q$ has $r$-path lifting now implies the existence of $x^1_1 \in F^1$ with $q(x^1_1) = p(x^0_1)$ and $r(x^1_1) = w_1$. Let $w_2 = s(x^1_1)$. Continuing inductively, we construct $x^1 = (x_i^1)_{i\in \N} \in X^+_F$ with $q(x^1_i) = p(x_i^0)$ for all $i$.

\end{minipage}
Following the same construction, we can create $(x^n)_{n \in \N} \subseteq \textsf X_F^+$ such that $p(x^n) = q(x^{n+1})$ for all $n\in \N$.  That is, by construction, $x = (x^i_j)_{(j,i)\in \N^2} \in \textsf{X}_T^+. \hfill  \square$

The following definition follows the notation of \cite{johnson-madden}; the textile system $\widehat{T}$ is called the dual textile system in \cite[page 15]{nasu}.

\begin{definition} [Inverted textile system]\label{dfn:inverted-textile}
Let $T=(p,q:F \to E)$ be a textile system. Define directed graphs $\widehat{E} = ( E^0 , F^0 , q , p )$  and $\widehat{F} = ( E^1 , F^1 , q , p )$. Then $\hat{p} = ( s_E , s_F)$, $\hat{q} = ( r_E , r_F )$ are graph morphisms from $\widehat{F}$ to $\widehat{E}$ and $\widehat{T}= ( \hat{p} , \hat{q} : \widehat{F} \to \widehat{E}  )$  is a textile system, called the \emph{inverted textile system}. 
\end{definition}

\noindent 

We record the following Lemma, whose proof is straightforward.

\begin{lemma}
For any textile system $T$, we have $(\widehat{T})\widehat{\vphantom{T}} = T$, and (in the notation of Definition \ref{dfn:inversion}) $\theta({\sf X}_T^+) = {\sf X}_{\widehat{T}}^+$. 
    Furthermore, if $T=(p,q:F \to E)$ is LR, then the inverted system $\widehat{T} = ( \hat{p},\hat{q}: \widehat{F}\to  \widehat{E} )$ is LR.
\end{lemma}

\section{2-graphs and their dynamical systems} \label{sec:catmodels}

Theorem~\ref{thm:PST} below shows that all LR textile systems can be described as higher-rank graphs of rank two, or $2$-graphs.  
As we mentioned in the introduction, graphical moves such as insplitting have been developed in \cite{EFGGGP} for 2-graphs (and indeed for higher-rank graphs in general), and a main goal of this paper is to clarify the relationship between 2-graph insplitting and textile system insplitting.  Thus, this section defines 2-graphs and establishes the equivalence between $2$-graphs and LR textile systems.

\subsection{Rank two graphs and two-colored graphs}\label{sec:2coloredpresentation}

In what follows, we view $\N^2$ as a category with one object, namely 0; composition of morphisms is given by coordinate-wise addition.  Thus, the standard notation ``$n \in \N^2$'' indicates that $n$ is a morphism in $\N^2$.  For consistency with this perspective, for a general category $\Lambda$, we will write ``$\lambda \in \Lambda$'' to indicate that $\lambda$ is a morphism in $\Lambda$. We will identify the objects in $\Lambda$ with the identity morphisms.

\begin{definition}\cite[Definitions 1.1]{kp}\label{def:2graph}
A \emph{rank two graph} (or a \emph{$2$-graph}) is a countable category $\Lambda$ with a \emph{degree} functor $d:\Lambda \to \NN^2$ satisfying the \emph{factorization property}: if $\lambda\in\Lambda$ and $m,n\in\NN^2$ are such that $d(\lambda)=m+n$, then there are unique $\mu,\nu \in \Lambda$ with $d(\mu)=m$, $d(\nu)=n$ and $\lambda=\mu\nu$. 
\end{definition}

Given $m \in \NN^2$, we define $\Lambda^m := d^{-1}(m)$. 

The factorization property then allows us to identify $\Lambda^0$ (the morphisms of degree 0) with the identity morphisms (equivalently, the objects  $\operatorname{Obj}(\Lambda)$) of $\Lambda$; given any $\lambda \in \Lambda$ there are unique $v, w \in \Lambda^0$ with $\lambda = v \lambda w$.  We write $v =: r(\lambda)$ and $w = :s(\lambda)$, and we call $\Lambda^0$ the \emph{vertices} of $\Lambda$. A vertex $v \in \Lambda^0$ is a \emph{sink} if $s^{-1} (v) = \{ v\} $, and $w \in \Lambda^0$ is a \emph{source} if $r^{-1} (w) = \{w\}$. 

For $v,w \in \Lambda^0$ and $F \subseteq \Lambda$, define $vF := r^{-1}(v) \cap F$, $Fw := s^{-1}(w) \cap F$, and $vFw := vF \cap Fw$. If $v \Lambda^m$ is finite for all $m \in \mathbb{N}^2$ and $v \in \Lambda^0$ then $\Lambda$ is called \textit{row-finite.} If $v \Lambda^m$ and $\Lambda^m v$ are nonempty for all $m\in \mathbb{N}^2$ and $v \in \Lambda^0$ then $\Lambda$ is called \textit{essential.}

By the factorization property, for each $\lambda\in\Lambda$ and $m \leq n \leq d(\lambda)$, we may write $\lambda=\lambda'\lambda''\lambda'''$, where $d(\lambda')=m, d(\lambda'') = n-m$ and $d(\lambda''')=d(\lambda)-n$; then $\lambda (m,n) :=\lambda''$.
We write 
\[ \ev_{m, n}(\lambda) := \lambda(m,n).\]

For more information about $k$-graphs, see \cite{kp,RSY1} for example.

\begin{example}
Following \cite{kp} let $\Omega_2 $ be the category with 
$\operatorname{Mor}(\Omega_2)= \{ (m,n) \in \mathbb{N}^2  \times \NN^2: 0 \leq m \le n \}$; 
composition of morphisms is given by $ ( \ell,m) (q,n)= \delta_{m,q}(\ell, n).$ 
Set $d ( m , n ) = n - m$; then $d: \Omega_2 \to \NN^2$ is
a functor and $( \Omega_2 , d )$ is a row-finite $2$-graph. The vertices $\Omega_2^0 = \{ (m,m) : m \in \mathbb{N}^2 \}$ may be identified with $\mathbb{N}^2 = \operatorname{Obj}(\Omega_2)$. 
\end{example}

One can also profitably think of a 2-graph as a quotient of the path category of a directed graph with 2 colors of edges.  To be precise, for $i = 1,2$ we call $d^{-1}(\varepsilon_i) =\Lambda^{\varepsilon_i} $ the \textit{edges of color $i$} in $\Lambda$, and call $\Lambda^1 := \Lambda^{\varepsilon_1} \sqcup \Lambda^{\varepsilon_2}$  the \textit{edges} in $\Lambda$. Then every element of $\Lambda$ can be written as a composition of colored edges: given $\lambda \in \Lambda$ with $d(\lambda) = n = (n_1, n_2)$, writing 
\[ n = \overbrace{\varepsilon_1 + \cdots + \varepsilon_1}^{n_1} + \overbrace{\varepsilon_2 + \cdots + \varepsilon_2}^{n_2} \]

and the factorization property tells us that $\lambda = e_1 \cdots e_{n_1} f_1 \cdots f_{n_2}$ for a unique collection of edges $e_i \in \Lambda^{\varepsilon_1}$ of color 1 and $ f_i \in \Lambda^{\varepsilon_2}$  of color 2.

To see why $\Lambda$ is a nontrivial quotient of the path category of the edge-colored directed graph $(\Lambda^0, \Lambda^{\varepsilon_1} \sqcup \Lambda^{\varepsilon_2}, r, s)$, observe that the representation $\lambda = e_1 \cdots e_{n_1} f_1 \cdots f_{n_2}$ depends on the order in which we write $n = (n_1, n_2)$ as a sum of generators of $\NN^2$.  For example, 
since $
 \varepsilon_1 + \varepsilon_2 = \varepsilon_2 + \varepsilon_1$, the factorization property implies that for any $\lambda \in \Lambda^{\varepsilon_1 + \varepsilon_2}$, there are unique $e, e' \in \Lambda^{\varepsilon_1}$ and $f, f' \in \Lambda^{\varepsilon_2}$ satisfying 
\[ \lambda = e f = f' e'.\]
Thus, to obtain $\Lambda$ from the path category of $(\Lambda^0, \Lambda^{\varepsilon_1} \sqcup \Lambda^{\varepsilon_2}, r, s)$, we need to take the quotient with respect to the equivalence relation $\sim$ generated by 
\[ ef \sim f'e' \, \text{ if } \, ef, f'e' \text{ represent the same element of } \Lambda.\]
The following makes this idea precise. 

\begin{definition}
\label{def:2-colored-graph}
A \emph{$2$-colored graph} $(G, d)$ is a directed graph $G=(G^0,G^1,r,s)$  along with a \emph{degree map} $d : G^1 \to \{\varepsilon_1,\varepsilon_2\}$.  We think of $\varepsilon_1, \varepsilon_2$ as the {\em colors} of the edges in $G$.
   
We extend the degree map to a map $d: G^* \to \NN^2$ on the path category of $G$ by 
\[ 
d(\lambda) = (m, n) \text{ if } \lambda \text{ has $m$ edges of degree $\varepsilon_1$ and $n$ edges of degree $\varepsilon_2$.}
\]
Let $\sim$ be an equivalence relation on $G^*$ that preserves the degree, range, and source of paths. 
Then by \cite[Theorems 4.4 and  4.5]{hazle-raeburn-sims-webster} the quotient $\Lambda= G^*/\sim$ of the path category of $G$ by an equivalence relation $\sim$ is a $2$-graph if and only if $\sim$ satisfies the following conditions:
\begin{enumerate}
\item[(KG0)] \textbf{Unique factorization:} If $\lambda \in G^*$ is a path such that $\lambda = \lambda_2 \lambda_1$, then the equivalence class $[\lambda] = [p_2] [p_1]$ whenever $p_i\in [\lambda_i]$ for $i=1,2$.
\item[(KG1)] \textbf{Non-degeneracy:} If $f, g\in G^1$ are edges, then $f \sim g \iff f = g.$
\item[(KG2)]\begin{minipage}{0.75\textwidth}
\textbf{Completeness:} Whenever $fg \in G^2$ 
is a path consisting of one edge of each color, there is a unique path $g' f' \sim fg$ with $d(g') = d(g)$ and $d(f') = d(f)$.
\end{minipage}
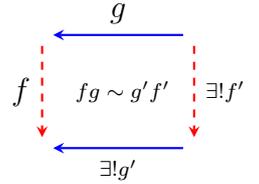
\begin{minipage}{0.25\textwidth}
\[ 
\begin{tikzpicture}[>=stealth,yscale=1.5,xscale=2]
\node at (0.53,0.5) {\tiny $fg \sim g'f' $
};
\node (00) at (0,0) {
};
\node (10) at (1,0) {$ $}
	edge[->,blue,thick] node[auto,black] {\tiny $\exists ! g'$} (00);
\node (01) at (0,1) {$ $}
	edge[->,red,dashed,thick] node[auto,black,swap] {$f$} (00);
\node (11) at (1,1) {$ $}
	edge[->,blue,thick] node[auto,black,swap] { $g$} (01)
	edge[->,red,dashed,thick] node[auto,black] {\tiny $\exists ! f'$} (10);
\end{tikzpicture}
\]
\end{minipage}
\end{enumerate}
    \end{definition}

\begin{remark}
\label{rmk:2-skel-2colograph}
If $\Lambda = G^*/\sim$ is a 2-graph, we call $(G,d)$ the {\em 1-skeleton} of $\Lambda$. 
As $\sim$ respects the degree, range, 
and source maps from $G$, these descend to maps on $\Lambda$. The fundamental building blocks of $\Lambda$ are the {\em  commuting squares} of $\Lambda = G^*/\sim$:
\[
\Cc_\Lambda = \{ (fg ,g'f') : fg \sim g'f',  \ d(f)=d(f') = \varepsilon_2, d(g) =d(g') = \varepsilon_1\} .
\]
\end{remark}

Indeed (cf.~\cite[Proposition 4.3]{pask-quigg-raeburn}), fixing a set $\mathcal C_\Lambda$  of commuting squares satisfying (KG2) leads to a uniquely-defined  equivalence relation $\sim$ on $G^*$ which respects degree, source, and range and also satisfies (KG0) and (KG1) (that is, an equivalence relation making $G^*/\sim$ into a 2-graph).  To be precise,  we guarantee (KG1) by defining 
$[e_1] = \{ e_1\}$ if $e_1 \in G^0 \cup G^1,$ and we ensure (KG0) by extending $\sim$ to be a symmetric and transitive relation on $G^*$.

\subsection{2-graphs and textile systems}
The perspective of 2-colored graphs enables us to articulate the link between 2-graphs and (LR) textile systems. 
 The following definition was first given in \cite{Yuxiang-Tang}.

 \begin{definition}[$2$-colored graph associated to a textile system]\label{dfn:2colgraph from T}
 Let $T = (p,q:F \to E)$ be a textile system. Define the 2-colored graph $(G_T,d)$ as follows: Let $(G_T)^0 = E^0$,  $(G_T)^1 = E^1 \sqcup F^0$. For $e \in E^1$ define
 $s  ( e ) = s_E (e)$, $r (e) = r_E (e)$ and $d (e) = \varepsilon_1$, and for $w \in F^0$ define
 $s (w) = p(w)$, $r(w) = q(w)$ and $d (w) = \varepsilon_2$. 
 
 We define an equivalence relation $\sim$ on $G_T^*$ by setting $[\alpha] = \{ \alpha\}$ for any $\alpha \in G_T^0 \cup G_T^1$; defining $\sim$ on 2-colored paths of length 2 by 
 \begin{equation} v e \sim e' w \text{ if and only if there exists } \ f \in F^1 \text{ with } r(f) = v, s(f) = w, p(f) = e, q(f) = e';
 \label{eq:squares}
 \end{equation}
 and extending $\sim$ to  paths of other lengths inductively by defining 
 $[\lambda \mu] = [\lambda] [\mu]$. 
 \end{definition}
That is, the commuting squares in $G_T$ correspond to the edges in $F,$ as follows:
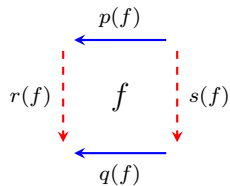
\[ 
\begin{tikzpicture}[>=stealth, scale=1.5]
\node (00) at (0,0) 
{};
\node (10) at (1,0) {$ $}
	edge[->,blue,thick] node[auto,black] {\tiny $q ( f)$} (00);
\node (01) at (0,1) {$ $}
	edge[->,red,dashed,thick] node[auto,black,swap] {\tiny $r ( f )$} (00);
\node (11) at (1,1) {$ $}
	edge[->,blue,thick] node[auto,black,swap] {\tiny $p(f)$} (01)
	edge[->,red,dashed,thick] node[auto,black] {\tiny $s(f)$}  (10);
 \node at (0.5, 0.5) {$f$};
\end{tikzpicture}
\]

\noindent
The next result, which is a major foundation for the research in this paper, is a generalization of \cite[Theorem 3.8]{Yuxiang-Tang}.

\begin{theorem} \label{thm:PST}
Let $T=(p,q:F \to E)$ be a textile system with $p,q$  surjective. Let $(G_T,d)$  
be the associated $2$-colored graph  and 
 $\sim$ be the equivalence relation on $G_T^*$ from Definition \ref{dfn:2colgraph from T}. 
Then 
$ \Lambda_T:= G_T^*/\sim$ is a 2-graph
if and only if $T$ is LR (that is, $p$ has unique $r$-path lifting and $q$ has unique $s$-path lifting).
\end{theorem}

\begin{proof}
First suppose that $p$ has unique $r$-path lifting and $q$ has unique $s$-path lifting. Let $\alpha\beta$ be a $2$-colored path in $G_T^2$. Suppose, without loss of generality, that 
$d(\alpha) = \varepsilon_1$ 
and $d(\beta) = \varepsilon_2$. 
By Definition~\ref{dfn:2colgraph from T} this implies $\alpha \in E^1$ and $\beta \in F^0$, and the fact that $\alpha \beta$ is a path in $G_T$ implies that $s_E(\alpha) = s_{G_T}(\alpha) =  r_{E_T}(\beta) = q(\beta)$. 

Since $q$ has unique $s$-path lifting, there is a unique $f\in F^1\beta$ with $q(f) = \alpha$. 
By definition of $\sim$, we have $r(f) p(f) \sim \alpha \beta$; the uniqueness of $f$ guarantees that $\sim$ satisfies (KG2).
If 
$\alpha$ has color 2 and $\beta$ has color 1 the same argument applies, using unique $r$-path lifting of $p$ in place of unique $s$-path lifting of $q$. 
Thus $\Lambda_T = G_T^*/\sim$ is a 2-graph.

Now suppose that 
$G_T^*/\sim$ is a 2-graph, that is, $\sim$ satisfies (KG2).
We will show that $p$ has unique $r$-path lifting. The argument for $q$ having unique $s$-path lifting is similar. Fix $e \in E^1$, $w \in F^0$  with $r_E(e) = p(w)$, that is, with $r_{G_T}(e) = s_{G_T}(w)$. Then 
$we$ is a 2-colored composable path in $G_T$, 
so by (KG2), there is a unique 2-colored path $e' v$ such that $d(e') = d(e)$, $d(v) = d(w)$,
and $we \sim e'v$.  That is, 
there exists a unique $f \in F^1$ with $r_F(f) = w$, and $p(f) =e$. So $p$ has unique $r$-path lifting. 
\end{proof}

In other words, an LR textile system $T = (p, q: F\to E)$ yields a 2-graph $\Lambda_T$ by identifying $F^1$ with $\Lambda_T^{\varepsilon_1 + \varepsilon_2}.$  Conversely, given a 2-graph $\Lambda$, we can define a textile system $T_\Lambda = (p, q: F_\Lambda \to E_\Lambda)$ by 
\begin{equation} 
\label{eq:textile-from-2graph}
E_\Lambda = (\Lambda^0, \Lambda^{\varepsilon_1}, r, s), \quad F_\Lambda = (\Lambda^{\varepsilon_2}, \Lambda^{\varepsilon_1 + \varepsilon_2}, \ev_{0, \varepsilon_2}, \ev_{\varepsilon_1, \varepsilon_1 + \varepsilon_2}), \quad p = \ev_{\varepsilon_2, \varepsilon_1 + \varepsilon_2}, \quad q = \ev_{0, \varepsilon_1}. \end{equation}
The uniqueness guaranteed by the factorization property implies that $T_\Lambda$ is an LR textile system, and we have $\Lambda_{T_\Lambda} = \Lambda$ and $T_{\Lambda_T} = T$.

 \begin{remark}
     \label{rmk:Lambda-T-range-source}
     By Definition \ref{dfn:2colgraph from T}, if $\Lambda = G^*/\sim$ is a 2-graph with associated textile system $T = (p,q: F \to E)$, then $\Lambda^0 = E^0$ and $\Lambda^{ 1} = E^1 \sqcup F^0.$ For any $v \in F^0$,
     \[ s_\Lambda(v) = p(v), \quad r_\Lambda(v) = q(v), \]
     whereas $s_\Lambda(e) = s(e), r_\Lambda(e) = r(e)$ for $e \in E^1.$
 \end{remark}
 
Next, we show that the notion of \textit{essentiality} translates well between LR textile systems and rank-2 graphs. 
Recall (Definition \ref{def:2graph}) that a 2-graph $\Lambda$ is called \textit{essential} if $v \Lambda^m$ and $\Lambda^m v$ are nonempty for all $m \in \mathbb{N}^2$ and $ v \in \Lambda^0$. We now introduce an analogous notion of \textit{essential} textile systems as follows:

\begin{definition}[Essential textile system] \label{def: essential textile system}
A textile system $T= (p,q: F \to E)$ is said to be \textit{essential} if $F$ is essential (as in Definitions \ref{defs:directed}) and $p$ and $q$ are surjective maps.

\end{definition}
Note that if $T$ is an essential textile system, it follows immediately that $E$ is also essential. We extensively use this fact in the following proposition to prove that essential  LR textile systems are equivalent to essential rank-2 graphs. 

\begin{proposition}\label{essential equivalence}
An LR textile system $T= (p,q: F \to E)$ is essential if and only if its corresponding 2-graph $\Lambda = \Lambda_T$ is essential.
\end{proposition}

\begin{proof}
First assume that the textile system $T= (p,q: F \to E)$ is essential. Since $p,q,r_E,s_E, r_F, s_F$ are all surjective, we have that for each $z \in \Lambda^0 = E^0$,
\begin{equation*}
z \Lambda^{(1,0)} = r_E^{-1}(z) \ne \emptyset, \;\;\;\;
 \Lambda^{(1,0)} z = s_E^{-1}(z) \ne \emptyset,\;\;\;\;
\end{equation*}
\begin{equation*} 
z \Lambda^{(0,1)}  = q^{-1}(z) \ne \emptyset,\;\;
 \Lambda^{(0,1)} z = p^{-1}(z) \ne \emptyset. 
 \end{equation*}
It follows that $z \Lambda^m$ and $\Lambda^m z$ are nonempty for all $m \in \mathbb{N}^2$. Thus, $\Lambda$ is essential when $T$ is essential.
 
Conversely, assume that $\Lambda$ is essential. 
Then, for each $z \in E^0$, we have 
\[ p^{-1}(z) = \Lambda^{(0,1)} z \ne \emptyset, \;\;\;\; q^{-1}(z) = z \Lambda^{(0,1)}  \ne \emptyset.\]
\[ r_E^{-1}(z) = z \Lambda^{(1,0)} \ne \emptyset, \;\;\;\; s_E^{-1}(z) = \Lambda^{(1,0)} z  \ne \emptyset.\]
Similarly, if $e \in E^1$ with $s_E(e) =z$, the fact that $z\Lambda^{(0,1)}\not= \emptyset$ implies that there exists $w \in F^0$ with $q(w) = z = s_E(e). $ As $T$ is LR, there exists (a unique) $f \in F^1$ with $q(f) = e$; that is, $q$ is surjective.   An analogous argument, using $r$-path lifting of $p$, will show that $p$ is surjective. 

 We have shown that $p$ and $q$ are surjective and $E$ is essential. Now, suppose $v \in F^0$ and $z = p(v)$. 
Since $E$ is essential, there exists $e \in E^1$ such that $r_E(e) = z = p(v)$.
By $r$-path lifting  of $p$, there exists $f \in F^1$ such that $p(f) = e$ and $r_F(f) = v$.
So, $r_F^{-1}(v) \ne \emptyset$ for any $v \in F^0$ and hence $r_F$ is surjective.
Similarly, suppose $v \in F^0$ and $z = q(v)$. 
Since $E$ is essential, there exists $e \in E^1$ such that $s_E(e) = z = q(v)$.
By $s$-path lifting  of $q$, there exists $f \in F^1$ such that $q(f) = e$ and $s_F(f) = v$.
So, $s_F^{-1}(v) \ne \emptyset$ for any $v \in F^0$ and hence $s_F$ is surjective.
Thus, $F$ is also essential.
\end{proof}

\subsection{The shift space associated to a rank two graph}

\begin{definitions}
Let $\Lambda$ be a $2$-graph. The \emph{one-sided infinite path space of $\Lambda$} is 
\[
\Lambda^\infty ~=~ \{ x : \Omega_2  \rightarrow \Lambda  : x \;
\mbox{is a degree preserving functor} \} .
\]

\noindent
For each $p \in \NN^2$, define $\sigma^p : \Lambda^\infty
\rightarrow \Lambda^\infty$ by $\sigma^p (x) (m,n) = x(m+p,n+p)$
for $x \in \Lambda^\infty$ and $(m,n) \in \Omega$. Note that
$\sigma^{p+q} = \sigma^p \circ \sigma^q$.

For $\lambda \in \Lambda$ define
\[
Z ( \lambda ) =  
\{ y  \in \Lambda^\infty: y ( 0 , d ( \lambda ) ) = \lambda \} .
\]
We endow $\Lambda^\infty$ with the topology generated by the basic open sets $\{ Z(\lambda): \lambda \in \Lambda\}$.  With this topology, the shift maps $\sigma^p$ are local homeomorphisms (cf.~\cite[Remarks 2.5]{kp}).
\end{definitions}

\noindent

Moreover, if we associate to $\Lambda$ the textile system $T_\Lambda$ as in Equation \eqref{eq:textile-from-2graph}, and construct the associated shift of finite type ${\sf X}_{T_\Lambda}^+$ as in Equation \eqref{eq:textiles}, we obtain the following result.

\begin{lemma}
Let $\Lambda$ be a 2-graph with $
\Lambda^{\varepsilon_1} \sqcup \Lambda^{\varepsilon_2}$ finite. Then the shift space $( \Lambda^\infty , \sigma )$ is finite-type.  In particular, we can identify $\Lambda^\infty$ with ${\sf X}^+_{T_\Lambda}$.
\label{lem:same-path-spaces}
\end{lemma}

\begin{proof}
Recall that ${\sf X}_{T_\Lambda}^+$ consists of all tilings of $\RR^2_{\geq 0}$, using squares from $F_\Lambda^1 = \Lambda^{\varepsilon_1 + \varepsilon_2}$, which satisfy Definition \ref{dfn:textiletiling}.  In particular, thanks to the factorization property in $\Lambda$ and the fact that adjacent squares in a tiling must coincide on their common edges, whenever $m \leq n$ in $\N^2$ and $x \in {\sf X}_{T_\Lambda}^+$, the squares $\{ x_{m+v}\}_{0\leq v \leq n-m}$ identify a unique morphism in $\Lambda$ of degree $n-m$.  We denote this morphism by $x(m,n),$ for coherence with the definition of $\Lambda^\infty.$ 

Indeed, we can define $\psi: {\sf X}_{T_\Lambda}^+ \to \Lambda^\infty$ by $\psi(x)(m,n) = x(m,n)$ whenever $m \leq n$, and the
inverse map $\varphi: \Lambda^\infty \to {\sf X}_{T_\Lambda}^+$ is given by (for $n \in \N^2$)
\[ \varphi(y)_n = y(n, n + (1,1)). \]
It is straightforward to check that $\psi$ and $\varphi$ are shift commuting and take cylinder sets to cylinder sets.  That is, $\psi$ is a conjugacy, giving the desired identification of $\Lambda^\infty$ with ${\sf X}_{T_\Lambda}^+$ as dynamical systems.
\end{proof}

\begin{remarks}
    [Dynamics results concerning higher-rank graphs] 
    We briefly summarize here some results in the literature related to applications of higher-rank graphs to dynamics.
\begin{enumerate}[leftmargin=19pt,label=(\roman*)]
\item In \cite{Skalski-Zacharias} the authors compute the topological entropy of the shift associated to a higher-rank graph. The factorization property ensures that it is zero. They also compute the entropy induced in the associated $C^*$-algebra (which is also zero).
\item In \cite{KumjianPask3} the authors study the dynamical system arising from a higher-rank graph. They show that the action of shift map on the path space of a  higher-rank graph is an expansive action. Additionally they show that if the higher-rank graph is primitive, then the shift map is topologically mixing. The existence of the Parry measure described in \cite[Proposition 4.2]{KumjianPask3} led to the work by others on the KMS states on higher-rank graphs (cf.~\cite{aHLRS,aHLRS14,FGKP,FGLP}).
\item In \cite{pask-raeburn-weaver} the authors use a set of basic data coming from algebraic information (inspired by the work of Schmidt) to create a set of 2-dimensional tiles. Using a standard overlap condition from dynamics they construct a $2$-graph from such tiles and show that, in many situations, the associated $C^*$-algebras are purely infinite and simple. Moreover, their shift spaces tally with the one coming from the data applied to the Schmidt examples (see \cite{Schmidt}). In \cite{Lewin-Pask} the authors verify that the entropy of these shifts are zero using different techniques to \cite{Skalski-Zacharias}. 
\item In \cite{carlsen-rout-kgraph,brownlowe-carlsen-whittaker,CEOR} the authors consider orbit equivalence of rank one and higher-rank dynamical systems.
\end{enumerate}
    
\end{remarks}

\section{Insplitting for textile systems and rank-2 graphs} \label{sec:2ginsplit}

In the literature, one can find both a definition for insplitting a textile system \cite{johnson-madden} and for insplitting a higher-rank graph \cite{EFGGGP}.  Unfortunately, when the higher-rank graph is the 2-graph associated to an LR textile system as in Theorem \ref{thm:PST}, these definitions do not agree.  We briefly review their definitions before moving on to the main results of this paper in Sections 
\ref{sec:ghinsp-JM} and \ref{sec:JM-ghinsp} below, which address the 
problem of reconciling the two notions of insplitting.

\subsection{Insplitting of textile systems} 
\label{sec:jm-insplits}


\begin{definition}\cite[Definition 2.4.7]{lind-marcus}
    For a directed graph $F$, an {\em insplitting partition} of $F$ is a partition of $r^{-1}(v)\subseteq F^1$, for each $v \in F^0$, into $m(v)$ (nonempty) subsets $\mathcal F_v^1, \ldots, \mathcal F_v^{m(v)}$, for some $m(v) \geq 1$.  The {\em insplit graph} $F_I$ has \[
F_I^0 = \{ v^i : v \in F^0 , 1 \le i \le m(v) \} \text{ and }
F_I^1 =\{ f^j : 1 \leq j \leq m(s(f)) , f \in F^1 \} = \bigcup_{f\in F^1} \bigcup_{i=1}^{m(s(f))} \{ f^i \}  ,
\]
with $s(f^j) = s(f)^j$ and $r(f^j) = r(f)^k$ if $f \in \mathcal F_{r(f)}^k$.
\label{def:graph insplitting}
\end{definition}

\vspace{3mm}

In \cite[\S 3]{johnson-madden}, Johnson and Madden give the definition of  insplitting for textile systems with $E = B_X$ being a one-vertex graph. 
Here we have adapted their definition to textile systems with 
no restrictions on $E$.

\begin{definition}[Textile insplitting] \label{dfn:textile-insplit}
Let $T = (p,q : F \to E )$ be a textile system.  
An {\em insplit textile system} is $T_I = (p_I, q_I: F_I \to E_I)$, where $E_I = E$ and $F_I$ is an insplit of $F$ as in Definition \ref{def:graph insplitting}. 

We have $p_I(f^j) = p(f), p_I(v^i) = p(v)$, and $q_I$ is similarly induced from $q$. 

We sometimes call $T_I$ a {\em Johnson--Madden insplit} of $T$, to distinguish from the 2-graph insplits of Definition \ref{def:2-graph-insplit}.
\end{definition}

\smallskip
\noindent
It is straightforward to check that $p_I, q_I$ are graph homomorphisms, and if $T$ is a textile system, then so is $T_I$.

\begin{remarks}[The effect of insplitting on squares] \label{rmk:upisOK} \hfill
\begin{enumerate}[leftmargin=0.7cm, label=(\roman*)]
 
\item Let $T=(p,q:F \to E)$ be a textile system. 
Suppose square $T_f$ is below square $T_g$, that is $p(f)=q(g)$.  Then this connection persists in any textile insplit $T_I = ( p_I , q_I : F_I \to E )$ of $T$, that is, $p_I ( f^j) = p(f) = q(g) = q_I (g^k)$ regardless of the values of $j,k$.

\item Let $T=(p,q:F \to E)$ be a textile system. Fix a partition $\{ \mathcal{F}_v^i : 1 \le i \le m(v) , v \in F^0 \}$ to give an insplitting $T_I$. Suppose $f,g \in F^1$ are such that $s(f)=r(g)$, so squares $T_f$ and $T_g$ are adjacent, and suppose $g \in \mathcal{F}_{r(g)}^\ell$. In order that $f^j , g^k \in F_I^1$ be  consecutive edges in $T_I$ we must have
\[
s(f)^j= s( f^j )\stackrel{!}{=} r( g^k )= r(g)^\ell ;
\]
that is $j=\ell$.
\end{enumerate}
\end{remarks}

\begin{example} [Textile system insplitting]\label{ex:JMex}
Given the directed graph $F$ shown on the left in \eqref{eq:before-after}. Let $X=\{a,b\}$ and form the textile system $T_{JM} = ( p,q:  F \to B_X )$ with squares
\begin{equation}  \label{eq:jmis1}
\begin{array}{ll}
\begin{tikzpicture}[scale=1]

\begin{scope}[xshift=0cm]

\node[inner sep=0.8pt] (00) at (0,0) {$. $};
\node[inner sep=0.8pt] (01) at (0,1) {$. $};    
\node[inner sep=0.8pt] (10) at (1,0) {$. $};
\node[inner sep=0.8pt] (11) at (1,1) {$. $};

\node[inner sep=1pt, circle] (94) at (0.5,0.5) {\tiny ${e}$};
     
\draw[->] (10)--(00) node[pos=0.5, below] {\tiny\color{black}$a$};
\draw[->] (11)--(01) node[pos=0.5, above] {\tiny\color{black}$a$};
\draw[->] (01)--(00) node[pos=0.5, left] {\tiny\color{black}$u$};
\draw[->] (11)--(10) node[pos=0.5, right] {\tiny\color{black}$u$};

\end{scope}

\begin{scope}[xshift=3cm]

\node[inner sep=0.8pt] (00) at (0,0) {$. $};
\node[inner sep=0.8pt] (01) at (0,1) {$. $};    
\node[inner sep=0.8pt] (10) at (1,0) {$. $};
\node[inner sep=0.8pt] (11) at (1,1) {$. $};

\node[inner sep=1pt, circle] (93) at (0.5,0.5) {\tiny ${f}$};
     
\draw[->] (10)--(00) node[pos=0.5, below] {\tiny\color{black}$b$};
\draw[->] (11)--(01) node[pos=0.5, above] {\tiny\color{black}$b$};
\draw[->] (01)--(00) node[pos=0.5, left] {\tiny\color{black}$v$};
\draw[->] (11)--(10) node[pos=0.5, right] {\tiny\color{black}$u$};

\end{scope}

\begin{scope}[xshift=6cm]

\node[inner sep=0.8pt] (00) at (0,0) {$ .$};
\node[inner sep=0.8pt] (01) at (0,1) {$. $};    
\node[inner sep=0.8pt] (10) at (1,0) {$. $};
\node[inner sep=0.8pt] (11) at (1,1) {$. $};

\node[inner sep=1pt, circle] (92) at (0.5,0.5) {\tiny ${g}$};
     
\draw[->] (10)--(00) node[pos=0.5, below] {\tiny\color{black}$a$};
\draw[->] (11)--(01) node[pos=0.5, above] {\tiny\color{black}$a$};
\draw[->] (01)--(00) node[pos=0.5, left] {\tiny\color{black}$v$};
\draw[->] (11)--(10) node[pos=0.5, right] {\tiny\color{black}$v$};

\end{scope}

\begin{scope}[xshift=9cm]
\node[inner sep=0.8pt] (00) at (0,0) {$. $};
\node[inner sep=0.8pt] (01) at (0,1) {$. $};    
\node[inner sep=0.8pt] (10) at (1,0) {$. $};
\node[inner sep=0.8pt] (11) at (1,1) {$. $};

\node[inner sep=1pt, circle] (91) at (0.5,0.5) {\tiny ${h}$};
    
\draw[->] (10)--(00) node[pos=0.5, below] {\tiny\color{black}$b$};
\draw[->] (11)--(01) node[pos=0.5, above] {\tiny\color{black}$b$};
\draw[->] (01)--(00) node[pos=0.5, left] {\tiny\color{black}$u$};
\draw[->] (11)--(10) node[pos=0.5, right] {\tiny\color{black}$v$};

\end{scope}

\end{tikzpicture} &
\end{array}
\end{equation}

\noindent
Set $m(u)=1,m(v)=2$, and $\mathcal{F}_v^1=\{ g \}$, $\mathcal{F}_v^2=\{f \}$. Then $F_I$ is as shown below on the right. 
\begin{equation} \label{eq:before-after}
\begin{array}{l}
\begin{tikzpicture}[yscale=0.75, >=stealth]

\node[inner sep=0.8pt] (u) at (0,0) {$\scriptstyle u$}
edge[loop,->,in=225,out=135,looseness=15] node[auto,black,swap] {$\scriptstyle e$} (u);

\node[inner sep=0.8pt] (v) at (2,0) {$\scriptstyle v$}
edge[loop,->,in=-45,out=45,looseness=15] node[auto,black,right] {$\scriptstyle g$} (v);

\draw[->] (u.north east)
    parabola[parabola height=0.5cm] (v.north west);

\node[inner sep=6pt, above] at (1,0.5) {$\scriptstyle f$};
\draw[->] (v.south west)
    parabola[parabola height=-0.5cm] (u.south east);

\node[inner sep=6pt, below] at (1,-0.5) {$\scriptstyle h$};

\node at (-2,0) {$F:=$};


\begin{scope}[xshift=7cm,yshift=-1cm]

\node at (-2.2,1.2) {$F_I:=$};

\node[inner sep=0.8pt] (u) at (0,1) {\tiny $u$} edge[loop,->,in=225,out=135,looseness=15] node[auto,black,swap] {$\scriptstyle e^1$} (u);

\node[inner sep=0.8pt] (v1) at (4,0) {\tiny $v^1$} edge[loop,->,in=-45,out=45,looseness=10] node[auto,black,right] {$\scriptstyle g^1$} (v1);
\node (v2) at (4,2) {\tiny $v^2$};

\draw[<-] (v1) -- (v2);

\node at (4.2,1) {\tiny $g^2$};

\draw[<-] (u) -- (v2);
\node at (2,0.7) {\tiny $h^1$};

\node[inner sep=6pt, above] at (2,2.1) {$\scriptstyle f^1$};
\draw[<-] (v2.north west)
    parabola[parabola height=0.5cm] (u.north east);

\draw[->] (v1) -- (u);
\node at (2,1.75) {\tiny $h^2$};

\end{scope}
\end{tikzpicture}
\end{array}
\end{equation}
\noindent 
Definition \ref{dfn:textile-insplit} tells us that the squares associated to $T_{JM_I}$ are as follows:
\begin{equation} \label{eq:jmis2}
\begin{array}{ll}
\begin{tikzpicture}[>=stealth]
\node[inner sep=0.8pt] (00) at (0,0) {$ .$};
\node[inner sep=0.8pt] (10) at (1,0) {$ .$}
	edge[->] node[auto,black] {\tiny $a$} (00);
\node[inner sep=0.8pt] (01) at (0,1) {$ .$}
	edge[->] node[auto,black,swap] {\tiny $u$} (00);
\node[inner sep=0.8pt] (11) at (1,1) {$. $}
	edge[->] node[auto,black,swap] {\tiny $a$} (01)
	edge[->] node[auto,black] {\tiny $u$} (10);
\node[inner sep=0.8pt]  at (0.5,0.5) {$\scriptstyle e^1$};

\begin{scope}[xshift=2.5cm]
\node[inner sep=0.8pt] (00) at (0,0) {$. $};
\node[inner sep=0.8pt] (10) at (1,0) {$. $}
	edge[->] node[auto,black] {\tiny $b$} (00);
\node[inner sep=0.8pt] (01) at (0,1) {$. $}
	edge[->] node[auto,black,swap] {\tiny $v^2$} (00);
\node[inner sep=0.8pt] (11) at (1,1) {$. $}
	edge[->] node[auto,black,swap] {\tiny $b$} (01)
	edge[->] node[auto,black] {\tiny $u$} (10);
\node[inner sep=0.8pt]  at (0.5,0.5) {$\scriptstyle f^1$};

\end{scope}
\begin{scope}[xshift=5.5cm]
\node[inner sep=0.8pt] (00) at (0,0) {$. $};
\node[inner sep=0.8pt] (10) at (1,0) {$ .$}
	edge[->] node[auto,black] {\tiny $a$} (00);
\node[inner sep=0.8pt] (01) at (0,1) {$. $}
	edge[->] node[auto,black,swap] {\tiny $v^1$} (00);
\node[inner sep=0.8pt] (11) at (1,1) {$. $}
	edge[->] node[auto,black,swap] {\tiny $a$} (01)
	edge[->] node[auto,black] {\tiny $v^1$} (10);
\node[inner sep=0.8pt]  at (0.5,0.5) {$\scriptstyle g^1$};

\end{scope}

\begin{scope}[xshift=8cm]
\node[inner sep=0.8pt] (00) at (0,0) {$. $};
\node[inner sep=0.8pt] (10) at (1,0) {$. $}
	edge[->] node[auto,black] {\tiny $b$} (00);
\node[inner sep=0.8pt] (01) at (0,1) {$. $}
	edge[->] node[auto,black,swap] {\tiny $u$} (00);
\node[inner sep=0.8pt] (11) at (1,1) {$. $}
	edge[->] node[auto,black,swap] {\tiny $b$} (01)
	edge[->] node[auto,black] {\tiny $v^1$} (10);
\node[inner sep=0.8pt]  at (0.5,0.5) {$\scriptstyle h^1$};

\end{scope}

\begin{scope}[xshift=10.5cm]
\node[inner sep=0.8pt] (00) at (0,0) {$. $};
\node[inner sep=0.8pt] (10) at (1,0) {$ .$}
	edge[->] node[auto,black] {\tiny $a$} (00);
\node[inner sep=0.8pt] (01) at (0,1) {$. $}
	edge[->] node[auto,black,swap] {\tiny $v^1$} (00);
\node[inner sep=0.8pt] (11) at (1,1) {$ $}
	edge[->] node[auto,black,swap] {\tiny $a$} (01)
	edge[->] node[auto,black] {\tiny $v^2$} (10);
\node[inner sep=0.8pt]  at (0.5,0.5) {$\scriptstyle g^2$};

\end{scope}

\begin{scope}[xshift=13cm]
\node[inner sep=0.8pt] (00) at (0,0) {$. $};
\node[inner sep=0.8pt] (10) at (1,0) {$. $}
	edge[->] node[auto,black] {\tiny $b$} (00);
\node[inner sep=0.8pt] (01) at (0,1) {$. $}
	edge[->] node[auto,black,swap] {\tiny $u$} (00);
\node[inner sep=0.8pt] (11) at (1,1) {$. $}
	edge[->] node[auto,black,swap] {\tiny $d$} (01)
	edge[->] node[auto,black] {\tiny $v^2$} (10);
\node[inner sep=0.8pt]  at (0.5,0.5) {$\scriptstyle h^2$};

\end{scope}

\end{tikzpicture} &
\end{array}
\end{equation}
\end{example}

\begin{remark}
Looking at the tiles in \eqref{eq:jmis1}, one can see that the original textile system is LR. However, \eqref{eq:jmis2} shows that in $T_{JM_I}$, there are two different squares (those labelled $g^1, g^2$) which both have left edge $v^1$ and top edge $a$. Hence $T_{JM_I}$ is not LR. 
Nonetheless, the 
shift spaces associated to  $T$ and $T_{JM_I}$ are conjugate by Theorem \ref{thm:insplit-conjugacy} below. \end{remark}

\noindent
The problem encountered in Example~\ref{ex:JMex} occurs in general.

\begin{theorem}[Insplitting always removes LR]
\label{thm:JM-insplit-not-LR}
Let $T = ( p,q : F \to E )$ be a textile system which is LR. Fix a partition $\{ \mathcal{F}_v^i :  1 \le i \le m(v), v \in F^0 \}$ with $m(v) \ge2$ for some $v$. Then the insplit textile system $T_I$ is not LR.
\end{theorem}

\begin{proof}
Let $v$ be a vertex with $m(v) \ge 2$. Let $1 \le i < j \le m(v)$, and fix an edge $f \in \mathcal{F}^k$ with source $v$. Consider the edges $f^i \neq f^j$ in $F_I$. By definition $r_I ( f^i) = r_I ( f^j ) = r(f)^k$. But $p_I( f^i )= p(f)= p_I( f^j )$ and 
\[
p_I ( r_I ( f^i ) ) = p_I ( r(f)^k ) = r(f).  
\]
Hence $p_I$ is not left-resolving.
\end{proof}

Recall from Definition~\ref{dfn:textiletiling} a textile system $T = ( p,q : F \to E )$ gives rise to a textile tiling $\textsf{X}_T^+$. By Remark~\ref{rmk:textilesysissft} we may consider this as a one-sided 2-dimensional shift of finite type with alphabet $\mathcal{A}=F^1$. The following result was stated in \cite[Lemma 3.4]{johnson-madden} for textile systems with $E = B_X$; for completeness, we include a proof of the general result here.

\begin{theorem}[Textile insplitting induces conjugacy] \label{thm:insplit-conjugacy}
Let $T= (p,q : F \to E )$ be a textile system. Fix a partition $\{ \mathcal{F}_v^i : 1 \le i \le m(v), v \in F^0 \}$ to form an insplitting $T_I$. Then $\textsf{X}_T^+$ is conjugate to $\textsf{X}_{T_I}^+$.
\end{theorem}

\begin{proof}
Define a $0$-block map $\Phi : B_{0}( \textsf{X}_{T_{I}}^+ ) \to B_0 ( \textsf{X}_T^+ )$ by $\Phi ( f^j)=f$
for $f \in F^1$, $1 \le j \le  m(s(f))$. Let $\phi : \textsf{X}_{T_I}^+ \to \textsf{X}_T^+$ be the induced map. 
To define the inverse map $\psi: \textsf{X}_T^+ \to \textsf{X}_{T_I}^+$, we use an $\varepsilon_1$ block map.
That is, we define $\Psi: B_{\varepsilon_1}(X_T^+) \to B_0(X_{T_I}^+)$ by 
\[ \Psi \begin{pmatrix}
   f & g
\end{pmatrix} = f^j \text{ where }g \in \mathcal F^j_{s(f)}.\]

To see  that $\psi \circ \phi = \text{id}$, let $x \in \textsf{X}^+_{T_I}$.  Then for each $m \in \NN^2, \ x_m = f^{j_m}$ for some $f\in F^1, 1 \leq j_m \leq m(s(f))$. As $s(f^{j_m}) = s(f)^{j_m}$, we must have $r(x_{m+\varepsilon_1}) = s(f)^{j_m},$ that is, 
\[ \Phi(x_{m+\varepsilon_1}) \in \mathcal F^{j_m}_{s(f)} \, \text{ for all }m.\]
Consequently, for all $m$, $ \Psi(\Phi(x_m) \Phi(x_{m +\varepsilon_1})) = \Phi(x_m)^{j_m} = x_m,$ i.e.,~$\psi \circ \phi = \text{id}.$

To see that $\phi \circ \psi = \text{id},$ let $y \in \textsf{X}_T^+.$  For each $m \in \N^2,$ we have $y_{m+\varepsilon_1} \in \mathcal F^{j_m}_{s(y_m)}$ for a unique $j_m$.  We then have  
\[ (\psi(y))_m = y_m^{j_m}\]
and so $(\phi(\psi(y)))_m = y_m;$ that is, $\phi \circ \psi = \text{id}.$  As both $\phi, \psi$ are block maps, Lemma \ref{lem:block-map-cts} and Definition \ref{def:conjugacy} imply that $\phi, \psi$ are conjugacies, as claimed. 

\end{proof}

\subsection{Insplitting for rank two graphs}\label{sec:2-graphinsplit} 

Following \cite[\S 3]{EFGGGP} a $2$-graph insplitting is defined using the $1$-skeleton 
of $\Lambda$. In order that the result is a $2$-graph, the authors of \cite{EFGGGP} impose the  {\em  pairing condition} 
on the insplitting partition.  For ease of notation, in \cite{EFGGGP} insplitting is defined at a single vertex $v$, and  the  insplitting partition consists of exactly two non-empty sets. To capture the full generality of insplitting, 
we extend the definition from \cite{EFGGGP} to the following. 

\begin{definitions}[$2$-graph insplittings, pairing property]\label{dfn:pairing}
Let $\Lambda = G^*/\sim$ be a row-finite $2$-graph with 1-skeleton $(G, d)$. 

\begin{minipage}[l]{0.7\textwidth}
An \emph{insplitting partition} of $G$ 
is a partition $\mathcal G = \{ \mathcal G_v : v \in G^0 \}$ of $G^1$, 
where $\mathcal G_v = \{ \mathcal G^1_v , \ldots , \mathcal G^{m(v)}_v\}$ is a partition of 
$v G^1$ into $m(v)$ nonempty sets  for each $v \in G^0$, which satisfies the 
\emph{pairing condition}:  whenever $a, f \in v G^1$
for some $v \in G^0$ and there exist edges $g ,b\in G^1$ 
such that $ ag  \sim fb$,
then $f \in \mathcal G^j_v$ if and only if $a \in \mathcal G^j_v$. 
\end{minipage}%
\begin{minipage}[l]{0.3\textwidth}
\begin{equation} \label{eq:pairing}
\begin{array}{l}
\begin{tikzpicture}[>=stealth,yscale=1.5,xscale=2]
\node at (0.53,0.5) {\tiny $ag \sim bf $
};
\node (00) at (0,0) {$v$};
\node (10) at (1,0) {$. $}
	edge[->,blue,thick] node[auto,black] {$a$} (00);
\node (01) at (0,1) {$. $}
	edge[->,red,dashed,thick] node[auto,black,swap] {$f$} (00);
\node (11) at (1,1) {$. $}
	edge[->,blue,thick] node[auto,black,swap] {\tiny $\exists b$} (01)
	edge[->,red,dashed,thick] node[auto,black] {\tiny $\exists g$} (10);
\end{tikzpicture}
\end{array}
\end{equation}
\end{minipage}%
\label{def:2-graph-insplit}

\noindent
An insplitting partition of a 2-graph gives rise to a new 2-graph (the {\em insplit 2-graph $\Lambda_I$ of $\Lambda$}), as follows.
Given an insplitting partition $\mathcal G$ of $G$, 
we define a $2$-colored graph $G_I = \big( (G^0_I, G^1_I, r_I, s_I),d_I\big)$, 
where
\begin{align*}
G_I^0 &= \{ v^1, v^2, \ldots, v^{m(v)}: v \in G^0\} 
\text{ and } \,
G^1_I= \bigcup_{f \in 
G^1  }    \{f^1, f^2, \ldots,  f^{m(s(f))} \}
\end{align*}
with $d_I(f^i)= d(f)$ for all $1 \le i \le m(v)$.

\noindent
The range and source maps in the directed graph $G_I$ are defined as follows:
\begin{align*}
r_I ( f^i ) & =  r ( f )^j \text{ where } f \in \mathcal G_{r(f)}^{j}  , \label{eq:rdef} \\
s_I (f^i) & = s(f)^i .   
\end{align*}

The equivalence relation in 
 $G_I$ is given by $f^i g^k \sim_I  a^j b^k$ if and only if $(g \in \mathcal G_{s(f)}^i, b \in \mathcal G_{s(a)}^j)$, and $ab \sim  fg \in \Lambda.$ (The fact that the two paths in a commuting square are required to have the same source forces $g, b$ to have the same superscript.)

\end{definitions}

\begin{remark}
    Requiring that the  partition $\mathcal G$ satisfies the pairing condition is equivalent to specifying that the edges $\lambda(0, \varepsilon_1)$ and $\lambda(0, \varepsilon_2)$ lie in the same partition set $\mathcal G^j_{r(\lambda)}$ for each $\lambda \in \Lambda^{\varepsilon_1 + \varepsilon_2}$.
\end{remark}

The fact that $\Lambda_I := G_I^*/\sim_I$ is indeed a 2-graph is established in the following Theorem.

\begin{theorem}[$2$-graph insplitting produces $2$-graph] \label{thm:inplitcomplete}
Let $\Lambda$ be an essential row-finite $2$-graph with associated $2$-colored graph $(G,d)$. Let $\mathcal G$ be  an insplitting partition of $(G,d)$. 
Then the associated quotient $G_I^*/\sim_I$ is a 2-graph $\Lambda_I$. \end{theorem}

\begin{proof}
Recall from Remark \ref{rmk:2-skel-2colograph} that we denote by $\sim_I$ the equivalence relation satisfying (KG0) and (KG1) which is generated by the set of commuting squares specified in Definitions \ref{def:2-graph-insplit}. To see that  $G_I^*/\sim_I$ is a 2-graph, 
it suffices to show that every 2-colored path $f^i g^k$ in $G_I$ is equivalent to a unique $2$-colored path $a^j b^k$.  If $f^i g^k \in G_I^2$, then $g \in \mathcal G_{s(f)}^i$; in particular $s(f) =  r(g)$, so $fg \in G^2$.  Since $\Lambda$ is a 2-graph, there is a unique path $ab \in G^2$ with $d(a) = d(g), d(b) = d(f)$ and $ab \sim fg$. In particular, $s(b) = s(g) $ and $r(b) = s(a)$.  There is a unique $j$ such that $b \in \mathcal G_{s(a)}^j$; therefore, $a^j b^k$ is the unique 2-colored path with $d(a^j) = d(g^k)$ and $a^j b^k \sim f^i g^k$.

\end{proof}

We prove in Appendix \ref{sec:$C^*$} that Definitions \ref{def:2-graph-insplit}  yield the same $C^*$-algebraic properties as the original definition of 2-graph insplitting from \cite{EFGGGP}.

\begin{remark}
    \label{rmk:squares-in-insplit}
    Thanks to the definition of the equivalence relation $\sim_I$ given at the end of Definitions \ref{def:2-graph-insplit}, every commuting square in $\Lambda_I$ arises from a (unique) commuting square in $\Lambda$.  In fact, each commuting square $\lambda = a b \sim f g$ in $\Lambda$ yields $m(s(g)) = m(s_\Lambda(\lambda))$ commuting-square ``children'' in $\Lambda_I$.
\end{remark}

\begin{example} [Failing the pairing property]\label{ex:colex}

Consider the $2$-colored graph $G$ below with commuting squares as shown.  Observe that each 2-colored path in $G$ occurs in exactly one commuting square, that is, $\sim$ satisfies (KG2) and $G^*/\sim$ is a 2-graph.
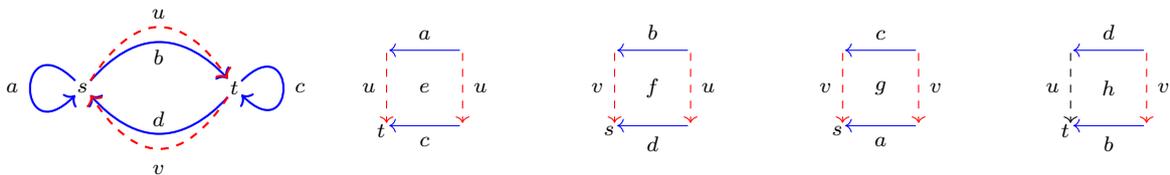
\[
\begin{tikzpicture}
\node[inner sep=0.8pt] (s) at (0,0.5) {$\scriptstyle s$}
edge[blue,loop,->,in=225,out=135,looseness=15,thick] node[auto,black,swap] {$\scriptstyle a$} (s);

\node[inner sep=0.8pt] (t) at (2,0.5) {$\scriptstyle t$}
edge[blue,loop,->,in=-45,out=45,looseness=15,thick] node[auto,black,right] {$\scriptstyle c$} (t);

\draw[blue,->,thick] (s.north east)
    parabola[parabola height=0.5cm] (t.north west);

\node[inner sep=6pt, above] at (1,0.6) {$\scriptstyle b$};
\draw[blue,->,thick] (t.south west)
    parabola[parabola height=-0.5cm] (s.south east);
\node[inner sep=6pt, below] at (1,0.4) {$\scriptstyle d$};

\node[inner sep=6pt, below] at (1,-0.3) {$\scriptstyle v$};
\draw[->,dashed,red,thick] (t.south west)
    parabola[parabola height=-0.7cm] (s.south east);
\draw[->,dashed,red,thick] (s.north east)
    parabola[parabola height=0.7cm] (t.north west);
    \node[inner sep=6pt, above] at (1,1.2) {$\scriptstyle u$};

\begin{scope}[xshift=4cm]

\node[inner sep=0.8pt] (00) at (0,0) {$ $};
\node[inner sep=0.8pt] (10) at (1,0) {$ $}
	edge[->,blue] node[auto,black] {\tiny $c$} (00);
\node[inner sep=0.8pt] (01) at (0,1) {$ $}
	edge[->,dashed,red] node[auto,black,swap] {\tiny $u$} (00);
\node[inner sep=0.8pt] (11) at (1,1) {$ $}
	edge[->,blue] node[auto,black,swap] {\tiny $a$} (01)
	edge[->,dashed,red] node[auto,black] {\tiny $u$} (10);
\node[inner sep=0.8pt]  at (0.5,0.5) {$\scriptstyle e$};
\node[inner sep=0.8pt]  at (-0.07,-0.07) {$\scriptstyle t$};

\end{scope}

\begin{scope}[xshift=7cm]
\node[inner sep=0.8pt] (00) at (0,0) {$ $};
\node[inner sep=0.8pt] (10) at (1,0) {$ $}
	edge[->,blue] node[auto,black] {\tiny $d$} (00);
\node[inner sep=0.8pt] (01) at (0,1) {$ $}
	edge[->,dashed,red] node[auto,black,swap] {\tiny $v$} (00);
\node[inner sep=0.8pt] (11) at (1,1) {$ $}
	edge[->,blue] node[auto,black,swap] {\tiny $b$} (01)
	edge[->,dashed,red] node[auto,black] {\tiny $u$} (10);
\node[inner sep=0.8pt]  at (0.5,0.5) {$\scriptstyle f$};
\node[inner sep=0.8pt]  at (-0.07,-0.07) {$\scriptstyle s$};

\end{scope}
\begin{scope}[xshift=10cm]
\node[inner sep=0.8pt] (00) at (0,0) {$ $};
\node[inner sep=0.8pt] (10) at (1,0) {$ $}
	edge[->,blue] node[auto,black] {\tiny $a$} (00);
\node[inner sep=0.8pt] (01) at (0,1) {$ $}
	edge[->,dashed,red] node[auto,black,swap] {\tiny $v$} (00);
\node[inner sep=0.8pt] (11) at (1,1) {$ $}
	edge[->,blue] node[auto,black,swap] {\tiny $c$} (01)
	edge[->,dashed,red] node[auto,black] {\tiny $v$} (10);
\node[inner sep=0.8pt]  at (0.5,0.5) {$\scriptstyle g$};
\node[inner sep=0.8pt]  at (-0.07,-0.07) {$\scriptstyle s$};

\end{scope}

\begin{scope}[xshift=13cm]
\node[inner sep=0.8pt] (00) at (0,0) {$ $};
\node[inner sep=0.8pt] (10) at (1,0) {$ $}
	edge[->,blue] node[auto,black] {\tiny $b$} (00);
\node[inner sep=0.8pt] (01) at (0,1) {$ $}
	edge[->,dashed] node[auto,black,swap] {\tiny $u$} (00);
\node[inner sep=0.8pt] (11) at (1,1) {$ $}
	edge[->,blue] node[auto,black,swap] {\tiny $d$} (01)
	edge[->,dashed,red] node[auto,black] {\tiny $v$} (10);
\node[inner sep=0.8pt]  at (0.5,0.5) {$\scriptstyle h$};
\node[inner sep=0.8pt]  at (-0.07,-0.07) {$\scriptstyle t$};

\end{scope}

\end{tikzpicture}
\]

\noindent

Since $av\sim vc$ and $du \sim vb$, the pairing condition implies that any insplitting partition $\mathcal G$ will have only one set in the partition $\mathcal G_s$. Similarly, since $ua \sim cu$ and  $ud \sim bv$, there will be only one set in $\mathcal G_t$.  That is, there is no nontrivial 2-graph insplitting of this 2-graph.
\end{example}

\section{2-graph insplitting in terms of Johnson--Madden moves}\label{sec:ghinsp-JM}

If we wish to insplit an LR textile system $T$, we have two options.  Of course, we can perform a textile-system insplitting, which (thanks to Theorem \ref{thm:JM-insplit-not-LR})  will yield a non-LR textile system.  However, we can  alternatively form the associated 2-graph $\Lambda = \Lambda_T$ as in Theorem \ref{thm:PST};  perform a 2-graph insplitting on $\Lambda_T$ to yield a 2-graph $\Lambda_I$; and construct the associated textile system $T_I$ as in \eqref{eq:textile-from-2graph}. Theorem \ref{thm:PST} tells us that $T_I$ will again be LR, that is, it cannot arise from $T$ via a textile insplitting. Indeed, this discrepancy was our inspiration for the research contained here.

Despite this inconsistency, we show in Theorem \ref{thm:6.1} below that we can reconstruct 2-graph insplitting from a sequence of conjugacy-preserving moves (textile-system insplitting and inversion) on the associated textile system.  

\begin{theorem}\label{thm:6.1}
Let $T=(p, q: F\to E)$ be a LR textile system and $\Lambda_T$ be the associated $2$-graph. Let $\mathcal G  = \{ \mathcal G^i_z: 1\leq i \leq m(z) , z \in \Lambda^0 \}$ be an insplitting partition 
of $\Lambda_T$, and let $\Lambda_I$ and $T_I= ( p_I, q_I : F_I \to E_I )$ be the resulting insplit 2-graph and associated textile system. 
Then by performing a Johnson-Madden insplit on $T$; inverting this textile system; performing a second Johnson-Madden insplit; and inverting again we construct a textile system $T_D$ such that $T_I$ and $T_D$ give rise to identical tilings of $\mathbb R_{\geq 0}^2$. 
In particular, $2$-graph insplitting gives a conjugacy of one-sided dynamical systems.
\end{theorem}

\begin{proof}

We will first perform four conjugacy-preserving textile-system moves on the textile system $T$ and then compare the resulting shift space with the one associated to the textile system $T_I$ associated to the $2$-graph insplitting on $\Lambda_T$. 

 \noindent\textbf{Step 1}   
Recall $\Lambda_T^0 = E^0$. By hypothesis  we can insplit the associated 2-graph $\Lambda_T$  using the partition $\{\mathcal{G}^i_z: 1\leq i\leq m(z) , \ z \in E^0 \}$ of $
\Lambda_T^1= E^1 \sqcup F^0$. We use this partition to create a partition of $F^1 = \Lambda_T^{\varepsilon_1 + \varepsilon_2}$ as follows: Fix $v\in F^0$, and write $z = p(v)$.  For each $1\leq i\leq m(z)$, we define 
\begin{equation}
\label{eq:first-JM-insplit} 
\mathcal F^i_v= \{ \lambda : \lambda \in F^1 ,  \ r ( \lambda ) = v  , \ p(\lambda) \in \mathcal G_z^i  \} .
\end{equation}
\noindent
The collection $\{\mathcal{F}_v^i: v\in F^0, 1\leq i\leq m(p(v))\}$ gives a partition of $F^1$. 
Use this partition to perform an insplitting of $T$ as in Definition~\ref{dfn:textile-insplit}, resulting in the textile system $T_A:= (p_A , q_A : F_A \to E_A= E )$, where
  \begin{multline*}
F_A^0 = \{ v^i : v \in F^0 , 1 \le i \le m(p(v)) \}  
\text{ and } F_A^1 = \{ {\lambda^i}
: \lambda \in F^1,\ 
1 \le i \le {m(p(s(\la)))} \} .
\end{multline*}

\noindent
{We have $p_A ( \lambda^i ) = p( \lambda )$ and $q_A ( \lambda^i ) = q ( \lambda )$, with $s_A ( \lambda^i ) = s (\lambda)^i$ and $r_A ( \lambda^i )=r(\la)^k= v^k$ where $\lambda \in \mathcal{F}_{v}^k$.}

 \noindent\textbf{Step 2} Recall the definition of inverted textile system from Definition~\ref{dfn:inverted-textile}. The inverted textile system $T_B:= (p_B, q_B: F_B \to E_B )$ from $T_A$ has $F_B^0= E_A^1= E^1$, $E_B^0= E_A^0= E^0$,
\[ 
F_B^1= \{ {\lambda^i}
: \lambda \in F^1, \ 
1 \le i \le m(p(s(\la)))\}, \ \text{ and } \ E_B^1= F_A^0 = \{ v^i : 1 \le i \le m(p(v)) \}, 
\]

\noindent
 with $r_{F_B}( {\lambda^i}
 )= q_A({\lambda^i} 
 ) =q(\lambda)$, $s_{F_B}({\lambda^i} 
 )= p_A({\lambda^i} 
 )=p(\lambda)$, $s_{E_B}(u^i) = p_A(u^i)=p(u)$, and $r_{E_B}(u^i)= q_A(u^i)= q(u)$. Similarly, $p_B ({\lambda^i} 
 )= s_A({\lambda^i} 
 ) = s (\lambda)^i, \ q_B ({\lambda^i} 
 ) = r_A ( {\lambda^i} 
 ) = r (\lambda)^\ell$ if $\lambda \in \mathcal{F}^\ell_{v}$.   We also have  $p_B(e) = s(e)$ and  $q_B(e) = r(e).$

 \noindent\textbf{Step 3} Recall the partition  $\{\mathcal{G}^i_z: z \in E^0 ,  \ 1\leq i\leq m(z)\}$ of $\Lambda_T^{1} = E^1 \sqcup F^0$ from Step 1. In a similar way to Step 1, we use this partition to create a partition of $F_B^1$ 
as follows: Fix $z \in E^0$ and  $e \in F_B^0 = E^1 $ with $p_B(e)= z$ and $1\le j \le m(z)$. Parallel to Step 1, we would like to define $\mathcal{H}^j_e = \{\lambda^i \in F_B^1 : \ e =  r_{F_B} ( \lambda^i ), \ p_B (  \lambda^i) \in \mathcal G^j_z\}$; however, $p_B(\lambda^i) = s(\lambda)^i$ is not in $\mathcal G^j_z \subseteq E^1 \sqcup F^0.$  Hence we define 
\begin{align*} \label{eq:gmedef}
\mathcal{H}^j_e &= 
\{\lambda^i \in F_B^1 : \ e =  r_{F_B} ( \lambda^i ), \ s(  \lambda)
\in \mathcal G_{p_B(e)}^j 
\} \\
& =\{ \lambda^i : \lambda \in F^1 , \ e = q( \lambda ) , \ s ( \lambda )  \in  \mathcal G_{s(e)}^j,  
 { 1 \leq i \leq m(p(s(\lambda)))}\}.
\end{align*} 

 We observe that $\mathcal H_e^j$ is always nonempty, since $T$ is LR and $\mathcal G$ is a 2-graph insplitting partition.  In particular, in each $\mathcal G^j_z$, we have at least one $w \in F^0$ and one $\wt e \in E^1.$  Thus, given $e \in E^1$ with $p_B(e) = s(e) = z$ and $1 \leq j\leq m(z),$ choose $w \in \mathcal G^j_z \cap F^0.$  Then $q(w) = z = s(e)$ (cf.~Definition \ref{dfn:2colgraph from T}).  Consequently, the unique $s$-path lifting of $q$ implies that there is a unique $\lambda \in F^1$ with $q(\lambda) = e$ and $s(\lambda) = w,$ and we have $\lambda^i \in \mathcal H^j_e$ for all $1\leq i\leq m(s(p(\lambda))).$ 

For every $\lambda \in F^1$, $s(\lambda)
$ lies in a unique $\mathcal G^j_{s(q(\lambda))}$, so 
$\{\mathcal{H}^j_e: e\in F_B^0, 
1 \leq j \leq m(p_B(e))\}$ forms a partition of $F_B^1$. We thus perform an insplitting using this partition and get the textile system $T_C:= (p_C, q_C : F_C\to E_C$). Then $F_C^0= \{ e^j: e\in F_B^0= E^1, 1 \le j \le m(s(e)) \}$, and 
\begin{align*}
F_C^1 = \{ ( \lambda^i )^j : \lambda \in F^1 , \ e = q_A ( \lambda^i ) , \  & s(\lambda) \in  \mathcal G_{s(e)}^k , \\
& 1 \leq i  \leq m(s(p(\lambda))),  
 1 \leq j \leq m(p_B(s_B(\lambda^i))) = m(s(p(\lambda))) \},
\end{align*}
and $E_C= E_B$ given in Step 2.
We have $s_{C} (( \lambda^i )^j ) = s_{F_{B}}(\lambda^i)^j$, $r_C((\lambda^i )^j )= r_{F_{B}}( \lambda^i )^{\ell}$, where $\lambda^i \in \mathcal{G}^\ell_{e}$, and $p_{C}((\lambda^i )^j )= p_B( \lambda^i)$, $q_C(( \lambda^i )^j)= q_B(\lambda^i) $. 

 \noindent\textbf{Step 4} Now we invert the textile system $T_C$ to get $T_D:= (p_D, q_D: F_D \to E_D)$ where $F_D^0 = E_C^1 = E_B^1= F_A^0$,
\[
F_D^1 = F_C^1  = \{(\lambda^i)^j : \lambda^i \in F_A^1  \ , 1 \leq i , j \leq m(p(s(\lambda))) 
\},  
\]
 and $E_D^1 = F_C^0$. In the graph $F_D$, we have $r_{F_D}((\lambda^i )^j) = q_C((\lambda^i)^j) = r(\lambda)^\ell$ if $p(\lambda) \in \G^\ell_{p(r(\lambda))},$ $s_{F_D}(( \lambda^i)^j) = p_C ((\lambda^i)^j) = s(\lambda)^i$; and on $E_D$, we have $r_{E_D} (e^j) = q_C (e^j) = q_B(e) = r(e)$, $s_{E_D}(e^j) = p_C (e^j) = p_B(e) = s(e)$. Finally, we have the graph homomorphisms $p_D((\lambda^i)^j )= s_C(( \lambda^i)^j ) = (s_B(\lambda^i))^j = p(\lambda)^j$ and $q_D((\lambda^i)^j)= r_C((\lambda^i)^j) = (r_B(\lambda^i))^k= q(\lambda)^k$ if $s(\lambda) \in \G^k_{s(q(\lambda))}$.

In order to compare the resulting textile system $T_D$ with the 2-graph $\Lambda_I$,
recall that in a textile system, the quadruple $(p, q, r, s)$ is injective.  Thus, we can visualize the textile system $T_D$ by describing the boundary of each square $(\lambda^i)^j \in F_D^1$ (for $1 \leq i, j \leq m(p(s(\lambda))))$.    Unpacking the definitions above, we see that 
 
\begin{align*}
p_D ((\lambda^i)^j ) &= s_C (( \lambda^i)^j )= s_{F_B} ( \lambda^i)^j =p_A (\lambda^i)^j = p (\lambda)^j, \\
q_D ((\lambda^i)^j ) &= r_C ( (\lambda^i)^j ) =  r_{F_B} (\lambda^i)^k = 
q ( \lambda )^k, \text{ where $\lambda^i\in \mathcal{H}^k_{q(\lambda)}$, i.e., }s_F(\lambda) \in \mathcal G^k_{s(q(\lambda))}, \\
r_D(( \lambda^i)^j ) &= q_C ( (\lambda^i)^j ) = q_B ( \lambda^i ) = 
r_A (\lambda^i ) = r_F (\lambda )^\ell \text{ where $\lambda\in \mathcal{F}^\ell_{r(\lambda)}$, i.e., }p(\lambda) \in \mathcal G^\ell_{p(r(\lambda))}, \\
s_D ((\lambda^i)^j ) &= p_C ( (\lambda^i)^j ) = p_{B} ( \lambda^i ) = s_A (\lambda^i ) = s_F(\lambda )^i.
\end{align*}
Pictorially, we have the square 
\[ 
\begin{tikzpicture}[>=stealth][scale=1.2]
\node[inner sep=0.8pt] (00) at (0,0) {$ .$};
\node[inner sep=0.8pt] (10) at (1,0) {$ .$}
	edge[->] node[auto,black] {\tiny $q(\lambda)^k $} (00);
\node[inner sep=0.8pt] (01) at (0,1) {$. $}
	edge[->] node[auto,black,swap] {\tiny $ r_F(\lambda)^\ell$} (00);
\node[inner sep=0.8pt] (11) at (1,1) {$. $}
	edge[->] node[auto,black,swap] {\tiny $p(\lambda)^j$} (01)
	edge[->] node[auto,black] {\tiny $s_F(\lambda)^i$} (10);
\node[inner sep=0.8pt]  at (0.5,0.5) {\tiny $ (\lambda^i)^j$};
\end{tikzpicture}
\]

Now, consider the result of performing 2-graph insplitting on $\Lambda_T$ with respect to the partition $\{ \mathcal G^i_z: z \in \Lambda^0 = E^0, 1\leq i \leq m(z)\}$.

Given $\lambda \in F^1 = \Lambda^{\varepsilon_1 + \varepsilon_2}$, write  $\lambda = q(\lambda) s_F(\lambda) \sim r_F(\lambda) p(\lambda)$ with $s_F(\lambda) \in \mathcal G_{s(q(\lambda))}^k$ and $p(\lambda) \in \mathcal G_{r(p(\lambda))}^\ell$.  Recall from Remark \ref{rmk:squares-in-insplit} that every commuting square in $\Lambda_I$ arises from a commuting square in $\Lambda$, and each $\lambda \in \Lambda^{\varepsilon_1 + \varepsilon_2}$ yields $m(s(p(\lambda)))$ ``children'' in $\Lambda_I^{\varepsilon_1 + \varepsilon_2}$. Indeed, the children of the aforementioned $\lambda$ are  $\{ r_F(\lambda)^\ell p(\lambda)^i \sim_I q(\lambda)^k s_F(\lambda)^i \}_{i=1}^{m(s(p(\lambda))}$.  That is, in $\Lambda_I$
we have the commuting squares
\[ 
\begin{tikzpicture}[>=stealth][scale=1.2]
\node[inner sep=0.8pt] (00) at (0,0) {$. $};
\node[inner sep=0.8pt] (10) at (1,0) {$ .$}
	edge[->] node[auto,black] {\tiny $q(\lambda)^k $} (00);
\node[inner sep=0.8pt] (01) at (0,1) {$ .$}
	edge[->] node[auto,black,swap] {\tiny $ r_F(\lambda)^\ell$} (00);
\node[inner sep=0.8pt] (11) at (1,1) {$. $}
	edge[->] node[auto,black,swap] {\tiny $p(\lambda)^i $} (01)
	edge[->] node[auto,black] {\tiny $s_F(\lambda)^i$} (10);
\node[inner sep=0.8pt]  at (0.5,0.5) {$ \lambda^i$};
\end{tikzpicture}
\]
for $1 \leq i \leq m(s(p(\lambda)))$.

While all of the squares from the 2-graph insplitting  appear in $T_D$ constructed earlier, we also have many ``extra'' squares in $T_D$, namely $\{ (\lambda^i)^{j}: i\not= j\}$. However, we claim that none of these extra squares will show up in any $2\times 2$ block; in particular, they will not appear in the infinite path space $X_D$. To establish the claim, suppose that we have a $2\times 2$ block in ${\sf X}_D^+$ of the form 
\[
\begin{tikzpicture}[scale=1.5,->=stealth]
        
\node[inner sep=1pt, circle] (10) at (1,0) {$.$};
\node[inner sep=1pt, circle] (11) at (1,1) {\tiny $. $};

\node[inner sep=1pt, circle] (22) at (2,2) {$.$};
\node[inner sep=1pt, circle] (21) at (2,1) {\tiny $.$};
        
\node[inner sep=1pt, circle] (30) at (3,0) {$.$};
\node[inner sep=1pt, circle] (31) at (3,1) {$.$};
    
\node[inner sep=1pt, circle] (12) at (1,2) {$.$};
\node[inner sep=1pt, circle] (20) at (2,0) {$.$};
\node[inner sep=1pt, circle] (32) at (3,2) {$.$};
    
    
\draw (11)--(10) node[pos=0.5, left] {\tiny\color{black}$ $};
    
    
\draw (20)--(10) node[pos=0.5, below] {\tiny\color{black}$ $};
\draw (21)--(11) node[pos=0.5, below] {\tiny\color{black}$ $};
\draw (21)--(20) node[pos=0.5, left] {\tiny\color{black}$ $};
    
    
\draw (30)--(20) node[pos=0.5, below] {\tiny\color{black}$ $};
\draw (31)--(21) node[pos=0.5, below]{\tiny\color{black}$ $};
\draw (31)--(30) node[pos=0.5, left] {\tiny\color{black}$ $};
    

\node at (1.5,1.2) {\tiny $ $};
\node at (1.5, 0.4) {\tiny $\lambda^{i^j}$};
\node at (2.5, 0.5) {\tiny $\mu^{k^n}$};

\node at (2.5,1.2) {\tiny $ $};
\node at (1.5,1.6) {\tiny $\eta^{a^b}$};

\draw (12)--(11) node[pos=0.5, right] {\tiny\color{black}$ $};
    
    
\draw (22)--(12) node[pos=0.5, below] {\tiny\color{black}$ $};

\draw (22)--(21) node[pos=0.5, left] {\tiny\color{black}$  $};

\draw (32)--(22) node[pos=0.5, below] {\tiny\color{black}$ $};
  \draw (32)--(31) node[pos=0.5, left] {\tiny\color{black}$ $};
\end{tikzpicture}
\]
where $i \neq j$. Then the superscript on the edge $q(\eta)= p(\lambda)$ in the diagram above must be $j$, that is, {$\eta^a \in \mathcal H^j_{q(\eta)}$}, and so $s(\eta)\in \mathcal G_z^j$ for some $z\in E^0$. Similarly, the superscript  on the edge $r(\mu)= s(\lambda)$ in the diagram above must be $i$, that is, $\mu \in \mathcal F_{r(\mu)}^i$, or in other words, $p(\mu) \in \mathcal G_{z'}^i$ for some $z'\in E^0$. Since the sets $\mathcal G_z^i$ satisfy the pairing condition, if $i \neq j$ there is no square in $\Lambda^{\varepsilon_1 + \varepsilon_2} = F^1$ whose left edge lies in $\mathcal G^j$ and whose bottom edge lies in $\mathcal G^i$.  Consequently, there is no edge in $F_D^1$ that will fill in the unlabeled top-right square of this $2\times 2$ block. In other words, no tiling in ${\sf X}_D^+$ will contain any of the squares $\lambda^{i^j}$ with $i\neq j$; so the map $\lambda^{i^i} \mapsto \lambda^i$ gives the desired relabeling which identifies ${\sf X}_D^+$ with ${\sf X}_I^+$. 

We can also view this relabeling in the setting of textile systems, as follows.  Define $F_D' = (F_D^0, \{ \lambda^{i^i}: \lambda \in F^1, 1 \leq i \leq m(p(s(\lambda)))\}, r_D, s_D).$ Then the restrictions of $p_D, q_D$ to $F_D'$ are still graph homomorphisms, so 
$T':= (p_D, q_D : F'_D \to  E_D)$ is a textile system.  Moreover, as we observed above, ${\sf X}_D^+ = {\sf X}_{T'}^+ = {\sf X}_{T_I}^+$.

\end{proof}

\begin{remark*}
Theorem \ref{thm:6.1} can be summarized as the following commuting diagram. Note that we do not claim that the textile systems $T_D$ and $T_I$ are identical but the corresponding infinite path spaces are, in the sense that the diagram commutes at the level of the infinite path spaces.

\begin{center}
\begin{tikzcd}
T \arrow[r, "JM"] \arrow[d, leftrightsquigarrow] & T_A \arrow[r, "inv"] & T_B \arrow[r, "JM"] & T_C \arrow[r, "inv"] & T_D \arrow[d, leftrightsquigarrow] \\
\Lambda \arrow[rrrr, "insplit"] & & & & \Lambda_I 
\end{tikzcd}\end{center}
\end{remark*}

\begin{example}\label{eq:the example}
Consider the LR textile system $T= (p, q: F \to E)$, where the graphs $F, E$ are depicted below. 
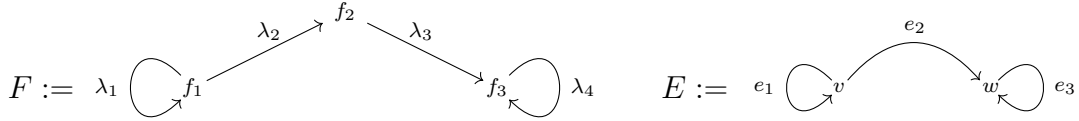
\[
\begin{tikzpicture}
    
\begin{scope}[xshift=-4cm,yshift=-1cm]
\node at (-2,1) {$F:=$};

\node[inner sep=0.8pt] (u) at (0,1) {\tiny $f_1$} edge[loop,->,in=225,out=135,looseness=10] node[auto,black,swap] {$\scriptstyle \lambda_1$} (u);

\node[inner sep=0.8pt] (v1) at (4,1) {\tiny $f_3$} edge[loop,->,in=-45,out=45,looseness=10] node[auto,black,right] {$\scriptstyle \lambda_4$} (v1);
\node (v2) at (2,2) {\tiny $f_2$};

\draw[<-] (v1) -- (v2);

\node at (3,1.75) {\tiny $\lambda_3$};

\draw[->] (u) -- (v2);


\node at (1,1.75) {\tiny $\lambda_2$};

\end{scope}


    
\begin{scope}[xshift=4.5cm,yshift=0cm]

\node at (-1.9,0) {$E:=$};

\node[inner sep=0.8pt] (v) at (0,0) {$\scriptstyle v$}
edge[loop,->,in=225,out=135,looseness=15] node[auto,black,swap] {$\scriptstyle e_1$} (v);

\node[inner sep=0.8pt] (w) at (2,0) {$\scriptstyle w$} 
edge[loop,->,in=-45,out=45,looseness=15] node[auto,black,right] {$\scriptstyle e_3$} (w);

\draw[->] (v.north east)
    parabola[parabola height=0.5cm] (w.north west);

\node[inner sep=6pt, above] at (1,0.5) {$\scriptstyle e_2$};

\end{scope}
\end{tikzpicture}
\]
The graph homomorphisms $p$ and $q$ are given by: $p(\lambda_1)= p(\lambda_2)= e_1$, $p(\lambda_3)= e_2$, $p(\lambda_4)= e_3$, $p(f_1)= p(f_2)= v$ and $p(f_3)= w$; $q(\lambda_1)= e_1$, $q(\lambda_2)= e_2$, $q(\lambda_3)= q(\lambda_4)= e_3$, and $q(f_1)= v$, $q(f_2)= q(f_3)= w$.
The associated 2-graph $\Lambda = \Lambda_T$ 
was described in \cite[Example 6.1]{EFGGGP} and 
has the following 1-skeleton:
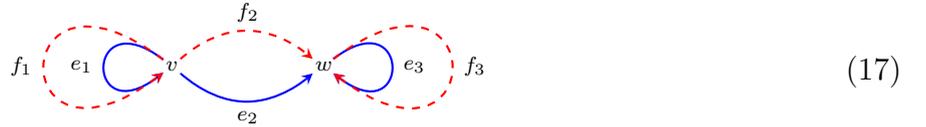
\begin{equation} \label{ex:6.1 EF+}
\begin{array}{l}
\begin{tikzpicture}[yscale=0.75, >=stealth]

\node[inner sep=0.8pt] (u) at (0,0) {$\scriptstyle v$}
edge[blue,loop,->,in=225,out=135,looseness=15,thick] node[auto,black,left] {$\scriptstyle e_1$} (u)
edge[blue,loop,->,in=225,out=135,looseness=30,dashed,red,thick] node[auto,black,left] {$\scriptstyle f_1$} (u);

\node[inner sep=0.8pt] (v) at (2,0) {$\scriptstyle w$}
edge[blue,loop,->,in=-45,out=45,looseness=15,thick] node[auto,black,right] {$\scriptstyle e_3$} (v)
edge[loop,->,in=-45,out=45,looseness=30,dashed,red,thick] node[auto,black,right] {$\scriptstyle f_3$} (v);

\draw[->,dashed,red,thick] (u.north east)
    parabola[parabola height=0.5cm] (v.north west);
\node[inner sep=6pt, above] at (1,0.5) {$\scriptstyle f_2$};

\draw[blue,<-,thick] (v.south west)
	parabola[parabola height=-0.5cm] (u.south east);
\node[inner sep=6pt, below] at (1,-0.5) {$\scriptstyle e_2$};

\end{tikzpicture}
\end{array}
\end{equation}
Observe that if we define \begin{equation}
\label{eq:2graph-insplit}
    \mathcal{G}_v^1= \{f_1, e_1\}, \, 
    \mathcal{G}_w^1= \{f_2, e_2\}, \, \mathcal{G}_w^2= \{f_3, e_3\},
\end{equation}
then the partitions $\{\mathcal{G}_v^1 \}, \{\mathcal{G}_w^1, \mathcal{G}_w^2\}$ of $v\Lambda^1, w\Lambda^1$ satisfy the pairing condition.  Insplitting $\Lambda_T$ using this partition yields 
 the 2-graph $\Lambda_I$ 
with 1-skeleton 
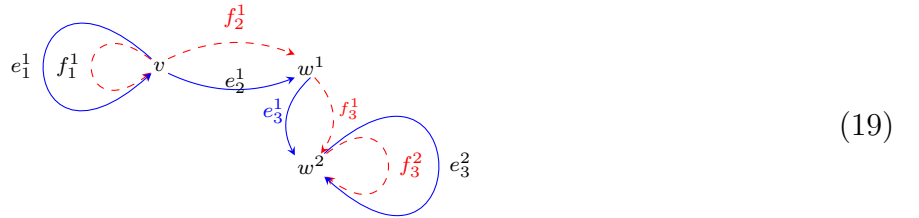
\begin{equation} \label{ex:2-graph LambdaI}
\begin{array}{l}
\begin{tikzpicture}[yscale=0.75, >=stealth]

\node[inner sep=0.8pt] (v1) at (0,0) {$\scriptstyle v$}
edge[loop,->, dashed,red, in=225,out=135,looseness=15] node[auto,black,swap] {$\scriptstyle f_1^1$} (v1) edge[loop,->, blue, in=230,out=130,looseness=30] node[auto,black,swap] {$\scriptstyle e_1^1$} (v1);

\node[inner sep=0.8pt] (w1) at (2,0) {$\scriptstyle w^1$};


\node[inner sep=0.8pt] (w2) at (2,-1.75) {$\scriptstyle w^2$}
edge[loop,->,dashed,red,in=-45,out=45,looseness=10] node[auto,black,right] {\color{red}{$\scriptstyle f_3^2$}} (w2)
edge[loop,->,blue,in=-50,out=50,looseness=20] node[auto,black,right] {$\scriptstyle e_3^2$} (w2);;

\draw[->, dashed, red] (v1.north east) to [bend left] (w1.north west);

\node[inner sep=3pt, above] at (1,0.5) {{\color{red}$\scriptstyle f_2^1$}};

\node at (2.5,-0.75) {{\color{red} \tiny $\scriptstyle f_3^1$}};

\draw[<-, blue] (w1.south west) to [bend left] (v1.south east);
\node[inner sep=3pt, below] at (1,0.1) {$\scriptstyle e_2^1$};


\draw[->, blue] (w1.265) to [bend right](w2.150);

\draw[->, dashed, red] (w1.285) to [bend left](w2);

\node at (1.5,-0.8) {\tiny {\color{blue} \tiny $e_3^1$}};

\end{tikzpicture}
\end{array}
\end{equation}
with the commuting squares $\lambda_j^i$ for $i\leq 2$ and $j= 1, 2, 3, 4$:

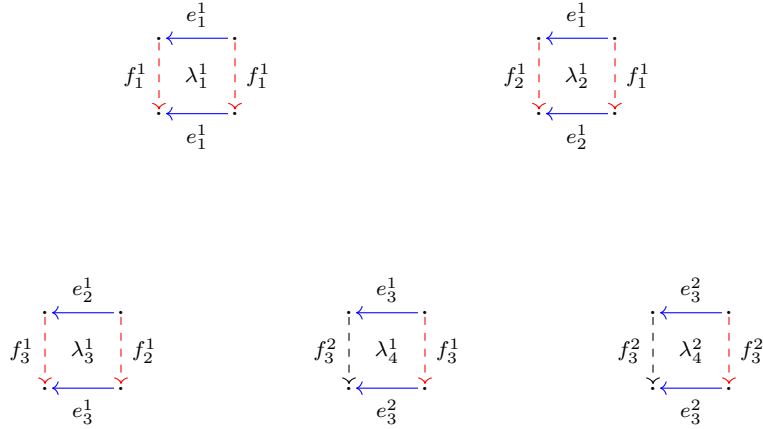
\[
\begin{tikzpicture}
\begin{scope}[xshift=2cm]
\node[inner sep=0.8pt] (00) at (0,0) {$. $};
\node[inner sep=0.8pt] (10) at (1,0) {$. $}
	edge[->,blue] node[auto,black] {\tiny $e_1^1$} (00);
\node[inner sep=0.8pt] (01) at (0,1) {$. $}
	edge[->,dashed,red] node[auto,black,swap] {\tiny $f_1^1$} (00);
\node[inner sep=0.8pt] (11) at (1,1) {$. $}
	edge[->,blue] node[auto,black,swap] {\tiny $e_1^1$} (01)
	edge[->,dashed,red] node[auto,black] {\tiny $f_1^1$} (10);
\node[inner sep=0.8pt]  at (0.5,0.5) {$\scriptstyle \lambda_1^1$};
\end{scope}

\begin{scope}[xshift=7cm]
\node[inner sep=0.8pt] (00) at (0,0) {$. $};
\node[inner sep=0.8pt] (10) at (1,0) {$. $}
	edge[->,blue] node[auto,black] {\tiny $e_2^1$} (00);
\node[inner sep=0.8pt] (01) at (0,1) {$. $}
	edge[->,dashed,red] node[auto,black,swap] {\tiny $f_2^1$} (00);
\node[inner sep=0.8pt] (11) at (1,1) {$. $}
	edge[->,blue] node[auto,black,swap] {\tiny $e_1^1$} (01)
	edge[->,dashed,red] node[auto,black] {\tiny $f_1^1$} (10);
\node[inner sep=0.8pt]  at (0.5,0.5) {$\scriptstyle \lambda_2^1$};
\end{scope}

\end{tikzpicture}
\]

\[
\begin{tikzpicture}
\begin{scope}[xshift=3cm]
\node[inner sep=0.8pt] (00) at (0,0) {$. $};
\node[inner sep=0.8pt] (10) at (1,0) {$ .$}
	edge[->,blue] node[auto,black] {\tiny $e_3^1$} (00);
\node[inner sep=0.8pt] (01) at (0,1) {$. $}
	edge[->,dashed,red] node[auto,black,swap] {\tiny $f_3^1$} (00);
\node[inner sep=0.8pt] (11) at (1,1) {$. $}
	edge[->,blue] node[auto,black,swap] {\tiny $e_2^1$} (01)
	edge[->,dashed,red] node[auto,black] {\tiny $f_2^1$} (10);
\node[inner sep=0.8pt]  at (0.5,0.5) {$\scriptstyle \lambda_3^1$};
\end{scope}

\begin{scope}[xshift=7cm]
\node[inner sep=0.8pt] (00) at (0,0) {$. $};
\node[inner sep=0.8pt] (10) at (1,0) {$. $}
	edge[->,blue] node[auto,black] {\tiny $e_3^2$} (00);
\node[inner sep=0.8pt] (01) at (0,1) {$. $}
	edge[->,dashed] node[auto,black,swap] {\tiny $f_3^2$} (00);
\node[inner sep=0.8pt] (11) at (1,1) {$. $}
	edge[->,blue] node[auto,black,swap] {\tiny $e_3^1$} (01)
	edge[->,dashed,red] node[auto,black] {\tiny $f_3^1$} (10);
\node[inner sep=0.8pt]  at (0.5,0.5) {$\scriptstyle \lambda_4^1$};
\end{scope}

\begin{scope}[xshift=11cm]
\node[inner sep=0.8pt] (00) at (0,0) {$. $};
\node[inner sep=0.8pt] (10) at (1,0) {$. $}
	edge[->,blue] node[auto,black] {\tiny $e_3^2$} (00);
\node[inner sep=0.8pt] (01) at (0,1) {$. $}
	edge[->,dashed] node[auto,black,swap] {\tiny $f_3^2$} (00);
\node[inner sep=0.8pt] (11) at (1,1) {$. $}
	edge[->,blue] node[auto,black,swap] {\tiny $e_3^2$} (01)
	edge[->,dashed,red] node[auto,black] {\tiny $f_3^2$} (10);
\node[inner sep=0.8pt]  at (0.5,0.5) {$\scriptstyle \lambda_4^2$};
\end{scope}
\end{tikzpicture}
\]

\noindent
We may deduce the formulas for $p_I,q_I$ from the above squares. The graphs $F_{\Lambda_I}$ and $E_{\Lambda_I}$ are as follows:
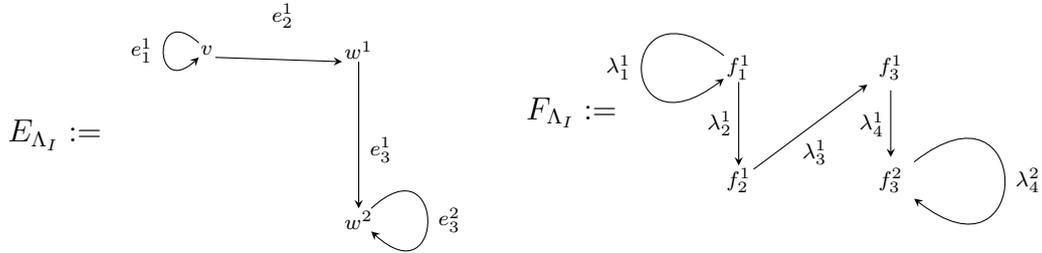
\[
\begin{tikzpicture}[yscale=0.75, >=stealth]
\begin{scope}[yshift=2.3cm]
\node at (-2,-1.5) {$E_{\Lambda_I}:=$};
\node[inner sep=0.8pt] (v1) at (0,0) {$\scriptstyle v$}
edge[loop,->, black, in=230,out=130,looseness=10] node[auto,black,swap] {$\scriptstyle e_1^1$} (v1);

\node[inner sep=0.8pt] (w1) at (2,0) {$\scriptstyle w^1$};


\node[inner sep=0.8pt] (w2) at (2,-3) {$\scriptstyle w^2$}
edge[loop,->,black,in=-50,out=50,looseness=10] node[auto,black,right] {$\scriptstyle e_3^2$} (w2);;

\draw[<-] (w1.south west)--(v1.south east);
\node[inner sep=6pt, above] at (1,0.1) {$\scriptstyle e_2^1$};



\draw[->] (w1)--(w2) node[pos=0.6,right] {\tiny $e_3^1$};

\end{scope}

\begin{scope}[xshift=7cm,yshift=-1cm]
\node at (-2.2,2.2) {$F_{\Lambda_I}:=$};
\node[inner sep=0.8pt] (u2) at (0,1) {\tiny $f_2^1$};

\node[inner sep=0.8pt] (u1) at (0,3) {\tiny $f_1^1$} edge[loop,->,in=225,out=135,looseness=10] node[auto,black,swap] {$\scriptstyle \lambda_1^1$} (u1);
\draw[<-] (u2) -- (u1); \node[inner sep=6pt, above] at (-0.25,1.5) {$\scriptstyle \lambda_2^1$};

\node (v1) at (2,3) {\tiny $f_3^1$};

\node (v2) at (2,1) {\tiny $f_3^2$}
edge[loop,->,in=-45,out=45,looseness=7] node[auto,black,right] {$\scriptstyle \lambda_4^2$} (v2);
\draw[->] (u2) -- (v1); \node[inner sep=6pt, below] at (1,2) {$\scriptstyle \lambda_3^1$};

\draw[<-] (v2) -- (v1); \node[inner sep=6pt, below] at (1.75,2.5) {$\scriptstyle \lambda_4^1$};
\end{scope}
\end{tikzpicture}
\]

As indicated in the proof of Theorem \ref{thm:6.1}, we can also iteratively  perform the textile-system moves of  insplitting (Definition~\ref{dfn:textile-insplit}) and inversion (Definition~\ref{dfn:inverted-textile}) on $T$. 
The resulting textile system (which is non-LR) will give rise to the  identical infinite path space as the infinite path space from the system $T_{\Lambda_I}$ associated to $\Lambda_I$.

\noindent\emph{Step 1: Johnson--Madden insplitting at all $f_i$, $i= 1, 2, 3$, on $T$ using the partition $\mathcal{F}_{f_i}^l$, $l= 1, 2$.} 
For $f_i\in F^0$, $i= 1, 2, 3$, the partitions $\mathcal F^l_{f_i}$ of Equation \eqref{eq:first-JM-insplit} become 
\begin{equation}
\label{eq:F-A-partition}
    \mathcal{F}_{f_1}^1= \{\lambda_1\},  \quad 
    \mathcal{F}_{f_2}^1= \{\lambda_2\},  \quad  
    \mathcal{F}_{f_3}^1= \{\lambda_3\},  \quad \mathcal{F}_{f_3}^2= \{\lambda_4\}.
\end{equation}

\noindent The resulting textile system $T_A= (p_A, q_A : F_A\to  E_A)$ 
has $E_A = E$ and $F_A$ the directed-graph insplit of $F$ using the partition \eqref{eq:F-A-partition};  the graph homomorphisms $p_A,  q_A$ are inherited from $p, q$ respectively.  

The graph $F_A$ is 

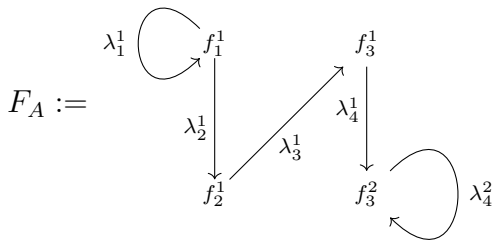
\begin{tikzpicture} 
\begin{scope}[xshift=7cm,yshift=-1cm]
\node at (-2.2,2.2) {$F_A:=$};
\node[inner sep=0.8pt] (u2) at (0,1) {\tiny $f_2^1$};

\node[inner sep=0.8pt] (u1) at (0,3) {\tiny $f_1^1$} edge[loop,->,in=225,out=135,looseness=10] node[auto,black,swap] {$\scriptstyle \lambda_1^1$} (u1);
\draw[<-] (u2) -- (u1); \node[inner sep=6pt, above] at (-0.25,1.5) {$\scriptstyle \lambda_2^1$};

\node (v1) at (2,3) {\tiny $f_3^1$};

\node (v2) at (2,1) {\tiny $f_3^2$}
edge[loop,->,in=-45,out=45,looseness=7] node[auto,black,right] {$\scriptstyle \lambda_4^2$} (v2);
\draw[->] (u2) -- (v1); \node[inner sep=6pt, below] at (1,2) {$\scriptstyle \lambda_3^1$};

\draw[<-] (v2) -- (v1); \node[inner sep=6pt, below] at (1.75,2.5) {$\scriptstyle \lambda_4^1$};
\end{scope}
\end{tikzpicture}

and the commuting squares are: 
\[
\begin{tikzpicture}
\begin{scope}[xshift=2cm]
\node[inner sep=0.8pt] (00) at (0,0) {$ $};
\node[inner sep=0.8pt] (10) at (1,0) {$ $}
	edge[blue] node[auto,black] {\tiny $e_1$} (00);
\node[inner sep=0.8pt] (01) at (0,1) {$ $}
	edge[dashed,red] node[auto,black,swap] {\tiny $f^1_1$} (00);
\node[inner sep=0.8pt] (11) at (1,1) {$ $}
	edge[blue] node[auto,black,swap] {\tiny $e_1$} (01)
	edge[dashed,red] node[auto,black] {\tiny $f^1_1$} (10);
\node[inner sep=0.8pt]  at (0.5,0.5) {$\scriptstyle \lambda_1^1$};
\end{scope}

\begin{scope}[xshift=5cm]
\node[inner sep=0.8pt] (00) at (0,0) {$ $};
\node[inner sep=0.8pt] (10) at (1,0) {$ $}
	edge[blue] node[auto,black] {\tiny $e_2$} (00);
\node[inner sep=0.8pt] (01) at (0,1) {$ $}
	edge[dashed,red] node[auto,black,swap] {\tiny $f_2^1$} (00);
\node[inner sep=0.8pt] (11) at (1,1) {$ $}
	edge[blue] node[auto,black,swap] {\tiny $e_1$} (01)
	edge[dashed,red] node[auto,black] {\tiny $f_1^1$} (10);
\node[inner sep=0.8pt]  at (0.5,0.5) {$\scriptstyle \lambda_2^1$};
\end{scope}

\begin{scope}[xshift=8cm]
\node[inner sep=0.8pt] (00) at (0,0) {$ $};
\node[inner sep=0.8pt] (10) at (1,0) {$ $}
	edge[blue] node[auto,black] {\tiny $e_3$} (00);
\node[inner sep=0.8pt] (01) at (0,1) {$ $}
	edge[dashed,red] node[auto,black,swap] {\tiny $f_3^1$} (00);
\node[inner sep=0.8pt] (11) at (1,1) {$ $}
	edge[blue] node[auto,black,swap] {\tiny $e_2$} (01)
	edge[dashed,red] node[auto,black] {\tiny $f_2^1$} (10);
\node[inner sep=0.8pt]  at (0.5,0.5) {$\scriptstyle \lambda_3^1$};
\end{scope}

\begin{scope}[xshift=11cm]
\node[inner sep=0.8pt] (00) at (0,0) {$ $};
\node[inner sep=0.8pt] (10) at (1,0) {$ $}
	edge[blue] node[auto,black] {\tiny $e_3$} (00);
\node[inner sep=0.8pt] (01) at (0,1) {$ $}
	edge[red,dashed] node[auto,black,swap] {\tiny $f_3^2$} (00);
\node[inner sep=0.8pt] (11) at (1,1) {$ $}
	edge[blue] node[auto,black,swap] {\tiny $e_3$} (01)
	edge[dashed,red] node[auto,black] {\tiny $f_3^1$} (10);
\node[inner sep=0.8pt]  at (0.5,0.5) {$\scriptstyle \lambda_4^1$};
\end{scope}

\begin{scope}[xshift=14cm]
\node[inner sep=0.8pt] (00) at (0,0) {$ $};
\node[inner sep=0.8pt] (10) at (1,0) {$ $}
	edge[blue] node[auto,black] {\tiny $e_3$} (00);
\node[inner sep=0.8pt] (01) at (0,1) {$ $}
	edge[red,dashed] node[auto,black,swap] {\tiny $f_3^2$} (00);
\node[inner sep=0.8pt] (11) at (1,1) {$ $}
	edge[blue] node[auto,black,swap] {\tiny $e_3$} (01)
	edge[dashed,red] node[auto,black] {\tiny $f_3^2$} (10);
\node[inner sep=0.8pt]  at (0.5,0.5) {$\scriptstyle \lambda_4^2$};
\end{scope}
\end{tikzpicture}
\]

\noindent\emph{Step 2: Invert $T_A$ to get $T_B= (p_B, q_B: F_B \to E_B)$, where $F_B^0= E_A^1= E^1$, $F_B^1= F_A^1= \{\lambda_1^1, \lambda_2^1, \lambda_3^1, \lambda_4^i: i= 1, 2\}$, $E_B^0= E_A^0= E^0$, and $E_B^1= F_A^0= \{f_1^1, f_2^1, f_3^i: i= 1, 2\}$.} 
The corresponding graphs are

\begin{tikzpicture}
    

\node at (-2.2,1.2) {$F_B:=$};

\node[inner sep=0.8pt] (u) at (0,1) {\tiny $e_1$} edge[loop,->,in=225,out=135,looseness=10] node[auto,black,swap] {$\scriptstyle {\lambda}_1^1$} (u);

\node[inner sep=0.8pt] (v1) at (4,1) {\tiny $e_3$} edge[loop,->,in=-35,out=35,looseness=10] node[auto,black,right] {$\scriptstyle {\lambda}_4^1$} (v1) edge[loop,->,in=-47,out=47,looseness=30] node[auto,black,right] {$\scriptstyle {\lambda}_4^2$} (v1);
\node (v2) at (2,2) {\tiny $e_2$};
\draw[<-] (v1) -- (v2);
\node at (3,1.75) {\tiny ${\lambda}_3^1$};

\draw[->] (u) -- (v2);
\node at (1,1.75) {\tiny ${\lambda}_2^1$};

\begin{scope}[xshift=10cm,yshift=0cm]
\node at (-2.2,1.2) {$E_B:=$};
\node[inner sep=0.8pt] (u) at (0,1) {$\scriptstyle v$}
edge[black,loop,->,in=225,out=135,looseness=15] node[auto,black,left] {$\scriptstyle f^1_1$} (u);

\node[inner sep=0.8pt] (v) at (2,1) {$\scriptstyle w$}
edge[black,loop,->,in=-45,out=45,looseness=15] node[auto,black,right] {$\scriptstyle f^2_3$} (v)
edge[loop,->,in=-45,out=45,looseness=30,black] node[auto,black,right] {$\scriptstyle f^1_3$} (v);

\draw[black,<-] (v.south west)
	parabola[parabola height=-0.5cm] (u.south east);
\node[inner sep=6pt, below] at (1,0.5) {$\scriptstyle f^1_2$};
\end{scope}
\end{tikzpicture}

\noindent
The commuting squares are the transposes of the commuting squares from Step 1.

\noindent\emph{Step 3: Johnson--Madden insplit $T_B$ at all $e_i$, $i= 1, 2, 3$, using the partition $\mathcal{H}^\ell_{e_i}= 
\{ \lambda^\ell  : \lambda \in F^1 , \ e_i = q( \lambda ) , \ s ( \lambda )   \in  \mathcal G_{s(e_i)}^\ell  
 \} $ for $\ell \leq 2$:}
This definition yields  
\begin{align*}
    \mathcal{H}^1_{e_1}= \{{\lambda}_1^1\}, & \quad 
    \mathcal{H}^1_{e_2}= \{{\lambda}_2^1\}, \\
    \mathcal{H}^1_{e_3}= \{{\lambda}^1_3\}, & \quad \mathcal{H}^2_{e_3}= \{{\lambda}_4^1, {\lambda}_4^2\}.
\end{align*}
The textile system $T_C= (p_C, q_C : F_C \to E_C)$ consists of the graph $F_C$ with $F_C^0= \{e_1^1=e_1, e_2^1=e_2, e_3^1, e_3^2\}$, $F_C^1= \{{\lambda}_k^{j^i}: k= 1, 2, 3, 4 \ \ \text{and} \ \ i, j \leq 2\}$, and the graph $E_C= E_B$.
The graph $F_C$ is as follows:

\begin{tikzpicture} 
\begin{scope}[xshift=7cm,yshift=-1cm]
\node at (-2.2,0.2) {$F_C:=$};

\node[inner sep=0.8pt] (u1) at (0,1) {\tiny $e_1^1$} edge[loop,->,in=225,out=135,looseness=5] node[auto,black,swap] {$\scriptstyle \lambda_1^{1^1}$} (u1); 
\node[inner sep=0.8pt] (w2) at (6,1) {\tiny $e_3^2$} edge[loop,->,in=-50,out=50,looseness=20] node[auto,black,right] {$\scriptstyle \lambda_4^{2^2}$} (w2)edge[loop,->,in=-45,out=45,looseness=5] node[auto,black,right] {$\scriptstyle \lambda_4^{1^2}$} (w2);
\node (v1) at (2,1) {\tiny $e_2^1$};
\draw[->] (u1.east)--(v1.west)node[inner sep=6pt,above] at (1,1) {$\scriptstyle \lambda_2^{1^1}$};

\node[inner sep=0.8pt] (w1) at (4,1) {\tiny $e_3^1$};
\draw[->] (v1.east)--(w1.west)node[inner sep=6pt,above] at (3,1) {$\scriptstyle \lambda_3^{1^1}$};
\draw[<-] (w2) -- (w1); \node[inner sep=6pt, below] at (4.65,0.65) {$\scriptstyle \lambda_4^{2^1}$};
\draw[->] (w1.320) to [bend right] (w2.220) node[
inner sep=12pt, above] at (5,0.8) {$\scriptstyle \lambda_4^{1^1}$};

\end{scope}
\end{tikzpicture}

The commuting squares are:
\[
\begin{tikzpicture}
\begin{scope}[xshift=3cm]
\node[inner sep=0.8pt] (00) at (0,0) {$. $};
\node[inner sep=0.8pt] (10) at (1,0) {$. $}
	edge[->,blue] node[auto,black] {\tiny $f^1_1$} (00);
\node[inner sep=0.8pt] (01) at (0,1) {$. $}
	edge[->,dashed,red] node[auto,black,swap] {\tiny $e_1^1$} (00);
\node[inner sep=0.8pt] (11) at (1,1) {$. $}
	edge[->,blue] node[auto,black,swap] {\tiny $f^1_1$} (01)
	edge[->,dashed,red] node[auto,black] {\tiny $e_1^1$} (10);
\node[inner sep=0.8pt]  at (0.5,0.5) {$\scriptstyle {\lambda}_1^{1^1}$};
\end{scope}

\begin{scope}[xshift=7cm]
\node[inner sep=0.8pt] (00) at (0,0) {$. $};
\node[inner sep=0.8pt] (10) at (1,0) {$. $}
	edge[->,blue] node[auto,black] {\tiny $f^1_2$} (00);
\node[inner sep=0.8pt] (01) at (0,1) {$. $}
	edge[->,dashed,red] node[auto,black,swap] {\tiny $e_2^1$} (00);
\node[inner sep=0.8pt] (11) at (1,1) {$. $}
	edge[->,blue] node[auto,black,swap] {\tiny $f^1_1$} (01)
	edge[->,dashed,red] node[auto,black] {\tiny $e_1^1$} (10);
\node[inner sep=0.8pt]  at (0.5,0.5) {$\scriptstyle {\lambda}_2^{1^1}$};
\end{scope}

\begin{scope}[xshift=11cm]
\node[inner sep=0.8pt] (00) at (0,0) {$. $};
\node[inner sep=0.8pt] (10) at (1,0) {$. $}
	edge[->,blue] node[auto,black] {\tiny $f^1_3$} (00);
\node[inner sep=0.8pt] (01) at (0,1) {$. $}
	edge[->,dashed,red] node[auto,black,swap] {\tiny $e_3^1$} (00);
\node[inner sep=0.8pt] (11) at (1,1) {$. $}
	edge[->,blue] node[auto,black,swap] {\tiny $f^1_2$} (01)
	edge[->,dashed,red] node[auto,black] {\tiny $e_2^1$} (10);
\node[inner sep=0.8pt]  at (0.5,0.5) {$\scriptstyle {\lambda}_3^{1^1}$};
\end{scope}

\end{tikzpicture}
\]

\[
\begin{tikzpicture}
\begin{scope}[xshift=2cm]
\node[inner sep=0.8pt] (00) at (0,0) {$. $};
\node[inner sep=0.8pt] (10) at (1,0) {$. $}
	edge[->,blue] node[auto,black] {\tiny $f_3^2$} (00);
\node[inner sep=0.8pt] (01) at (0,1) {$. $}
	edge[->,red,dashed] node[auto,black,swap] {\tiny $e_3^2$} (00);
\node[inner sep=0.8pt] (11) at (1,1) {$. $}
	edge[->,blue] node[auto,black,swap] {\tiny $f_3^1$} (01)
	edge[->,dashed,red] node[auto,black] {\tiny $e_3^1$} (10);
\node[inner sep=0.8pt]  at (0.5,0.5) {$\scriptstyle {\lambda}_4^{1^1}$};
\end{scope}

\begin{scope}[xshift=4.5cm]
\node[inner sep=0.8pt] (00) at (0,0) {$. $};
\node[inner sep=0.8pt] (10) at (1,0) {$. $}
	edge[->,blue] node[auto,black] {\tiny $f_3^2$} (00);
\node[inner sep=0.8pt] (01) at (0,1) {$. $}
	edge[->,red,dashed] node[auto,black,swap] {\tiny $e_3^2$} (00);
\node[inner sep=0.8pt] (11) at (1,1) {$. $}
	edge[->,blue] node[auto,black,swap] {\tiny $f_3^1$} (01)
	edge[->,dashed,red] node[auto,black] {\tiny $e_3^2$} (10);
\node[inner sep=0.8pt]  at (0.5,0.5) {$\scriptstyle {\lambda}_4^{1^2}$};
\end{scope}

\begin{scope}[xshift=7cm]
\node[inner sep=0.8pt] (00) at (0,0) {$. $};
\node[inner sep=0.8pt] (10) at (1,0) {$. $}
	edge[->,blue] node[auto,black] {\tiny $f_3^2$} (00);
\node[inner sep=0.8pt] (01) at (0,1) {$. $}
	edge[->,red,dashed] node[auto,black,swap] {\tiny $e_3^2$} (00);
\node[inner sep=0.8pt] (11) at (1,1) {$. $}
	edge[->,blue] node[auto,black,swap] {\tiny $f_3^2$} (01)
	edge[->,dashed,red] node[auto,black] {\tiny $e_3^1$} (10);
\node[inner sep=0.8pt]  at (0.5,0.5) {$\scriptstyle {\lambda}_4^{2^1}$};
\end{scope}

\begin{scope}[xshift=9.5cm]
\node[inner sep=0.8pt] (00) at (0,0) {$. $};
\node[inner sep=0.8pt] (10) at (1,0) {$. $}
	edge[->,blue] node[auto,black] {\tiny $f_3^2$} (00);
\node[inner sep=0.8pt] (01) at (0,1) {$. $}
	edge[->,red,dashed] node[auto,black,swap] {\tiny $e_3^2$} (00);
\node[inner sep=0.8pt] (11) at (1,1) {$. $}
	edge[->,blue] node[auto,black,swap] {\tiny $f_3^2$} (01)
	edge[->,dashed,red] node[auto,black] {\tiny $e_3^2$} (10);
\node[inner sep=0.8pt]  at (0.5,0.5) {$\scriptstyle {\lambda}_4^{2^2}$};
\end{scope}
\end{tikzpicture}
\]

\noindent\emph{Step 4: Invert $T_C$ to get the final textile system $T_D:= (p_D, q_D: F_D \to E_D)$, where $F_D^0= E_C^1= \{f^i_1, f^i_2, f_3^i: i \leq  2\}$, $F_D^1= F_C^1= \{\lambda_k^{j^i}: k= 1, 2, 3, 4 \text{ and } i, j \leq 2\}$, $E_D^0= E_C^0= \{v, w\}$, and $E_D^1= F_C^0= \{e_1^i, e_2^i, e_3^i: i \leq 2\}$.}  That is, we have 

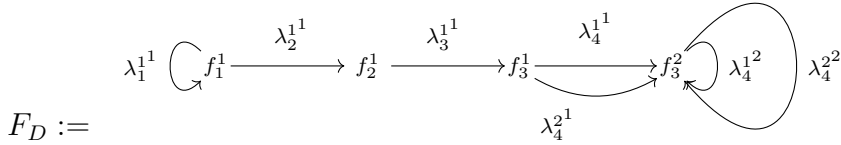
\begin{tikzpicture} 
\begin{scope}[xshift=7cm,yshift=-1cm]
\node at (-2.2,0.2) {$F_D:=$};

\node[inner sep=0.8pt] (u1) at (0,1) {\tiny $f_1^1$} edge[loop,->,in=225,out=135,looseness=5] node[auto,black,swap] {$\scriptstyle \lambda_1^{1^1}$} (u1); 
\node[inner sep=0.8pt] (w2) at (6,1) {\tiny $f_3^2$} edge[loop,->,in=-50,out=50,looseness=20] node[auto,black,right] {$\scriptstyle \lambda_4^{2^2}$} (w2)edge[loop,->,in=-45,out=45,looseness=5] node[auto,black,right] {$\scriptstyle \lambda_4^{1^2}$} (w2);
\node (v1) at (2,1) {\tiny $f_2^1$};
\draw[->] (u1.east)--(v1.west)node[inner sep=6pt,above] at (1,1) {$\scriptstyle \lambda_2^{1^1}$};

\node[inner sep=0.8pt] (w1) at (4,1) {\tiny $f_3^1$};
\draw[->] (v1.east)--(w1.west)node[inner sep=6pt,above] at (3,1) {$\scriptstyle \lambda_3^{1^1}$};
\draw[<-] (w2) -- (w1); \node[inner sep=6pt, below] at (4.5,0.65) {$\scriptstyle \lambda_4^{2^1}$};
\draw[->] (w1.320) to [bend right] (w2.220) node[inner sep=8pt,above] at (5,1)
{$\scriptstyle \lambda_4^{1^1}$};

\end{scope}
\end{tikzpicture}

\noindent 
and $E_D \cong E_C = E_B$ by relabeling each edge $f_i^j \in E_B^1$ as $e_i^j$.
We obtain the commuting squares for $T_D$ by  flipping
the 7 commuting squares obtained in Step 3 along the diagonal. 

To recover the infinite path space associated to $T_{\Lambda_I}$ from the above four Johnson–Madden moves, we will delete the edges  $\lambda_{k}^{j^i} \in  F_D$ where  $i\neq j$. 
Note that this procedure converts $F_D$ into 
$F_{\Lambda_I}$.

As further evidence that the conjugacy ${\sf X}_{T_I}^+ \cong {\sf X_D}^+$ guaranteed by  Theorem \ref{thm:6.1} does not arise from an isomorphism of textile systems, however, note that we do not have $E_D \cong E_{\Lambda_I}$, as $E_{\Lambda_I}$ has three vertices and $E_D$ has two. 
\end{example}

\section{When textile-system insplitting yields a 2-graph insplitting}\label{sec:JM-ghinsp}

In the previous section, we showed in Theorem \ref{thm:6.1} how to reconstruct a 2-graph insplitting via  four Johnson--Madden moves.  In this section, we identify several other perspectives on the relationship between 2-graph insplitting and Johnson--Madden insplitting.  We first show (Theorem \ref{thm:Priyanga}) that, alternatively, the 2-graph insplitting may be reconstructed via a single Johnson--Madden insplit, together with an insplitting of the base graph $E$ of the textile system.  Then, Theorem \ref{thm:lr-insplit=2g-insplit} shows that certain Johnson--Madden insplits yield 2-graph insplits, and Theorem \ref{Thm: Main result III} establishes 
that certain insplits of the base graph $E$ of a textile system yield both Johnson--Madden insplits and 2-graph insplits.  Finally, Section \ref{subsec: unification} reveals the compatibility of these perspectives.

These three theorems 
can be summarized by the diagrams below, where the solid arrows indicate relationships that are assumed (or known in general), and the dotted arrows indicate relationships which are constructed in the theorem.

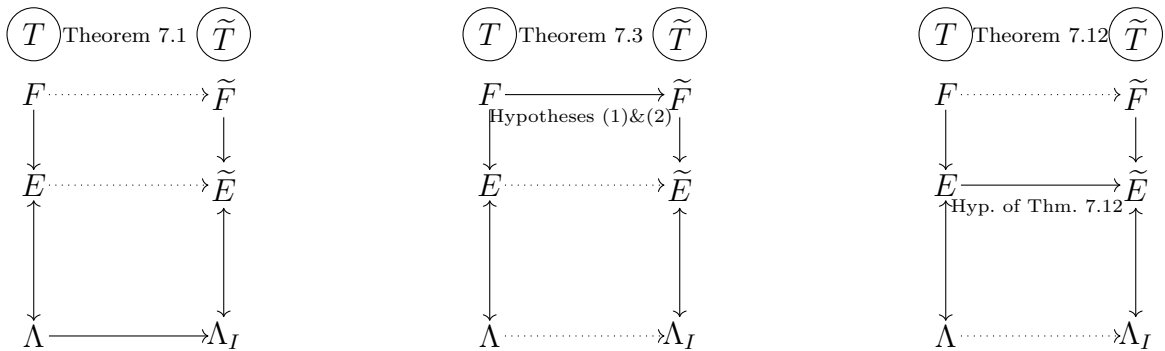
\[ \begin{tikzpicture}

\node at (0, 4) {$T$};
\node at (1.2,4) {\tiny{Theorem \ref{thm:Priyanga}}};
\node (MR12) at (0,3.2) {$F$};
\node (MR14) at (0,2) {$E$};
\node (MR12) at (0,0) {$\Lambda$};
\draw[->] (0,3) -- (0,2.2);
\draw[<->] (0,1.8) -- (0,0.2);
\draw  (0,4) circle (10pt);

\node at (2.5, 4) {$\widetilde T$};
\node (MR22) at (2.5,3.2) {$\widetilde{F}$};
\node (MR24) at (2.5,2) {$\widetilde{E}$};
\node (MR32) at (2.5,0) {$\Lambda_I$};
\draw[->] (2.5,2.9) -- (2.5,2.3);
\draw[<->] (2.5,1.7) -- (2.5,0.25);

\draw[->,dotted] (0.2,3.2) -- (2.3,3.2);
\draw[->,dotted] (0.2,2) -- (2.3,2);
\draw[->] (0.2,0) -- (2.3,0);
\draw (2.5, 4) circle (10pt);

\begin{scope}[xshift=6cm]
    
\node at (0, 4) {$T$};
\node at (1.2,4) {\tiny{Theorem \ref{thm:lr-insplit=2g-insplit}}};
\node (MR12) at (0,3.2) {$F$};
\node (MR14) at (0,2) {$E$};
\node (MR12) at (0,0) {$\Lambda$};
\draw[->] (0,3) -- (0,2.2);
\draw[<->] (0,1.8) -- (0,0.2);
\draw  (0,4) circle (10pt);

\node at (2.5, 4) {$\widetilde T$};
\node (MR22) at (2.5,3.2) {$\widetilde{F}$};
\node (MR24) at (2.5,2) {$\widetilde{E}$};
\node (MR32) at (2.5,0) {$\Lambda_I$};
\draw[->] (2.5,2.9) -- (2.5,2.3);
\draw[<->] (2.5,1.7) -- (2.5,0.25);

\draw[->] (0.2,3.2) -- (2.3,3.2);
\draw[->,dotted] (0.2,2) -- (2.3,2);
\draw[->,dotted] (0.2,0) -- (2.3,0);
\node at (1.2,2.9) {\Tiny{Hypotheses (1)\&(2)}};
\draw (2.5, 4) circle (10pt);

\end{scope}
\begin{scope}[xshift=12cm]
    
\node at (0, 4) {$T$};
\node at (1.25,4) {\tiny{Theorem \ref{Thm: Main result III}}};
\node (MR12) at (0,3.2) {$F$};
\node (MR14) at (0,2) {$E$};
\node (MR12) at (0,0) {$\Lambda$};
\draw[->] (0,3) -- (0,2.2);
\draw[<->] (0,1.8) -- (0,0.2);
\draw  (0,4) circle (10pt);

\node at (2.5, 4) {$\widetilde T$};
\node (MR22) at (2.5,3.2) {$\widetilde{F}$};
\node (MR24) at (2.5,2) {$\widetilde{E}$};
\node (MR32) at (2.5,0) {$\Lambda_I$};
\draw[->] (2.5,2.9) -- (2.5,2.3);
\draw[<->] (2.5,1.7) -- (2.5,0.25);

\draw[->,dotted] (0.2,3.2) -- (2.3,3.2);
\draw[->] (0.2,2) -- (2.3,2);
\draw[->,dotted] (0.2,0) -- (2.3,0);
\node at (1.2,1.7) {\Tiny{Hyp.~of Thm.~\ref{Thm: Main result III}}};
\draw (2.5, 4) circle (10pt);

\end{scope}

\end{tikzpicture}\]

The results in Section \ref{subsec: unification} confirm that the diagrams above can be superimposed on each other; or, equivalently, that any one of the horizontal arrows in the diagram determines the others.  For example, the dotted arrow $F \to \widetilde F$ in the first diagram (which is constructed in Theorem \ref{thm:Priyanga} ) satisfies the hypothesis in Theorem \ref{thm:lr-insplit=2g-insplit}.

\subsection{Textile insplitting from 2-graph insplitting}
While Theorem \ref{thm:6.1} shows that a 2-graph insplit can be recovered by four Johnson--Madden moves, we establish in the following Theorem that we could alternatively recover the 2-graph insplitting by combining a single Johnson--Madden insplitting with an insplit of the base graph $E$ of the textile system.

\begin{theorem}\label{thm:Priyanga}
Let $T = (p,q:F\to E)$ be an LR textile system. 
Let $\Lambda$ denote the 2-graph associated to $T$ and assume that $\Lambda$ can be 
(2-graph) insplit using the partition $\{ \mathcal{G}^j_z : z \in E^0, 1\leq j \leq m(z)\}$ 
of $\Lambda^1$.
For each  $v \in F^0$, define
\begin{equation}
    \mathcal{F}^i_v := \{ \lambda \in F^1: r(\lambda) = v,\;\; p(\lambda) \in \mathcal{G}^i_{r(p(\lambda)) = p(v)} \},
    \label{eq:F-partition-sets}
\end{equation}
and for $ z \in E^0$, define
\begin{equation} 
\mathcal{E}^i_z := \mathcal{G}^i_z \cap z E^1.
\label{eq:E-partition-sets}
\end{equation}
Denote by $\wt{F}$ the directed graph resulting from directed-graph insplitting on $F$ at every vertex $v \in F^0$, using the partition 
$\{ \mathcal F_v^i: 1 \leq i \leq m(p(v)), v \in F^0\}.$
Let $\wt{E}$ denote the directed graph resulting from directed-graph insplitting on $E$ at each vertex $z \in E^0$, using the partition 
$\{ \mathcal E^i_z: 1 \leq i \leq m(z), z \in E^0\}.$
Define maps $\wt{p}, \wt{q}: \wt{F} \to \wt{E}$ by 
\begin{align*}
\wt{p}(\lambda^i) & :=  p(\lambda)^i, \ \wt{q}(\lambda^i) :=  q(\lambda)^m, \mbox{ where } s(\lambda) \in \mathcal{G}^m_{q(s(\lambda))}, \\
\wt{p}(v^i) & :=  p(v)^i, \ \wt{q}(v^i)  :=  q(v)^m, \mbox{ where } v \in \mathcal{G}^m_{q(v)},  
\end{align*}
where $\lambda \in F^1$, $v \in F^0$. 
Then 
\begin{enumerate}
\item The maps $\wt{p}, \wt{q}:  \wt{F}  \to  \wt{E}$ are 
graph homomorphisms.
\item $\wt{T} := (\wt{p}, \wt{q} : \wt{F} \to \wt{E} )$ is an LR textile system.
\item The 2-graph associated to $\wt{T}$ (denoted by $\wt{\Lambda}$) is identical to the 2-graph ${\Lambda}_I$ resulting from insplitting  $\Lambda$ using the partition $\{ \mathcal G^j_z: z \in E^0, 1 \leq j \leq m(z)\}$. 
\end{enumerate}
\end{theorem}

\begin{proof}    
To see (1), 
we will check that $r(\wt p(\lambda^i)) = \wt p(r(\lambda^i))$ and that $r (\wt q(\lambda^i)) = \wt q(r(\lambda^i))$ for all $\lambda^i \in \wt F^1$. (The checks that $s\circ \wt p = \wt p \circ s$ and $s \circ \wt q = \wt q \circ s$ are straightforward computations which we leave to the reader.)
 For the first computation,  if $\lambda \in \mathcal F^j_{r(\lambda)}$ then $r(\lambda^i) = r(\lambda)^j$ and $p(\lambda) \in \mathcal G^j_{p(r(\lambda))}$. 
 If $r(p(\lambda)) (=p(r(\lambda))) =  z \in E^0$, then when we insplit $E$ we have $r(p(\lambda)^k) = z^j$ for all $k$.  Therefore, 
\[ r(\wt p(\lambda^i)) = r(p(\lambda)^i) = z^j \qquad \text{ and } \qquad \wt p(r(\lambda^i)) = \wt p(r(\lambda)^j) = z^j,\]
as desired. 

To see that $r(\wt q(\lambda^i)) = \wt q(r(\lambda^i))$, observe that since $\mathcal G$ is a 2-graph insplitting partition, we will have $r(\lambda)$ and $q(\lambda)$ in the same partition set $\mathcal G^n_{r(q(\lambda))}$.  Therefore, regardless of which $m$ satisfies $s(\lambda) \in \mathcal G^m_{q(s(\lambda))}$, we have 
\[ r(\wt q(\lambda^i)) = r(q(\lambda)^m) = r(q(\lambda))^n,\]
and  regardless of which $\ell$ satisfies $\lambda \in \mathcal F^\ell_{r(\lambda)}$, we have 
$ \wt q(r(\lambda^i)) = \wt q(r(\lambda)^\ell) = q(r(\lambda))^n = r(q(\lambda))^n$,
since $q$ is a graph homomorphism.  We conclude that $\wt p, \wt q$ are graph homomorphisms, as claimed.

For (2), to see that $\wt{T}$ is an LR textile system, we  first observe that the function $\wt{F}^1 \ni \lambda^i \mapsto (r(\lambda^i), \wt{p}(\lambda^i), s(\lambda^i), \wt{q}(\lambda^i)) \in \wt{F}^0 \times \wt{E}^1 \times \wt{F}^0 \times \wt{E}^1$ is injective, because if $r(\lambda^i) = r(\mu^j)$ then $r(\lambda) = r(\mu) =: v$ and $\lambda, \mu \in \mathcal F^n_v$ lie in the same partition set.  
Furthermore, if $\wt p(\lambda^i) = \wt p(\mu^j)$ then $p(\lambda)^i = p(\mu)^j$ so we must have $p(\mu) = p(\lambda)$ and $i=j$.  
Since $T$ is LR, $r(\lambda) = r(\mu)$ and $p(\lambda) = p(\mu)$, we conclude that $\lambda = \mu$. As $i=j$ we also have $\lambda^i = \mu^j$, so the function is injective as claimed and $\wt T$ is a textile system.

To see that  $\wt{T}$ is LR, observe that if $\wt{p}(v^i)= r(e^j)\in \wt{E}^0$, then the definitions of the insplitting $\wt E $ and the homomorphism $\wt p$ imply that  we have $p(v)^i= r(e)^n$ where $e \in \mathcal E^n_{r(e)}$.  In particular, we have $p(v) = r(e)$ and $i=n$.  Since the textile system $T$ is LR, there exists a unique $\lambda\in F^1$ with $p(\lambda)= e$ and $r(\lambda)= v$.  As $p(\lambda) = e \in \mathcal G^i_{p(v)}$,  we have $\lambda \in \mathcal F^i_v$, and so $r(\lambda^j)  = v^i$.  The uniqueness of $\lambda$ with $p(\lambda) = e$ and $r(\lambda) = v$ implies that $\lambda^j \in \wt F^1$ is the unique edge with $\wt p(\lambda^j) = e^j$ and $r(\lambda^j) = v^i$.  That is, $\wt{p}$ is left resolving. 
To check that $\wt q$ is right resolving, suppose that $\wt{q}(v^i)= s(e^m)\in \wt{E}^0$ for some $e^m\in \wt{E}^1$. As $s(e^m) = s(e)^m$ in $\wt E^0$, the definition of $\wt q$ implies that $s(e) = q(v)$, and moreover that $v \in \mathcal G^m_{q(v)}$. 
Since $T$ is LR, i.e.~$q$ is right resolving, there exists a unique $\lambda \in F^1$ such that $q(\lambda)= e$ and $s(\lambda)= v$. 
For any $1 \leq j \leq m(p(r(\lambda)))$, we will then have   $\wt{q}(\lambda^j)= q(\lambda)^m= e^m$ and $s(\lambda^j)= s(\lambda)^j= v^j$. That is, there is a unique edge $f$ in $\wt F^1$ with source $v^i$ and $q(f) =e^m$, namely, $f = \lambda^i$.  We conclude that $\wt{q}$ is right resolving, and so 
$\wt{T}$ is LR.

For (3), we claim that the 2-graph $\Lambda_I$ resulting from 2-graph insplitting $\Lambda$ 
according to the partition $\{ \mathcal G^i_z: z \in E^0, 1\leq i \leq m(z)\}$
is identical to the graph $\wt{\Lambda}$ associated to the textile system $\wt{T}$. 
To see this, we first recall that since each set $\mathcal G^j_z$ contains at least one edge of each color, each set $\mathcal E^i_z$ is nonempty.  Thus, 
\[ \wt \Lambda^0 =\wt E^0 = \{ z^i: z \in E^0, 1\leq i \leq m(z)\}.\]
Similarly, $ \Lambda_I^0 = \{ z^i: z \in E^0, 1 \leq i \leq m(z)\},$
so $\Lambda_I^0 = \wt \Lambda^0.$

Next, note that $\wt \Lambda^{\varepsilon_1} = \wt E^1 = \{ e^j: e \in E^1, 1 \leq j \leq m(s(e))\},$ while 
\[ \Lambda_I^{\varepsilon_1} = \{e^i:  e \in \Lambda^{\varepsilon_1}, 1 \leq i \leq m(s(e))\} =  \{ e^i: e \in E^1, 1 \leq i \leq m(s(e))\} = \wt \Lambda^{\varepsilon_1}.\]
 As $s_\Lambda(v) = p(v)$, we also have $\Lambda_I^{\varepsilon_2} = \{ v^i: v \in F^0, 1 \leq i \leq m(s_\Lambda(v))\} $, while
\[ \wt \Lambda^{\varepsilon_2} = \{ v^j: v \in F^0,  1 \leq j \leq m(p(v))\} = \Lambda^{\varepsilon_2}_I.\]
Moreover, the range and source maps are the same in $\Lambda_I$ as in $\wt \Lambda$.  To see this, we first observe that for each $e^i\in \wt{E}^1 = \wt \Lambda^{\varepsilon_1} = \Lambda_I^{\varepsilon_1},$ we have $r_{\wt{\Lambda}}(e^i)= r_{\wt E}(e)= r(e)^j$ if $e \in \mathcal E^j_{r(e)} \subseteq \mathcal G^j_{r(e)}$, while $r_{\Lambda_I}(e^i) = r(e)^j$.  Similarly,  $s_{\wt{\Lambda}}(e^i)= s_{\wt{E}}(e^i)= s_E(e)^i$, and $s_{\Lambda_I}(e^i) = s_\Lambda(e)^i = s_E(e)^i$.  For $v^i\in \wt{F}^0$, we also have $s_{\wt{\Lambda}}(v^i)= \wt{p}(v^i)= p(v)^i$, whereas $ s_{\Lambda_I}(v^i) = s_\Lambda(v)^i = p(v)^i$. Finally,  if $v \in \mathcal G^m_{q(v)},$ $r_{\wt{\Lambda}}(v^i)= \wt{q}(v^i)= q(v)^m$ and $ r_{\Lambda_I}(v^i) = r_\Lambda(v)^m = q(v)^m$ since  $v\in \mathcal{G}^m_{q(v)}$.

It remains to check that the factorization rules are the same in $\Lambda_I$ as in $\wt \Lambda$. To that end, suppose that 
$w^k f^\ell$ is a composable 2-color path in $\Lambda_I$ with $w^k \in \Lambda^{\varepsilon_2}_I, f^\ell \in \Lambda^{\varepsilon_1}_I$. As $\Lambda_I$ is a 2-graph by Theorem \ref{thm:inplitcomplete}, there is a unique composable path $e^i v^j$ with $e^i v^j \sim_{\Lambda_I} w^k f^\ell.$  

Since $s_{\Lambda_I}(w^k) = s_\Lambda(w)^k = p(w)^k$, we must have $f \in \G^k_{r(f) = p(w)}$.  Since $r(f) = p(w)$ and $p$ is left resolving, there is a unique $\lambda \in F^1$ with $p(\lambda) = f, r(\lambda) = w$.  
If $e^i v^j \sim_{\Lambda_I} w^k f^\ell$, then in particular $ev \sim_\Lambda w f$. That is, the unique $\lambda \in F^1$ which satisfies $p(\lambda) = f, r(\lambda) =w$ also has $q(\lambda) = e, s(\lambda) = v$ (and vice versa).  Moreover, 
as $s_{\Lambda_I}(v^j) = p(v)^j$ and $s_{\Lambda_I}(f^\ell) = s(f)^\ell$, we must have $j = \ell.$

Now, consider $\sim_{\wt{\La}}$.  We will show that we also have $w^k f^\ell \sim_{\wt \Lambda} e^i v^\ell$. Since the source and range maps are the same in $\wt \Lambda$ and in $\Lambda_I$, we know that  $w^k f^\ell$ is composable in $\wt \Lambda$, that is, $r_{\wt \Lambda} (f^\ell)= s_{\wt \Lambda}(w^k) =  \wt p(w^k)$.  In other words,   $f \in \mathcal G^k_{p(w)}$. As we established when we checked (as part of Statement (2)) that $\widetilde p$ was left resolving, there is a unique $\lambda^\ell \in \wt F^1$ with $\wt p(\lambda^\ell) = f^\ell = p(\lambda)^\ell$ and $r(\lambda^\ell)= w^k.$ 
In particular, $\lambda \in F^1$ satisfies $p(\lambda) = f$ and $r(\lambda) = w$, so $q(\lambda) = e$ and $s(\lambda) = v$ for the $e, v$ discussed in the previous paragraph.  That is, we have $v \in \G^i_{s(e)}.$
Thus, $w^k f^\ell \sim_{\wt \Lambda} e^i v^\ell$, since $v \in \mathcal G^i_{s(e)}$ implies that 
$ r_{\wt \Lambda}(v^\ell) = \wt q(v^\ell) = q(v)^i = s(e)^i = s(e^i).$

As both $\wt \Lambda$ and $\Lambda_I$ are 2-graphs (thanks to Theorems \ref{thm:PST} and \ref{thm:inplitcomplete} respectively),  every blue-red path is equivalent to a unique red-blue path. The fact that any 2-color path $w^k f^\ell$ is equivalent to the same path $e^i v^\ell$ under both equivalence relations implies that $\wt \Lambda = \Lambda_I$ as claimed.

\end{proof}

\begin{example}
Consider again the 2-graph $\Lambda$ associated to the textile system from Example \ref{eq:the example}:   
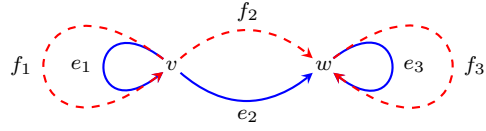
\begin{equation*} 
\begin{array}{l}
\begin{tikzpicture}[yscale=0.75, >=stealth]

\node[inner sep=0.8pt] (u) at (0,0) {$\scriptstyle v$}
edge[blue,loop,->,in=225,out=135,looseness=15,thick] node[auto,black,left] {$\scriptstyle e_1$} (u)
edge[blue,loop,->,in=225,out=135,looseness=30,dashed,red,thick] node[auto,black,left] {$\scriptstyle f_1$} (u);

\node[inner sep=0.8pt] (v) at (2,0) {$\scriptstyle w$}
edge[blue,loop,->,in=-45,out=45,looseness=15,thick] node[auto,black,right] {$\scriptstyle e_3$} (v)
edge[loop,->,in=-45,out=45,looseness=30,dashed,red,thick] node[auto,black,right] {$\scriptstyle f_3$} (v);

\draw[->,dashed,red,thick] (u.north east)
    parabola[parabola height=0.5cm] (v.north west);
\node[inner sep=6pt, above] at (1,0.5) {$\scriptstyle f_2$};

\draw[blue,<-,thick] (v.south west)
	parabola[parabola height=-0.5cm] (u.south east);
\node[inner sep=6pt, below] at (1,-0.5) {$\scriptstyle e_2$};

\end{tikzpicture}
\end{array}
\end{equation*}

\noindent
\color{black}
The commuting squares are 

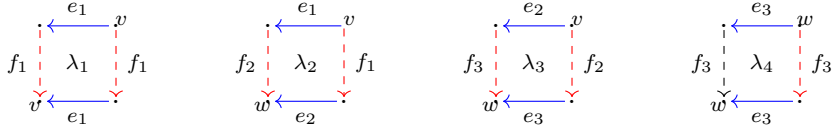
\[
\begin{tikzpicture}

\begin{scope}[xshift=4cm]

\node[inner sep=0.8pt] (00) at (0,0) {$. $};
\node[inner sep=0.8pt] (10) at (1,0) {$ .$}
	edge[->,blue] node[auto,black] {\tiny $e_1$} (00);
\node[inner sep=0.8pt] (01) at (0,1) {$. $}
	edge[->,dashed,red] node[auto,black,swap] {\tiny $f_1$} (00);
\node[inner sep=0.8pt] (11) at (1,1) {$. $}
	edge[->,blue] node[auto,black,swap] {\tiny $e_1$} (01)
	edge[->,dashed,red] node[auto,black] {\tiny $f_1$} (10);
\node[inner sep=0.8pt]  at (0.5,0.5) {$\scriptstyle \lambda_1$};
\node[inner sep=0.8pt]  at (-0.07,-0.07) {$\scriptstyle v$};
\node[inner sep=0.8pt]  at (1.07,1.07) {$\scriptstyle v$};
\end{scope}

\begin{scope}[xshift=7cm]
\node[inner sep=0.8pt] (00) at (0,0) {$. $};
\node[inner sep=0.8pt] (10) at (1,0) {$. $}
	edge[->,blue] node[auto,black] {\tiny $e_2$} (00);
\node[inner sep=0.8pt] (01) at (0,1) {$. $}
	edge[->,dashed,red] node[auto,black,swap] {\tiny $f_2$} (00);
\node[inner sep=0.8pt] (11) at (1,1) {$ $}
	edge[->,blue] node[auto,black,swap] {\tiny $e_1$} (01)
	edge[->,dashed,red] node[auto,black] {\tiny $f_1$} (10);
\node[inner sep=0.8pt]  at (0.5,0.5) {$\scriptstyle \lambda_2$};
\node[inner sep=0.8pt]  at (-0.07,-0.07) {$\scriptstyle w$};
\node[inner sep=0.8pt]  at (1.07,1.07) {$\scriptstyle v$};

\end{scope}
\begin{scope}[xshift=10cm]
\node[inner sep=0.8pt] (00) at (0,0) {$. $};
\node[inner sep=0.8pt] (10) at (1,0) {$. $}
	edge[->,blue] node[auto,black] {\tiny $e_3$} (00);
\node[inner sep=0.8pt] (01) at (0,1) {$. $}
	edge[->,dashed,red] node[auto,black,swap] {\tiny $f_3$} (00);
\node[inner sep=0.8pt] (11) at (1,1) {$. $}
	edge[->,blue] node[auto,black,swap] {\tiny $e_2$} (01)
	edge[->,dashed,red] node[auto,black] {\tiny $f_2$} (10);
\node[inner sep=0.8pt]  at (0.5,0.5) {$\scriptstyle \lambda_3$};
\node[inner sep=0.8pt]  at (-0.07,-0.07) {$\scriptstyle w$};
\node[inner sep=0.8pt]  at (1.07,1.07) {$\scriptstyle v$};
\end{scope}

\begin{scope}[xshift=13cm]
\node[inner sep=0.8pt] (00) at (0,0) {$. $};
\node[inner sep=0.8pt] (10) at (1,0) {$. $}
	edge[->,blue] node[auto,black] {\tiny $e_3$} (00);
\node[inner sep=0.8pt] (01) at (0,1) {$. $}
	edge[->,dashed] node[auto,black,swap] {\tiny $f_3$} (00);
\node[inner sep=0.8pt] (11) at (1,1) {$. $}
	edge[->,blue] node[auto,black,swap] {\tiny $e_3$} (01)
	edge[->,dashed,red] node[auto,black] {\tiny $f_3$} (10);
\node[inner sep=0.8pt]  at (0.5,0.5) {$\scriptstyle \lambda_4$};
\node[inner sep=0.8pt]  at (-0.07,-0.07) {$\scriptstyle w$};
\node[inner sep=0.8pt]  at (1.07,1.07) {$\scriptstyle w$};
\end{scope}
\end{tikzpicture}
\]
\noindent

As we observed in Equation \eqref{eq:2graph-insplit} of Example \ref{eq:the example}, if we define 
\begin{eqnarray*}
    \mathcal{G}_v^1= \{f_1, e_1\}, \, 
    \mathcal{G}_w^1= \{f_2, e_2\}, \, \mathcal{G}_w^2= \{f_3, e_3\},
\end{eqnarray*}
then the partition $ \mathcal G = \{\mathcal{G}_v^1, \mathcal{G}_w^1, \mathcal{G}_w^2\}$ of $\Lambda^{\varepsilon_1} \sqcup \Lambda^{\varepsilon_2}$ satisfies the pairing condition.  Applying the construction of Theorem \ref{thm:Priyanga} to this partition, we see that 
\begin{align*}
    \mathcal{E}_v^1= \{e_1\}, & \quad 
    \mathcal{E}_w^1= \{e_2\},  \quad \mathcal{E}_w^2 = \{e_3\}, \quad \text{and}  \\
    \mathcal{F}_{f_1}^1 & = \{\lambda_1\}, \quad
    \mathcal{F}_{f_2}^1= \{\lambda_2\}, \quad
    \mathcal{F}_{f_3}^1= \{\lambda_3\},  \quad \mathcal{F}_{f_3}^2= \{\lambda_4\}.
\end{align*}

Let $\wt{F}$ be the directed graph insplitting of $F$ at every vertex $f_j\in F^0$ using the partition $\{\mathcal{F}^i_{f_j}\}_{i\leq 2}$, and $\wt{E}$ be the directed graph insplitting of $E$ at every vertex $z\in E^0$ using the partition $\{\mathcal{E}_z^i\}_{i\leq 2}$. 
Since the partition $\{\mathcal{F}^i_{f_j}\}_{i\leq 2}$ is precisely that of Equation \eqref{eq:F-A-partition}, we have   $\wt F = F_A$.  Similarly, $\wt E = E_{\Lambda_I}.$ The definition of $\wt p, \wt q: \wt F \to \wt E$ from Theorem \ref{thm:Priyanga} tells us that the commuting squares of $\wt T$ are 

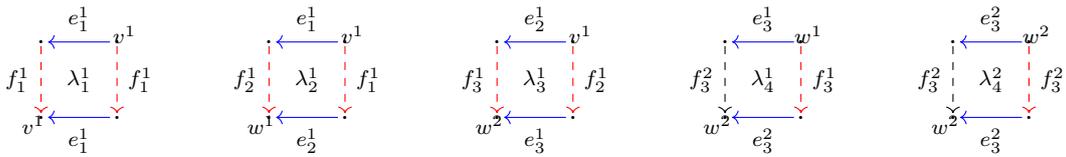
\[
\begin{tikzpicture}

\begin{scope}[xshift=4cm]

\node[inner sep=0.8pt] (00) at (0,0) {$. $};
\node[inner sep=0.8pt] (10) at (1,0) {$. $}
	edge[->,blue] node[auto,black] {\tiny $e_1^1$} (00);
\node[inner sep=0.8pt] (01) at (0,1) {$. $}
	edge[->,dashed,red] node[auto,black,swap] {\tiny $f_1^1$} (00);
\node[inner sep=0.8pt] (11) at (1,1) {$. $}
	edge[->,blue] node[auto,black,swap] {\tiny $e_1^1$} (01)
	edge[->,dashed,red] node[auto,black] {\tiny $f_1^1$} (10);
\node[inner sep=0.8pt]  at (0.5,0.5) {$\scriptstyle \lambda_1^1$};
\node[inner sep=0.8pt]  at (-0.10,-0.10) {$\scriptstyle v^1$};
\node[inner sep=0.8pt]  at (1.10,1.10) {$\scriptstyle v^1$};
\end{scope}

\begin{scope}[xshift=7cm]
\node[inner sep=0.8pt] (00) at (0,0) {$. $};
\node[inner sep=0.8pt] (10) at (1,0) {$. $}
	edge[->,blue] node[auto,black] {\tiny $e_2^1$} (00);
\node[inner sep=0.8pt] (01) at (0,1) {$. $}
	edge[->,dashed,red] node[auto,black,swap] {\tiny $f_2^1$} (00);
\node[inner sep=0.8pt]  (11) at (1,1) {$. $}
	edge[->,blue] node[auto,black,swap] {\tiny $e_1^1$} (01)
	edge[->,dashed,red] node[auto,black] {\tiny $f_1^1$} (10);
\node[inner sep=0.8pt]  at (0.5,0.5) {$\scriptstyle \lambda_2^1$};
\node[inner sep=0.8pt]  at (-0.1,-0.1) {$\scriptstyle w^1$};
\node[inner sep=0.8pt]  at (1.1,1.1) {$\scriptstyle v^1$};

\end{scope}

\begin{scope}[xshift=10cm]
\node[inner sep=0.8pt] (00) at (0,0) {$. $};
\node[inner sep=0.8pt] (10) at (1,0) {$. $}
	edge[->,blue] node[auto,black] {\tiny $e_3^1$} (00);
\node[inner sep=0.8pt] (01) at (0,1) {$. $}
	edge[->,dashed,red] node[auto,black,swap] {\tiny $f_3^1$} (00);
\node[inner sep=0.8pt] (11) at (1,1) {$. $}
	edge[->,blue] node[auto,black,swap] {\tiny $e_2^1$} (01)
	edge[->,dashed,red] node[auto,black] {\tiny $f_2^1$} (10);
\node[inner sep=0.8pt]  at (0.5,0.5) {$\scriptstyle \lambda_3^1$};
\node[inner sep=0.8pt]  at (-0.1,-0.1) {$\scriptstyle w^2$};
\node[inner sep=0.8pt]  at (1.1,1.1) {$\scriptstyle v^1$};
\end{scope}

\begin{scope}[xshift=13cm]
\node[inner sep=0.8pt] (00) at (0,0) {$. $};
\node[inner sep=0.8pt] (10) at (1,0) {$. $}
	edge[->,blue] node[auto,black] {\tiny $e_3^2$} (00);
\node[inner sep=0.8pt] (01) at (0,1) {$. $}
	edge[->,dashed] node[auto,black,swap] {\tiny $f_3^2$} (00);
\node[inner sep=0.8pt] (11) at (1,1) {$. $}
	edge[->,blue] node[auto,black,swap] {\tiny $e_3^1$} (01)
	edge[->,dashed,red] node[auto,black] {\tiny $f_3^1$} (10);
\node[inner sep=0.8pt]  at (0.5,0.5) {$\scriptstyle \lambda_4^1$};
\node[inner sep=0.8pt]  at (-0.1,-0.1) {$\scriptstyle w^2$};
\node[inner sep=0.8pt]  at (1.1,1.1) {$\scriptstyle w^1$};
\end{scope}

\begin{scope}[xshift=16cm]
\node[inner sep=0.8pt] (00) at (0,0) {$. $};
\node[inner sep=0.8pt] (10) at (1,0) {$. $}
	edge[->,blue] node[auto,black] {\tiny $e_3^2$} (00);
\node[inner sep=0.8pt] (01) at (0,1) {$. $}
	edge[->,dashed] node[auto,black,swap] {\tiny $f_3^2$} (00);
\node[inner sep=0.8pt] (11) at (1,1) {$. $}
	edge[->,blue] node[auto,black,swap] {\tiny $e_3^2$} (01)
	edge[->,dashed,red] node[auto,black] {\tiny $f_3^2$} (10);
\node[inner sep=0.8pt]  at (0.5,0.5) {$\scriptstyle \lambda_4^2$};
\node[inner sep=0.8pt]  at (-0.1,-0.1) {$\scriptstyle w^2$};
\node[inner sep=0.8pt]  at (1.1,1.1) {$\scriptstyle w^2$};
\end{scope}

\end{tikzpicture}\]
From the given commuting squares, we see that the textile system $\wt{T}$ is LR.  Moreover,   the associated 2-graph $\Lambda_{\wt{T}}$ is precisely the 2-graph $\Lambda_I$ (described in Equation \eqref{ex:2-graph LambdaI} from Example \ref{eq:the example}) arising from insplitting $\Lambda$ using the partition $\mathcal G.$

\end{example}

\subsection{Which textile insplits yield 2-graph insplits?} \label{Sami sec}

The following Theorem explains how to identify when a Johnson--Madden insplitting of an LR textile system could alternatively be obtained from  a 2-graph insplitting. In other words, while we know from Theorem \ref{thm:JM-insplit-not-LR} that Johnson--Madden insplitting never yields an LR textile system,  Theorem \ref{thm:lr-insplit=2g-insplit} identifies which Johnson--Madden insplittings of an LR textile system can be combined with a directed-graph insplitting of the base graph $E$ to yield an LR textile system. 
Thus, the insplittings described in the theorem below take an LR textile system to an LR textile system.

\begin{theorem}\label{thm:lr-insplit=2g-insplit}
    Let $T = (p,q: F \to  E)$ be an LR textile system  in which $p$ is surjective and $F$ is source-free, and let $\La = \La_T$ be the corresponding 2-graph. Suppose for each $v \in F^0$ we have a partition $\mathcal F_v = \{ \mc{F}_v^1, \dots, \mc{F}_v^{m(v)}\}$ of $vF^1$, and let $\mc{F} = \{\mc{F}_v\}_{v\in F^0}$. Suppose that $\mc{F}$ satisfies 
    \begin{enumerate}
        \item\label{lr-is-hyp-1} If $v, w \in F^0$ and $p (\mathcal F^i_v) \cap p(\mathcal F^j_w) \not= \emptyset$, then $p(\mathcal F^i_v) = p(\mathcal F^j_w);$ and
        \item\label{lr-is-hyp-2} For all $v \in F^0$, there exists $w \in F^0$ and $1 \leq j \leq m(w)$ such that $q(v F^1) \subseteq p(\mathcal F^j_w)$.
    \end{enumerate}
    Let $T_{JM}= (p_{JM}, q_{JM}:F_{JM} \to E_{JM})$ be the result of performing Johnson--Madden insplitting with respect to the partition $\mathcal F$. Then there is a directed graph insplit $\widetilde E$ of $E_{JM}$, together with graph homomorphisms $\wt p, \wt q: F_{JM} \to \widetilde E$, so that the textile system $\widetilde T := (\wt{p}, \wt{q}: F_{JM}\to \widetilde E)$ is LR. 
    
    Moreover, $\widetilde T$ can be identified as the textile system $T_{\La_I}$ of 
    a 
    2-graph $\La_I$ which arises from 2-graph insplitting on $\La$ 
   
    using a partition $\mathcal G$ (described in Proposition  \ref{lr-insplit-3} below) of $\Lambda^{1}$ which arises from $\mathcal F.$
\end{theorem}

Theorem \ref{thm:lr-insplit=2g-insplit} will be proved in a series of propositions.

\begin{proposition}\label{prop:lr-ptn}
    Let $T=(p,q:F\to E)$ be an LR textile system, and for each $v \in F^0$, $\mc{F}_v = \{ \mc{F}_v^i : i=1, \dots, m(v)\}$ be a partition of $vF^1$. Then:
    \begin{enumerate}
        \item\label{lr-ptn-1} if $v \in F^0$ and $p(\mc{F}_v^i) \cap p(\mc{F}_v^j) \neq \emptyset$, then $i=j$, and
        \item\label{lr-ptn-2} if $v, w \in F^0$ and $p(v) = p(w)$, then for every $i \in \{1, \dots, m(v)\}$, there exists $j \in \{1, \dots, m(w)\}$ such that $p(\mc{F}_v^i) \cap p(\mc{F}_w^j) \neq \emptyset$.
    \end{enumerate}
\end{proposition}
\begin{proof}
    (\ref{lr-ptn-1}): Let $e \in p(\mc{F}_v^i) \cap p(\mc{F}_v^j)$. Then there are $\lambda \in \mc{F}_v^i$ and $\mu \in \mc{F}_v^j$ such that $p(\la) = p(\mu) = e$. 
    Because $p$ is a graph homomorphism, $p(r(\lambda)) = r(p(\lambda)) = r(e) = p(r(\mu))$, so by unique $r$-path lifting of $p$ (since $T$ is LR), $\la = \mu$. Thus, $\mc{F}_v^i \cap \mc{F}_v^j \not= \emptyset$. Since $\{\mc{F}_v^k : k=1, \dots, m(v)\}$ is a partition of $vF^1$, this implies $i=j$.  
    
    (\ref{lr-ptn-2}): Let $e \in p(\mc{F}_v^i)$. Then $r(e) = p(v) = p(w)$, 
    so by unique $r$-path lifting for $p$, there exists a unique $\eta \in F^1$ such that $r(\eta) = w$ and $p(\eta) = e$. 
    We have $\eta \in \mc{F}_w^j$ for some $j \in \{1, \dots, m(w)\}$. Hence, $e \in p(\mc{F}_v^i) \cap p(\mc{F}_w^j)$.
\end{proof}

\begin{proposition}\label{prop:lr-insplit-1'}
    Let $T=(p,q:F\to E)$ be an LR textile system with a partition $\mc{F}$ of $F^1$ satisfying condition (\ref{lr-is-hyp-1}) of Theorem \ref{thm:lr-insplit=2g-insplit}. If $v, w \in F^0$ with $p(v) = p(w)$, then $m(v) = m(w)$.
\end{proposition}

\begin{proof}
    By Proposition \ref{prop:lr-ptn}(\ref{lr-ptn-2}) and condition (\ref{lr-is-hyp-1}) of Theorem \ref{thm:lr-insplit=2g-insplit}, if $p(v) = p(w)$, then for each $i \in \{1, \dots, m(v)\}$, there exists $j_i \in \{1, \dots, m(w)\}$ such that $p(\mc{F}_v^i) = p(\mc{F}_w^{j_i})$. If $i, k \in \{1, \dots, m(v)\}$ and $J:= j_i = j_k$, then $p(\mc{F}_v^i) = p(\mc{F}_w^J) = p(\mc{F}_v^k)$, so $i = k$, hence $i \mapsto j_i$ is injective. The map is surjective because for each $k \in \{1, \dots, m(w)\}$, there is $i_k \in \{1, \dots, m(v)\}$ such that $p(\mc{F}_w^k) = p(\mc{F}_v^{i_k})$ (again, by combining Proposition \ref{prop:lr-ptn}(\ref{lr-ptn-2}) and condition (\ref{lr-is-hyp-1}) of Theorem \ref{thm:lr-insplit=2g-insplit}) so $k = j_{i_k}$.
    Hence, $m(v) = m(w)$.
\end{proof}

\begin{remark}
We can thus assume without loss of generality that for each $z \in E^0$, for all $v, w \in p\inv(z)$ and $i=1,\dots,m(v)$, $p(\mc{F}_v^i) = p(\mc{F}_w^i)$, and we are justified in setting $m(z) = m(v)$.
\label{rmk:same-index}
\end{remark}

\begin{proposition}\label{lr-insplit-2} 
    Under the assumptions of Theorem \ref{thm:lr-insplit=2g-insplit}, the collection $$\mc{E}_z =\{ \mathcal E^{i}_z := p(\mathcal F^i_v) \mid p(v)=z, i=1, \dots, m(z) \}$$ 
    is a partition of $zE^1$ for each $z \in E^0$. If $\widetilde E$ is the insplit of $E$ with respect to this partition,
    then the maps $\wt p, \wt q : F_{JM} \to \widetilde E$ defined, 
    {for $v \in F^0,$ $\la \in F^1$, by 
    \[\wt p(\lambda^i) = p(\lambda)^{i} , \qquad \wt p(v^i) = p(v)^{i},\]
    and, if $s_F(\la) = v$ and $q(vF^1) \sse p(\mc{F}_w^j),$ 
    \[\wt q(\lambda^i)= q(\lambda)^{j}, \qquad  \wt q(v^i) = q(v)^{j},\]}
     
    are graph homomorphisms, and so $\widetilde T := (\wt p, \wt q: F_{JM} \to \widetilde E)$ is a textile system.
    
\end{proposition}

\begin{proof}
To ease the burden of notation in this proof, we abuse notation and use the same symbols $r,s$ for the range and source maps in every graph which appears in the proof.  We trust that the context will suffice to indicate the domain and range of each occurrence of $r, s$.

    Since $p$ is surjective, every $e \in zE^1$ is in some $\mathcal E_z^{i} = p(\mc{F}_v^i)$. Moreover, Remark \ref{rmk:same-index} and Condition \eqref{lr-is-hyp-1} of Theorem \ref{thm:lr-insplit=2g-insplit} indicate that $\mathcal E^i_z \cap \mathcal E^j_z= \emptyset$ if $i\not= j$, so $\mathcal E_z$ is indeed a partition of $z E^1.$ 
    
To see that $\wt p$ is well defined, suppose that $e = p(\lambda) = p(\mu)$.  Then $s(e) = s(p(\lambda)) = p(s(\lambda)) = p(s(\mu))$, so  Remark \ref{rmk:same-index} implies that insplitting $E$ and $F$ creates  the same number of ``copies'' of $e, \lambda,$ and $\mu$.  It follows that  
 
 $\wt{p}$ is well-defined as claimed. To see the same for $\wt{q}$, we suppose $q(vF^1) \sse p(\mc{F}_w^j) \cap 
p(\mc{F}_u^k)$. Since $F$ is assumed to be source-free, $v F^1 \not= \emptyset$ and consequently $q(vF^1) \neq \emptyset$. 
    We then have $p(\mc{F}_w^j) \cap  p(\mc{F}_u^k) \neq \emptyset$, and so Condition \eqref{lr-is-hyp-1} of Theorem \ref{thm:lr-insplit=2g-insplit} yields $p(\mc{F}_w^j) =  p(\mc{F}_u^k)$. Remark \ref{rmk:same-index} implies $j = k$, 
    and hence $\wt q$ is well-defined.

    Now we check that $\wt{p}$ and $\wt{q}$ are graph homomorphisms. First, suppose $\la \in \mc{F}_u^k$, with $s(\la) = v$, $q(vF^1) \sse p(\mc{F}_w^j)$, and {$q(uF^1) \sse p(\mc{F}_{\hat u}^n)$}. Let $e = p(\la)$, $f = q(\la)$, $z = p(u),$ $x = q(u) = p(\hat u),$ $y = p(v),$ and $t = q(v) = p(w)$. 
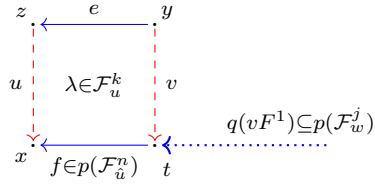
    \[ 
    \begin{tikzpicture}[scale=1.6]
       
\node[inner sep=0.8pt] (00) at (0,0) {$. $};
\node[inner sep=0.8pt] (10) at (1,0) {$. $}
	edge[->,blue] node[auto,black] {$\scriptstyle  f \in p(\mathcal F^n_{\hat u})$} (00);

 \node at (2.5,0) {}
  edge [->,thick, dotted, blue] (10);
\node[inner sep=0.8pt] (01) at (0,1) {$. $}
	edge[->,dashed, red] node[auto,black,swap] {\tiny $u$} (00);
\node[inner sep=0.8pt] (11) at (1,1) {$. $}
	edge[->,blue] node[auto,black,swap] {\tiny $e$} (01)
	edge[->,dashed,red] node[auto,black] {\tiny $v$} (10);
\node[inner sep=0.8pt]  at (0.5,0.5) {$\scriptstyle \lambda \in \mathcal F^k_u$};
\node[inner sep=0.8pt]  at (-0.1,-0.1) {$\scriptstyle x$};
\node[inner sep=0.8pt] at (-0.1, 1.1) {$\scriptstyle z$};
\node[inner sep=0.8pt]  at (1.1,1.1) {$\scriptstyle y$};
\node[inner sep=0.8pt] at (1.1,-0.2) {$\scriptstyle t$};
\node at (2.2, 0.2) {$\scriptstyle q(vF^1) \subseteq p(\mathcal F^j_w)$};


    \end{tikzpicture}
    \]
    Since $p, q$ are graph homomorphisms, 
    \[r(e) = r(p(\la)) = p(r(\la)) = p(u) = z,\] and similarly, $s(e) = y$, $r(f) = x$, and $s(f) = t$.
    Now, 
    \begin{align*}
        \wt{p}(s(\la^i)) = \wt{p}(s(\la)^i) = \wt{p}(v^i) = p(v)^i = y^i, \quad \text{ and } \quad s(\wt{p}(\la^i)) = s(p(\la)^i) = s(e^i) = s(e)^i = y^i,
    \end{align*}
    so $\wt{p}(s(\la^i)) = y^i = s(\wt{p}(\la^i))$.
   
    Next, since $\lambda \in \mathcal F^k_u$, we have 
$   
        \wt{p}(r(\la^i)) = \wt{p}(u^k) 
        = p(u)^{k} 
        = z^{k}. 
        $
On the other hand,  since $e=p(\la) \in p(\mc{F}_u^k) = \mc{E}_z^k$, 
$
        r(\wt{p}(\la^i)) = r(p(\la)^i)= r(e^i)= z^k. 
$
  That is, $r(\wt{p}(\la^i)) = z^{k}= \wt{p}(r(\la^i))$, and we conclude that $\wt{p}$ is a graph homomorphism.

Now we check  that $\wt{q}$ is a graph homomorphism. First, since $\lambda \in \mathcal F^k_u$ and $q(uF^1) \subseteq p(\mathcal F^n_{\hat u}),$
\[ \wt q(r(\lambda^i)) = \wt q(u^k) = q(u)^n = x^n,\]

while the facts that $s(\lambda) = v$ and $q(v F^1) \subseteq p(\mathcal F^j_w)$ imply 
$ r(\wt q(\lambda^i)) = r(q(\lambda)^j) = r(f^j) = x^n,$

since 
$f \in q(u F^1) \sse p(\mc{F}_{\hat u}^n) = \mc{E}_x^n$. So $\wt{q}(r(\la^i)) = x^n = r(\wt{q}(\la^i))$.

Finally, since $q(vF^1) \sse p(\mc{F}_w^j)$, we have  $  \wt{q}(s(\la^i)) 
= \wt{q}(v^i) =  q(v)^{j} 
    = t^{j}, $
and
\begin{align*}
    s(\wt{q}(\la^i)) & = s(q(\la)^{j}) 
    = s(f^j) = s(f)^j = t^j.
\end{align*}
Thus $\wt{q}(s(\la^i)) = t^j = s(\wt{q}(\la^i))$, and $\wt q$ is also a graph homomorphism, as claimed.
\end{proof}

We have the following observation, which will be needed for the next proof. 

\begin{lemma}\label{lem:p-and-q}
    Under the conditions of Theorem \ref{thm:lr-insplit=2g-insplit}, if $u, v \in F^0$ and $\mu \in vF^1$ such that $q(\mu) \in p(\mc{F}_u^i)$, then $q(vF^1) \sse p(\mc{F}_u^i)$.
\end{lemma}
\begin{proof}
    By Proposition \ref{prop:lr-ptn}(\ref{lr-ptn-2}), there exist $w \in F^0$ and $j \in \{1, \dots, m(w)\}$ such that $q(vF^1) \sse p(\mc{F}_w^j)$. So $q(\mu) \in p(\mc{F}_w^j) \cap p(\mc{F}_u^i)$. Hence, by condition (\ref{lr-is-hyp-1}) in Proposition \ref{prop:lr-ptn}, $p(\mc{F}_w^j) = p(\mc{F}_u^i)$.
\end{proof}

\begin{proposition}\label{lr-insplit-3}
Under the assumptions of Theorem \ref{thm:lr-insplit=2g-insplit}, let $\La = \La_T$ be the 2-graph of the LR textile system $T$, and for each $z \in \La^0 = E^0$ and each $i \in \{1, \dots, m(z)\}$, set
\[{{\mc{G}_z^{i} = r_F(q\inv(\mc{E}_z^{i})) \sqcup \mc{E}_z^{i}}}.\]
Then the collection ${\mc{G} = \{\mc{G}_z^{i} : z \in \La^0, i=1, \dots, m(z)\}}$ is a partition of $\La^{\varepsilon_1} \sqcup \Lambda^{\varepsilon_2}$ which satisfies the pairing condition for 2-graph insplitting. 
\end{proposition}

\begin{proof}
   
    First we show that $\mc{G}$ is a partition of $\La^{\varepsilon_1}  \sqcup \Lambda^{\varepsilon_2} (=  E^1 \sqcup F^0)$. By 
    Proposition \ref{lr-insplit-2} the collection $\mc{E} = \{\mc{E}_z^i : z \in E^0 = \La^0, i=1, \dots, m(z)\}$ is a partition of $E^1$. Hence it suffices to show that the collection $\mc{C} = \{ F^0 \cap \mc{G}_z^i = r_F(q^{-1}(\mc{E}_z^i)) : z \in \La^0, i=1,\dots,m(z)\}$ is a partition of $F^0$.
    
    Suppose $v \in F^0$. Since $F$ is source-free, $v=r(\la)$ for some $\la \in F^1$. 
    Hence $q(\la) \in \mc{E}_z^{i}$ for some $z \in E^0, i\in \{1, \dots, m(z)\}$, and therefore $v = r(\lambda) \in r_F(q^{-1}(\mc{E}_z^{i})) $ lies in some set in $\mc C$. 
    To see that the sets in $\mc{C}$ are pairwise disjoint, observe first that the hypotheses of Theorem \ref{thm:lr-insplit=2g-insplit} guarantee that  $q(vF^1) \subseteq  p(\mathcal F^j_w) $ for 
    a unique set $p(\mathcal F^j_w).$
    Thus,    
    fix $v \in F^0$, and suppose that there exist $z, x \in \La^0 = E^0$ such that $v \in 
    r_F(q^{-1}(\mc{E}_z^i)) \cap r_F(q^{-1}(\mc{E}_x^j))$. That is, there  exist $\mu, \nu \in F^1$ such that $v = r(\mu) = r(\nu)$, $q(\mu) \in \mc{E}_z^{i} = p(\mathcal F^i_w)$ for some $w \in F^0$ with $p(w) = z$, and $q(\nu) \in \mc{E}_x^{j} = p(\mathcal F^j_u)$ for some $u \in F^0$ with $p(u) =x$. 
  As $\mu, \nu \in vF^1$, Lemma \ref{lem:p-and-q} forces $p(\mathcal F^j_u) = p(\mathcal F^i_w)$, and hence $z=x$ and $i=j$.  In other words, $ r_F(q^{-1}(\mc{E}_z^i)) \cap r(q^{-1}(\mc{E}_x^j))\not= \emptyset$ implies $z=x$ and $i=j$, so $\mathcal C$ is indeed a partition of $F^0$ as claimed. 

    For the partition $\mathcal G$ to satisfy the pairing condition, we must have that for  all $\la \in F^1$,
     $q(\la) \in \mc{G}_z^i$ if and only if $r_F(\la) \in \mc{G}_z^i$. To see this, 
    fix $\la \in F^1$, and suppose $q(\la) = e \in  E^1 \cap \mc{G}_z^{i} = \mc{E}_z^{i}$. Then 
   $r_F(\la) \in r_F(q\inv(\mc{E}_z^i)) \sse \mc{G}_z^i$, as desired. 
    Now suppose 
    $r_F(\la) = v \in F^0 \cap \mc{G}_z^{i} = 
    r_F(q\inv(\mc{E}_z^{i}))$. Thus there exists $\mu \in vF^1$ such that $q(\mu) \in \mc{E}_{z}^{i} = p(\mc{F}_u^i)$ for some $u \in F^0$ with $p(u) = z$. 
    Since $q(\mu) \in p(\mc{F}_u^i), $  Lemma \ref{lem:p-and-q}
    implies that $p(\mathcal F^i_u) \supseteq q(vF^1) \ni q(\lambda).$ As $\mathcal E^i_z = p(\mathcal F^i_u)$, we have $q(\lambda) \in \mathcal E^i_z \subseteq \mathcal G^i_z$ whenever $r_F(\lambda) \in \mathcal G^i_z.$ We conclude that  
    $\mathcal G$ satisfies the pairing condition.
   
\end{proof}

Thus, we can perform 2-graph insplitting on $\Lambda_T$ with respect to the partition $\mathcal G$.  Recall from Definitions \ref{dfn:2colgraph from T} and \ref{def:2-graph-insplit}  that the resulting 2-graph $\Lambda_I$ has $\Lambda_I^0 = \{ z^i: z \in E^0, 1 \leq i \leq m(z)\}$ and 
\[ \Lambda_I^{\varepsilon_1} = \{ e^i: e \in E^1, 1 \leq i \leq m(s(e))\}, \quad \Lambda_I^{\varepsilon_2} = \{ v^i: v \in F^0, 1 \leq i \leq m(p(v))\}, \]
with $s_I(e^i) = s(e)^i, s_I(v^i) = p(v)^i, r_I(e^i) = r(e)^j$ if $e \in \mathcal E^j_z$, and $r_I(v^i) = q(v)^k$ if $v \in r_F(q^{-1}(\mathcal E^k_z)),$ that is, if $q(vF^1) \subseteq p(\mathcal F^k_u)$ for some $u$ with $p(u) = q(v) = z.$  We have 
\[ v^j e^i \sim_I f^\ell w^k \iff i = k, e \in \mathcal E^j_{r(e)}, w \in r_F(q^{-1}(\mathcal E^\ell_{s(f)})), \text{ and } ve \sim f w \in \Lambda_T. \]
That is,  $v^j e^i \sim_I f^\ell w^i$ if and only if
\begin{itemize} 
\item there is $\lambda \in F^1$ with $r_F(\lambda) = v, s_F(\lambda) = w, p(\lambda) = e, q(\lambda) = f;$ 
\item we have $q(wF^1) \subseteq p(\mathcal F^\ell_{u})$ for some $u \in F^0$ with $p(u) = q(w)$;
\item and  $\lambda \in \mathcal F^j_v$ (so that $e = p(\lambda) \in \mathcal E^j_{r(e)} = \mc{E}^j_{p(v)}$).
\end{itemize}
In other words, each $\lambda = ve \sim_\Lambda f w \in F^1$ yields commuting squares $\{ \lambda^i:= v^j e^i \sim_I f^\ell w^i : i=1, \dots, m(s(e)) \}$  in $\Lambda_I$ (note that $m(s(e)) = m(p(w))$); the indices $j, \ell$ are determined by $\lambda$ and are the same for each $\lambda^i$.

\begin{theorem}\label{lr-insplit-4}
Under the assumptions of Theorem \ref{thm:lr-insplit=2g-insplit} above, let $\La_I$ be the 2-graph insplitting (see ~Definitions \ref{def:2-graph-insplit}) of $\La_T$ with respect to the partition {$ \{\mc{G}_z^i: z\in \Lambda^0, i= 1, \ldots, m(z)\}$} from Proposition \ref{lr-insplit-3}, and  let $T_I = (p_I, q_I: F_I \to E_I)$ be the textile system built from $\La_I$ as in  \eqref{eq:textile-from-2graph}. Then $T_I \cong \widetilde T = (\wt{p}, \wt{q}: \widetilde{F} := F_{JM} \to \widetilde{E})$, the textile system given in Proposition \ref{lr-insplit-2}. \end{theorem}
\begin{proof}
By  \eqref{eq:textile-from-2graph} and Remark \ref{rmk:same-index}, the  LR textile system $T_{I} = (p_I, q_I: F_{I} \to E_{I})$ has 
\[ F_I^0 = \{ v^i: v \in F^0, 1 \leq i \leq m (p(v))\}, \ F_I^1 = \{ \lambda^i: \lambda \in F^1, 1 \leq i \leq m(s(p(\lambda)))\}, \]
$s_{F_1}(\lambda^i) = s_F(\lambda)^i$, and $ \ r_{F_I}(\lambda^i) = r_F(\lambda)^j \text{ if } \lambda \in \mathcal F^j_v.$
When we compare $F_I$ with the construction of $F_{JM}$ as in Definition \ref{dfn:textile-insplit}, the fact that $m(p(v)) = m(v)$ for all $v \in F^0$ implies that $F_I = F_{JM}.$ Similarly, we observe that the definitions of $\widetilde E, \wt p, \wt q$ in Proposition \ref{lr-insplit-2} exactly match the definitions of $E_I, p_I, q_I$ from Equation \eqref{eq:textile-from-2graph}. 
The fact that $\widetilde E = E_I$ follows from the fact that the insplitting partition $\mathcal E$ used to create $\widetilde E$ satisfies $\mathcal E^i_z = \Lambda^{\varepsilon_1}_I \cap \mathcal G^i_z ,$ and the source and range maps in $\widetilde E$ coincide with the 2-graph insplitting source and range maps. Moreover,  $\wt p = p_I$ and $\wt q = q_I$, since for any $\lambda^i = r_F(\lambda)^k p(\lambda)^i \sim_I q(\lambda)^\ell  s_F(\lambda)^i \in F_I^1$ and any $v^i \in F_I^0 = \Lambda_I^{\varepsilon_2},$

\[ p_I(\lambda^i) = p(\lambda)^i, \quad p_I(v^i) = s_{\Lambda_I}(v^i) = p(v)^i, \quad q_I(v^i) = r_{\Lambda_I}(v^i) = q(v)^k \]
if $v \in \mc{G}^k_{q(v)}$, or equivalently, $v \in r_F(q^{-1}(\mathcal E^k_{q(v)}).$  That is, $q_I(v^i) = \widetilde q(v^i).$
Finally, 
to compare $q_I(\lambda^i)$ and $\wt q(\lambda^i)$, recall that  $\lambda^i = r_F(\lambda)^k p(\lambda)^i \sim_I q(\lambda)^\ell  s_F(\lambda)^i \in F_I^1$ where $p(\lambda) \in \mc E^k_{p(r_F(\lambda)))}$ and $s_F(\lambda) \in \mc G^\ell_{q(s_F(\lambda))}$. By construction, then, $q_I(\lambda^i) = q(\lambda)^\ell$.  On the other hand, $s_F(\lambda) \in \mc G^\ell_{q(s_F(\lambda))}$ implies that $s_F(\lambda) \in r_F(q^{-1}(\mc E^\ell_z))$, so (thanks to Lemma \ref{lem:p-and-q}) $q(s_F(\lambda) F^1) \subseteq p(\mathcal F^\ell_w)$ for some $w \in F^0$.
Consequently, $\wt q(\lambda^i) = q(\lambda)^\ell = q_I(\lambda^i)$ as desired. 

   In other words, $T_I = \widetilde T;$ since $T_I$ is LR, being a textile system arising from a 2-graph, it follows that $\widetilde T$ is also LR.
\end{proof}
The above Theorem completes the proof of Theorem \ref{thm:lr-insplit=2g-insplit}.

\begin{example}
Consider again the LR textile system $T= (p, q: F\to E)$ in Example~\ref{eq:the example} and its corresponding $2$-graph $\Lambda$, pictured in \eqref{ex:6.1 EF+}. Note that for each $f_k\in F^0$, $k= 1, 2, 3$, the  partition $\mc{F}_{f_k}= \{\mc{F}^i_{f_k}: i \leq 2\}$ 
of Equation \eqref{eq:F-A-partition}

satisfies  the conditions given in Theorem~\ref{thm:lr-insplit=2g-insplit}.
That is, the graph $F_A$ of  Example \ref{eq:the example} is the graph $F_{JM}$ of Theorem \ref{thm:lr-insplit=2g-insplit}. 

Observe that 
\[ q(f_1 F^1) = \{ e_1\} = p(\mathcal F^1_{f_1}) = p(\mathcal F^1_{f_2}); \quad q(f_2F^1) = \{ e_2\} = p(\mathcal F^1_{f_3}); \text{ and} \quad q(f_3 F^1) = \{ e_3\} = p(\mathcal F^2_{f_3}).\]
As $p(f_1) = p(f_2) = v$ and $q(f_3) = w,$ we have 
\[ \mathcal E^1_{v} = p(\mathcal F^1_{f_1}) = \{e_1\} = p(\mathcal F^1_{f_2}), \qquad  \mathcal E^1_{w} = p(\mathcal F^1_{f_3}) = \{e_2\}, \quad \text{ and } \quad \mathcal E^2_w = \{ e_3\}.\]
It follows that 

the directed graph insplitting $\widetilde E$ of $E$ with this partition is precisely the directed graph $E_{\Lambda_I}$ shown in Example \ref{eq:the example}.

Thus, applying the formulas from Proposition \ref{lr-insplit-2} in this case, we obtain maps $\wt p, \wt q: F_{JM} = F_A \to \widetilde E = E_{\Lambda_I}$ which are given by 
\[ \wt q(\lambda_1^1) = e_1^1, \quad \wt q(\lambda_2^1) = e_2^1, \quad \wt q(\lambda_3^1) = e_3^1, \quad \wt q(\lambda_4^1) = e_3^2 = \wt q(\lambda_4^2),\]
and $\wt p(\lambda_j^i) = p(\lambda_j)^i$ for all $i, j$.
These formulas agree with the commuting squares of $\Lambda_I$, as asserted by Theorem \ref{lr-insplit-4}.  Indeed, note that 
\[ q^{-1}(\mathcal E^1_v) = \{ \lambda_1\}, \quad q^{-1}( \mathcal E^1_w) = \{ \lambda_2\}, \quad q^{-1}(\mathcal E^2_w) = \{ \lambda_3, \lambda_4\},\]
so 
the partition $\mathcal G$ of Proposition \ref{lr-insplit-3} is given by 
\[ \mathcal G^1_v = \{ f_1, e_1\}, \quad \mathcal G^1_w = \{ f_2, e_2\}, \quad \mathcal G^2_w = \{ f_3, e_3\}.\]
In other words, the 2-graph insplit induced by the partition $\mathcal F$ of Equation \eqref{eq:F-A-partition} is precisely the initial 2-graph insplit of Equation \eqref{eq:2graph-insplit}.

\end{example}

\subsection{Textile and 2-graph insplits from \textit{E}-insplits}

\noindent Although Johnson--Madden insplitting focuses on the top graph $F$ of a textile system, to understand the connection between Johnson--Madden insplitting and 2-graph insplitting, it is perhaps more natural to start by insplitting the bottom graph $E$, as we now explain. This leads us to an alternate perspective to Theorem \ref{thm:lr-insplit=2g-insplit}. 

Let $T = (p,q:F\to E)$ be an LR textile system and let $\Lambda_T$ denote the 2-graph associated to $T$. The following Theorem shows that by starting with a  suitable partition of 
$E$, one can construct an LR textile system $\widetilde{T}$ 
which yields a 2-graph $\Lambda_I$ coinciding with the 2-graph insplitting of $\Lambda_T$.

\begin{theorem} \label{Thm: Main result III}
Let $T = (p, q: F \to E)$ be an LR textile system with associated 2-graph $\Lambda$. 
Assume $p$ is surjective and $F$ is source free.
Suppose that, 
for each $z \in E^0$, 
we have a partition  $\{ \E^1_z, \E^2_z, \ldots, \E^{m(z)}_z\}$ of $zE^1$ such that 
for each $u \in F^0$, there exists  $ z \in E^0$ and  $j \in \{ 1, 2, \ldots, m(z)\}$ such that $q(uF^1) \subseteq \E^j_z$.

We construct a partition $\{ \mathcal F^i_v: v \in F^0, 1\leq i \leq m(p(v))\}$ of $F^1$ and a partition $\{ \G^i_z: z \in \Lambda^0 = E^0, 1 \leq i \leq m(z)\}$ of $\Lambda^{1}$  as follows. For each  $v \in F^0$ and $ z \in E^0$, define
\begin{equation}
\mathcal{F}^i_v := vF^1 \cap p^{-1}(\E^i_{p(v)}), \;\;\; \text{ for } 1\le i\le m(p(v)),
\end{equation}
\begin{equation}
\mathcal{G}^i_z := \E^i_z \sqcup r_F(q^{-1}(\E_z^i)), \;\;\; \text{ for } 1\le i\le m(z).
\end{equation}

Let $\widetilde{F}$ denote the directed graph resulting from directed-graph insplitting  $F$ using the partition $\{\mathcal{F}^i_v\}_{i,v}$.
Let $\widetilde{E}$ denote the directed graph resulting from directed-graph insplitting  $E$ using the partition $\{\mathcal{E}^i_z\}_{i,z}$.
Define maps $\tilde{p}, \tilde{q}: \widetilde{F} \to \widetilde{E}$ by 
\begin{eqnarray*}
\tilde{p}(v^i) :=  p(v)^i, &  \tilde{q}(v^i)  :=  q(v)^j, \\
 \tilde{p}(\lambda^i) :=  p(\lambda)^i, & \tilde{q}(\lambda^i) :=  q(\lambda)^j, 
\end{eqnarray*}
where $\lambda \in F^1$, $v \in F^0$, $s(\lambda) = v$ and $j$ 
satisfies $q(vF^1) \subseteq \E^j_{q(v)}$. Then 
\begin{enumerate}
\item The maps $\tilde{p}, \tilde{q}:  \widetilde{F}  \to  \widetilde{E}$ are 
graph homomorphisms and $\widetilde{T} := (\tilde{p}, \tilde{q} : \wt{F} \to \wt{E} )$ is an LR textile system.
\item The collection $\{ \mathcal{G}^i_z := \E^i_z \sqcup r(q^{-1}(\E_z^i)) : z \in E^0, 1\le i \le m(z) \}$ satisfies the pairing condition. 
\item The 2-graph associated to $\widetilde{T}$ (denoted by $\widetilde{\Lambda}$) is identical to the 2-graph ${\Lambda}_I$ resulting from insplitting $\Lambda$ using the partition 
$\{ \G^j_z\}_{j,z}.$ 
\end{enumerate}

\end{theorem}

\begin{proof}
We begin by observing  that if $q(uF^1) \subseteq \mathcal E^j_z$, then  $z = q(u)$ and $j$ is unique, since $q$ is a graph homomorphism and  $\{ \mathcal E^j_z\}_j$ is a partition of $z E^1$.  Thus, $\tilde q$ is well defined.

Next, we show that $\{ \mathcal F_v^i: v \in F^0, 1 \leq i \leq m(p(v)) \}$ is a well-defined partition of $F^1$.  First, note that if $v \in F^0$ and $p(v) = z$, then 
\begin{equation}
vF^1 \cap p^{-1}(\E^i_z) \ne \emptyset \text{ for all } i \in \{ 1, 2, \ldots, m(z)\}.
\end{equation}
This is because there exists some $e \in \mathcal{E}^i_z$ (partition sets are nonempty by convention), and since $p(v) = z = r(e)$, unique $r$-path lifting  of $p$ implies that there exists a unique $\lambda \in F^1$ such that $r_F(\lambda) = v$ and $p(\lambda) = e$. Hence, $\lambda \in vF^1 \cap p^{-1}(\E^i_z)$.
It follows that $\mathcal{F}^i_v \ne \emptyset$ for all $v \in F^0$ and all $ 1 \le i \le m(p(v))$.
Consequently,  $\{ \mathcal F_v^i: v \in F^0, 1 \leq i \leq m(p(v)) \}$ is a well-defined partition of $F^1$: 
\[ \displaystyle \bigcup_i \mathcal{F}^i_v = vF^1 \cap p^{-1} (\bigcup_i \mathcal{E}^i_{p(v)}) = vF^1 \cap p^{-1}(p(v) E^1) = vF^1,\]
and
$ \F^k_u \cap \F^l_v \ne \emptyset \implies \left( v = u \text{ and } \E^k_{p(u)} \cap \E^l_{p(v)} \ne \emptyset\right)  \implies v= u \text{ and } k=l.$

In fact, the partition $\{ \mathcal F^i_v\}_{i,v}$ satisfies Hypotheses (1) and (2) of Theorem \ref{thm:lr-insplit=2g-insplit}.  To see this, first observe  that the sets $\mathcal E^i_z$ and $\mathcal F^i_w$  of the current Theorem satisfy $\mathcal E^i_z = p(\mathcal F^i_v)$ for any $v$ with $p(v) = z$.  (The inclusion $\supseteq$ is immediate; to see equality, choose $e \in \mathcal E^i_z$ and $v \in F^0$ with $p(v) = z$.  Then $r$-path lifting yields $f \in F^1$ with $r(f) = v$ and $p(f) = e$, that is, $f \in p^{-1} (\mathcal E^i_{p(v)}) \cap v F^1 = \mathcal F^i_v$ and $e = p(f)$.)  Thus, (1) holds since $\{ \mathcal E^i_z\}_{i,z}$ is a partition of $E^1$, and (2) holds by our hypothesis on the partition $\{ \mathcal E^i_z\}_{i,z}$.

Consequently, the graph $\wt E$ of the current theorem is the same as the graph $\wt E$ of Theorem \ref{thm:lr-insplit=2g-insplit}, and our $\wt F$ is the graph $F_{JM}$ of Theorem \ref{thm:lr-insplit=2g-insplit}.  Moreover, the maps $\tilde{p}, \tilde{q}:  \wt{F}  \to  \wt{E}$ defined in the statement of the current theorem are precisely the maps discussed in Proposition \ref{lr-insplit-2}.  Hence, that proposition establishes that $\tilde{p}, \tilde{q}:  \wt{F}  \to  \wt{E}$  are well-defined graph homomorphisms.  The fact that $\wt{T} := (\tilde{p}, \tilde{q} : \wt{F} \to \wt{E} )$ is an LR textile system follows from the identification $\wt T = T_I$ of Theorem \ref{lr-insplit-4}.

Finally, observe that the sets $\mathcal G^i_z$ of the current theorem are precisely the same as those of Proposition \ref{lr-insplit-3}, so that proposition yields statement (2) of the theorem; and the proof of Theorem \ref{lr-insplit-4}  yields statement (3).
 \qedhere 

\end{proof}

\subsection{Uniting the approaches} \label{subsec: unification}

As mentioned at the beginning of this Section, the relationships between 2-graph insplitting and Johnson--Madden insplitting established in Theorems \ref{thm:Priyanga}, \ref{thm:lr-insplit=2g-insplit}, and \ref{Thm: Main result III} are compatible.  That is,  the coherent choice of notation in Theorems \ref{thm:Priyanga}, \ref{thm:lr-insplit=2g-insplit}, and \ref{Thm: Main result III} was neither accidental nor an abuse of notation, because all the constructions in these results are equivalent.
To obtain a clean statement of the equivalences, we assume throughout this section that the top graph $F$ is source-free and the graph homomorphism $p$ is surjective, as done in Theorem \ref{thm:lr-insplit=2g-insplit}. 

 \begin{remark}
 While it is sufficient to assume that $p$ is surjective and $F$ is source-free, we note that the standard dynamical assumption of an essential system (essential textile system $T$ or equivalently essential 2-graph $\Lambda_T$) also implies all the results in this section.      
 \end{remark}

We begin by showing that the setup in Theorem \ref{thm:Priyanga} leads to the results in Theorem \ref{thm:lr-insplit=2g-insplit}.

\begin{theorem}[Theorem \ref{thm:Priyanga} $\longrightarrow$ Theorem \ref{thm:lr-insplit=2g-insplit}]
Let $T = (p, q: F \to E)$ be an LR textile system and let $\Lambda$ be its associated 2-graph. Further, assume that $p$ is surjective and $F$ is source free.
Let $\{ \mathcal{G}^i_z:  z \in E^0, 1 \le i \le m(z)\}$ be a 2-graph insplitting partition of $\Lambda$. Then, 
\begin{enumerate}

\item The partition $\{\mathcal{F}^i_v: v \in F^0, 1 \le i \le m(v)\}$ of $F^1$ defined in Theorem \ref{thm:Priyanga} as 
\begin{equation}
\mathcal{F}^i_v := \{ \lambda \in F^1: r(\lambda) = v,\;\; p(\lambda) \in \mathcal{G}^i_{r(p(\lambda)) = p(v)}  \} 
\end{equation}
satisfies the hypotheses of Theorem \ref{thm:lr-insplit=2g-insplit}, namely:
\begin{enumerate}[label = (\alph*)]
\item If $v, w \in F^0$ such that $p(\F^i_v) \cap p(\F^j_w) \ne \emptyset $, then $p(\F^i_v) = p(\F^j_w)$; and 
\item For each $v \in F^0$, there exists $w \in F^0$ and $1 \le j \le m(w)$ such that $q(vF^1) \subseteq p(\F^j_w)$.
\end{enumerate}

\item The partition $\{\mathcal{E}^i_z: z \in E^0, 1 \le i \le m(z)\}$ of $E^1$ defined in Theorem \ref{thm:Priyanga} by 
\begin{equation} 
\mathcal{E}^i_z := \mathcal{G}^i_z \cap z E^1
\label{def:E from G}
\end{equation}
satisfies $\mathcal{E}^i_z = p(\mathcal{F}^i_v)$ for any $v\in F^0$ such that $z= p(v)$, and hence agrees with the corresponding partition of $E^1$ in Theorem \ref{thm:lr-insplit=2g-insplit}.

\item The original 2-graph partition $\{ \mathcal G^i_z: z \in E^0, 1 \leq i \leq m(z)\}$ satisfies
\begin{equation}
\G^i_z = \mathcal E^i_z \sqcup r_F(q^{-1}(\mathcal E^i_z))
\end{equation}
for all $v$ such that  $z = p(v) \in E^0$ and for all $1\le i\le m(v)$. 
In particular, the 2-graph insplitting partition constructed in Theorem \ref{thm:lr-insplit=2g-insplit}  agrees with the original partition $\{\mathcal{G}^i_z: z\in E^0, 1\leq i\leq m(z) \}$.

\item The constructions of the textile system $\wt{T} = (\wt{p}, \wt{q}: \wt{F}\to \wt{E})$ in Theorems \ref{thm:Priyanga} and \ref{thm:lr-insplit=2g-insplit} coincide. 
\end{enumerate}

\end{theorem}

\begin{proof}

We first observe that for any $v \in F^0$ and $1 \leq i \leq m(p(v)),$ 
\begin{equation}
    p(\mathcal F^i_v) = \G^i_{p(v)} \cap E^1.
    \label{eq: new defn F^i_v}
\end{equation}
The containment $\subseteq$ follows from the definition, and the reverse containment follows from  $r$-path lifting for $p$.  
It immediately follows that (2) holds.
Furthermore, $p(\mathcal F^i_v) = p(\mathcal F^i_w)$
 whenever $p(v) = p(w).$
 
To see (1a), recall from Theorem \ref{thm:Priyanga} that $\{ \mathcal E^i_z\}_{i,z}$ is a partition of $E^1$.  Thus, (1a) follows from Equation \eqref{eq: new defn F^i_v}.

For (1b), suppose $\lambda \in vF^1$ has $q(\lambda)\in p(\mathcal F^j_w) = \mathcal E^j_{p(w)} \subseteq \G^j_{p(w)}$.  (For each $v$, such a $\lambda$ exists since $F$ is source free.) As $\{ \mathcal G^j_z\}_{j,z}$ satisfies the pairing condition, we have $v \in \G^j_z$ as well.  Thus, the pairing condition, together with the fact that $p(\mathcal F^i_u) = p(\mathcal F^i_w)$ whenever $p(u) = p(w)$, implies that for any $\mu \in vF^1$, $q(\mu) \in \G^j_z \cap E^1 = p(\mathcal F^j_w).$  That is, (1b) holds.

 For  (3), Equation  \eqref{def:E from G} implies that $\G^i_z \cap E^1 = \E^i_z$. To see that $\G^i_z \cap F^0 = r_F(q^{-1}(\E^i_z))$, we use  the fact that $F$ is source free, the pairing condition, and the definition of $r_\Lambda$.  From these, we conclude that $w \in F^0 \cap \G^i_z$ if and only if $q(w) = r_\Lambda(w) = z$ and there is $\lambda \in F^1$ with $r_F(\lambda) = w$ and $q(\lambda) \in \G^i_z \cap E^1 = \E^i_z .$  That is, $w \in F^0 \cap \G^i_z$ if and only if $w \in r_F(q^{-1}(\E^i_z))$, so (3) holds.

 As we have now proved that the partitions $\mathcal E$ and $\mathcal F$ constructed in Theorem \ref{thm:Priyanga} have the same properties as the partitions of Theorem \ref{thm:lr-insplit=2g-insplit}, and the formulas for  $\tilde p, \tilde q$ are the same in both Theorems, (4) holds.

\end{proof}

We next show that starting from the setting in Theorem \ref{thm:lr-insplit=2g-insplit}, one can obtain the results in Theorem \ref{Thm: Main result III}.

\begin{theorem}[Theorem \ref{thm:lr-insplit=2g-insplit} $\longrightarrow$ Theorem \ref{Thm: Main result III}]

Let $T= (p,q: F \to E)$ be an LR textile system with associated 2-graph $\Lambda$. Further, assume that $p$ is surjective and $F$ is source free.
Let $\{\mathcal{F}^i_v: v \in F^0, 1 \le i \le m(v)\}$ be a partition of $F^1$ as in Theorem \ref{thm:lr-insplit=2g-insplit}, namely:
\begin{enumerate}[label = (\alph*)]
\item If $v, w \in F^0$ such that $p(\F^i_v) \cap p(\F^j_w) \ne \emptyset $, then $p(\F^i_v) = p(\F^j_w)$; and 
\item For each $v \in F^0$, there exists $w \in F^0$ and $1 \le j \le m(w)$ such that $q(vF^1) \subseteq p(\F^j_w)$.
\end{enumerate}
\noindent Then, 
\begin{enumerate}

\item The partition $\{\mathcal{E}^i_z: z \in E^0, 1 \le i \le m(v)\}$ of $E^1$
defined in Theorem \ref{thm:lr-insplit=2g-insplit} as $\mathcal{E}^i_z := p(\F^i_v)$, where $p(v) = z$, 
satisfies the hypothesis of Theorem \ref{Thm: Main result III}, that is,

For each $u \in F^0$, there is a unique $j \in \{ 1, 2, \ldots, m(q(u))\}$ such that $q(uF^1) \subseteq \E^j_{q(u)}$.

\item The partition $\{\mathcal{F}^i_v: v \in F^0, 1 \le i \le m(v)\}$ of $F^1$ constructed from $\{ \E^i_z\}_{i,z}$ in Theorem \ref{Thm: Main result III} 
agrees with the original partition $\{\mathcal{F}^i_v \}$.
\end{enumerate}
Consequently, the 2-graph insplitting partitions and the textile systems $\wt T$ constructed in the two Theorems coincide.
\end{theorem}

\begin{proof}
Let $\{\mathcal{F}^i_v: v \in F^0, 1 \le i \le m(v)\}$ be a partition of $F^1$ as in Theorem \ref{thm:lr-insplit=2g-insplit}. 
 \begin{enumerate}
 \item Recall from Proposition \ref{lr-insplit-2} that 
$\{ \mathcal{E}^i_z := p(\F^i_v)\}_{i,z }$, where $z = p(v)$, is a well-defined partition of $E^1$. 
Fix $u \in F^0$. Then by hypothesis (b), there exists $w \in F^0$ and $1 \le j \le m(w)$ such that $q(uF^1) \subseteq p(\F^j_w) = \E^j_{p(w)}$. 
As $\E^j_{p(w)} \subseteq p(w)E^1$ and $r_E(q(uF^1)) = q(r_F(uF^1)) = q(u)$, we must have $q(u) = p(w)$.  As $\mathcal E$ is a partition and hence $\E^k_{q(u)} \cap \E^j_{q(u)} \ne \emptyset$ implies $k=j$, (1) holds.

 \item It is sufficient to prove that for all $v \in F^0$ and $1\le i\le m(v)$, 
 \[ \mathcal{F}^i_v = vF^1 \cap p^{-1}(p(\F^i_v)). \]
 Clearly, $\mathcal{F}^i_v \subseteq vF^1 \cap p^{-1}(p(\F^i_v))$. To prove the opposite containment, suppose that  $f \in vF^1 \cap p^{-1}(p(\F^i_v))$
 but $f \in \F^k_v$ for some $k$. 
 Then $p(f)  = p(\lambda)$ for some $\lambda \in \F^i_v$, i.e., $p(f) \in p(\F^i_v) \cap p(\F^k_v)$, so by hypothesis (a),  $p(\F^i_v) = p(\F^k_v)$. Proposition \ref{prop:lr-ptn} now tells us that $i = k$, as desired. 
 
 \end{enumerate}

The 2-graph insplitting partitions constructed in both Theorems are the same, by definition:
\begin{equation} \label{eq: 2-graph partition constructed}
\{ \mathcal{G}^i_z := \E^i_z \sqcup r_F(q^{-1}(\E_z^i)) : z \in E^0, 1\le i \le m(z) \}.
\end{equation}
 Finally, 
as we have already observed that the partitions used to construct $\wt F$ and $\wt E$ in Theorems \ref{thm:lr-insplit=2g-insplit} and \ref{Thm: Main result III} coincide, the graphs $\wt F, \wt E$ resulting from the two Theorems are the same.  Moreover, the two Theorems use the same definitions of $\tilde p, \tilde q$, thus, the two textile systems coincide.\qedhere

 \end{proof}

Lastly, we show that starting from the setting of Theorem \ref{Thm: Main result III}, we recover the results in Theorem \ref{thm:Priyanga}.

\begin{theorem}[Theorem \ref{Thm: Main result III} $\longrightarrow$ Theorem \ref{thm:Priyanga}]
Let $T= (p,q: F \to E)$ be an LR textile system with associated 2-graph $\Lambda$. Further, assume that $p$ is surjective and $F$ is source free.
Let $\{ \E^i_z: z \in E^0, 1 \le i \le m(z)\}$ be a partition of $E^1$ as in Theorem \ref{Thm: Main result III}, so that for each $u \in F^0$, there exist $ z = q(u) \in E^0$ and a unique $j \in \{ 1, 2, \ldots, m(z)\}$ such that $q(uF^1) \subseteq \E^j_z$. 

Define partitions of $F^1$ and $\Lambda^{1}$ as in Theorem \ref{Thm: Main result III}, namely: For each  $v \in F^0$ and $ z \in E^0$, let
\begin{equation}
\mathcal{F}^i_v := vF^1 \cap p^{-1}(\E^i_{p(v)}), \;\;\; \text{ for } 1\le i\le m(p(v)),
\label{eq:defne F from E}
\end{equation}
\begin{equation} \label{eq: defne G from E}
\mathcal{G}^i_z := \E^i_z \sqcup r_F(q^{-1}(\E_z^i)), \;\;\; \text{ for } 1\le i\le m(z).
\end{equation}

\noindent Then, 
\begin{enumerate}
\item The partition of $E^1$ constructed in Theorem \ref{thm:Priyanga} from the 2-graph partition \eqref{eq: defne G from E} agrees with the original partition $\{\mathcal{E}^i_z \}$, that is 
\[ \mathcal{E}^i_z = \mathcal{G}^i_z \cap z E^1. \]
\item The partition of $F^1$ constructed in Theorem \ref{thm:Priyanga} from $\{\mathcal{G}^i_z\}_{i,z}$ agrees with the partition of Equation \eqref{eq:defne F from E}, that is
\[ \mathcal{F}^i_v = \{ \lambda \in vF^1  :  p(\lambda) \in \mathcal{G}^i_{ p(v)} \}. \]
\item The constructions of the textile system $\wt{T} = (\tilde{p}, \tilde{q}: \wt{F}\to \wt{E})$ in Theorems \ref{thm:Priyanga} and  \ref{Thm: Main result III} coincide.
\end{enumerate}
\end{theorem}

\begin{proof}
Following the fact that $\mathcal{G}^i_z \cap z E^1 = \bigg(\E^i_z \sqcup r(q^{-1}(\E_z^i)) \bigg) \cap z E^1 = \E^i_z \cap z E^1 = \mathcal{E}^i_z$, we obtain (1).
For (2), we have $f \in \F^i_v$ if and only if $f \in \{ \lambda \in vF^1 :  p(\lambda) \in \E^i_{p(v)}\}$ which is equivalent to $f \in \{ \lambda \in vF^1 :  p(\lambda) \in \G^i_{p(v)}\}$ as $p(f)\in E^1$ but $p(f) \notin r_F(q^{-1}(\E_z^i))$.

For (3), it is sufficient to show that the maps $\tilde{q}$ defined in Theorem \ref{thm:Priyanga} and Theorem \ref{Thm: Main result III} are the same. To this end, we note that for $v \in F^0$,
\begin{eqnarray*}
v \in \mathcal{G}^m_{q(v)} & \iff & v \in r_F(q^{-1}(\E^m_{q(v)}))  \iff \exists \lambda \in vF^1 \text{ such that } q(\la) \in \E^m_{q(v)} \\
& \iff & q(vF^1) \subseteq \mathcal{E}^m_{q(v)},
\end{eqnarray*}
where the last equivalence follows from the hypothesis.
So, the definition of $\tilde{q}$ in Theorem \ref{thm:Priyanga} and Theorem \ref{Thm: Main result III} coincide, and hence the textile systems are the same.
 \end{proof}

\appendix
\newpage
\section{\texorpdfstring{$C^*$}-algebraic results}
\label{sec:$C^*$}

In this appendix, we discuss the $C^*$-algebraic implications of our work.  First, we show in Theorem \ref{thm:2graph-insplit-isom} that the $C^*$-algebra of a 2-graph is unchanged under a 2-graph insplitting; this result was established in \cite[Theorem 3.9]{EFGGGP} for the one-vertex-at-a-time version of 2-graph insplitting discussed in that paper.  We also show  that this isomorphism is gauge-invariant and diagonal-preserving.  Thus, \cite[Theorem 3.2]{carlsen-rout-kgraph} shows that 2-graph insplitting yields an eventual 1-sided conjugacy.  (In fact, as we have expressed 2-graph insplitting as a composition  of conjugacy-preserving moves on textile systems in  Theorem \ref{thm:6.1}, we have already proved the stronger result that 2-graph insplitting is a 1-sided conjugacy.)  We then define outsplitting for 2-graphs, as introduced in \cite{listhartke}, and explain why this paper has focused on linking the textile-system and 2-graph definitions of insplitting rather than outsplitting.

To simplify our calculations we will use an edge-based  definition of $2$-graph algebras which has seen some light in the literature, see \cite[Definition 2.1.1]{EFGGGP}, \cite[Theorem C.1]{rsy2}, \cite[Definition 7.4]{Kumjian-Pask-Sims-Homology} and alluded to in \cite[Remarks 1.6]{kp}. 

\begin{definition}[Edge definition of a $2$-graph $C^*$-algebra] \label{dfn:edgecstardef}
Let $\Lambda$ be a row-finite $2$-graph. An \emph{edge Cuntz--Krieger
$\Lambda$-family} in a $C^*$-algebra $A$ is a function $s : 
\Lambda^0 \cup \Lambda^{\varepsilon_1} \cup \Lambda^{\varepsilon_2} \to \text{PIsom}(A)$, denoted $x \mapsto s_x$, which assigns a partial isometry $s_x$ to each $x \in \Lambda^0 \cup \Lambda^{\varepsilon_1} \cup \Lambda^{\varepsilon_2} $ such that
\begin{itemize}
\item[(ECK1)] $\{s_v : v\in \Lambda^0\}$ is a collection of mutually orthogonal  projections;
\item[(ECK2)] $s_e s_f  = s_{f'} s_{e'}$ whenever $e,e' \in \Lambda^{\varepsilon_1}$, $f,f' \in \Lambda^{\varepsilon_2}$ and $e f \sim f' e'$;
\item[(ECK3)] $s^*_f s_f = s_{s(f)}$ for all $\lambda\in \Lambda^{\varepsilon_1 + \varepsilon_2}$; and
\item[(ECK4)] $s_v = \sum_{f \in v\Lambda^{\varepsilon_i}} s_f s_f^*$ for all  $v\in \Lambda^0$ and $i=1,2$.
\end{itemize}
The $C^*$-algebra $C^*(\Lambda)$ is the universal $C^*$-algebra generated by an edge Cuntz--Krieger $\Lambda$-family, and we often write $\{ t_\lambda: \lambda \in \Lambda^0 \cup \Lambda^{\varepsilon_1} \cup \Lambda^{\varepsilon_2}\}$ for the generators of the universal $C^*$-algebra $C^*(\Lambda)$.   That is, for any edge Cuntz--Krieger family $s: \Lambda^0 \cup \Lambda^{\varepsilon_1} \cup \Lambda^{\varepsilon_2} \to A$, there is a unique surjective $*$-homomorphism $\pi_s$ from $C^*(\Lambda) $ onto the $C^*$-algebra generated by the image of $s$, which satisfies $\pi_s(t_\lambda) = s_\lambda$ for all generators $t_\lambda$.\footnote{The existence of $C^*(\Lambda)$ can be proved by following for instance \cite[Proposition 1.21]{raeburn}.}
\end{definition}

\noindent
For general elements $\lambda \in \Lambda$, if $e_1 \cdots e_n$ is a representative of $\Lambda$ with each $e_j \in \Lambda^{\varepsilon_i}$ for some $i=1,2$, we define $t_\lambda = t_{e_1} \cdots t_{e_n}.$  Condition (ECK2) guarantees that $t_\lambda$ is independent of the choice of representative of $\lambda$. 

Because of the universal property of $C^*(\Lambda)$, there is a canonical action $\gamma$ of ${\mathbb T}^2$ on $C^*(\Lambda)$, called the \textit{gauge action}, which satisfies 
\[ 
\gamma_z (t_e) = z^{d(e)} t_e \qquad \text{ and } \qquad \gamma_z (t_v) = t_v
\]

\noindent
for all $z \in {\mathbb T}^2, e \in \Lambda^{\varepsilon_1} \sqcup \Lambda^{\varepsilon_2}$ and $ v \in \Lambda^0$. A standard $\varepsilon/3$ argument shows that this action is strongly continuous.

\begin{remarks}[Switching between edge-based and path-based models for $C^* ( \Lambda )$] \label{rmk:itsthesame}
One can extract from \cite{Kumjian-Pask-Sims-Twisted, Kumjian-Pask-Sims-Homology} a proof that Definition \ref{dfn:edgecstardef} and the original definition  \cite[Definitions 1.5]{kp} of $C^*(\Lambda)$ agree.  To be precise, 
Definition~\ref{dfn:edgecstardef} corresponds to the definition of Cuntz--Krieger $\iota$-family  given in \cite[Definition 7.4]{Kumjian-Pask-Sims-Homology}, where $\iota$ denotes the trivial cubical cocycle $\iota \in Z^2 ( \Lambda , \mathbb{T} )$, so the universal $C^*$-algebra of Definition \ref{dfn:edgecstardef} agrees with the $C^*$-algebra $C^*_\iota(\Lambda)$ of \cite[Definition 7.5]{Kumjian-Pask-Sims-Homology}. Then \cite[Theorem 3.16]{Kumjian-Pask-Sims-Twisted} identifies $\iota$ with the trivial categorical cocycle $c_\iota \in \underline{Z}^2 ( \Lambda , \mathbb{T} )$. Next \cite[Theorem 5.3]{Kumjian-Pask-Sims-Twisted} gives a gauge-equivariant isomorphism $\phi : C_\iota^* ( \Lambda ) \to C^* ( \Lambda , c_\iota )$. 
Finally, one checks that $C^* ( \Lambda , c_\iota ) $ defines the same object as the original $ C^* ( \Lambda )$ from \cite[Definitions 1.5]{kp}, as both are universal $C^*$-algebras with the same generators and relations. 
\end{remarks}

In order to link our work with that of Carlsen and Rout \cite{carlsen-rout-kgraph}, we recall that the {\em diagonal subalgebra} $\mathcal D(\Lambda)$ of $C^*(\Lambda)$ is the Abelian subalgebra densely spanned by $\{ t_\lambda t_\lambda^*: \lambda \in \Lambda\}.$
\begin{theorem}[$2$-graph insplitting induces gauge-invariant  isomorphism]
Let $\Lambda$ be an essential row-finite \textit{2}-graph with associated 
1-skeleton $(G_\Lambda ,c_\Lambda)$. Let $\E$ be an insplitting partition of $(G_\Lambda ,c_\Lambda)$ satisfying Definition \ref{dfn:pairing}. Let $\Lambda_I$ be the 2-graph resulting from insplitting $\Lambda$ with this partition.  Then there is a gauge-equivariant diagonal-preserving isomorphism $\pi_S : C^*(\Lambda ) \to C^*(\Lambda_I)$.
\label{thm:2graph-insplit-isom}
\end{theorem}

\begin{proof}
Let $\{ s_v , s_f :v \in \Lambda^0 ,  f \in  \Lambda^{\varepsilon_1} \sqcup \Lambda^{\varepsilon_2}  \}$ be an edge Cuntz-Krieger $\Lambda$-family generating $C^* ( \Lambda )$ 
and $\{ t_{v^j} , t_{f^j} : v^j \in \Lambda_I^0, f^j \in \Lambda_I^{\varepsilon_1} \sqcup \Lambda_I^{\varepsilon_2} \}$ be an edge Cuntz-Krieger $\Lambda_I$-family generating $C^* ( \Lambda_I )$. For $v \in \Lambda^0$ and $f \in \Lambda^{\varepsilon_1} \sqcup \Lambda^{\varepsilon_2}$ set 
\[
S_v = \sum_{1 \le i \le m(v)} t_{v^i} 
\ \text{  and  } \ S_f = \sum_{1 \le j \le m(s(f))} t_{f^j}
.
\]

\noindent
We claim that the family $\{ S_v , S_f : v \in \Lambda^0 , f \in \Lambda^{\varepsilon_1} \sqcup \Lambda^{\varepsilon_2} \}$ is an edge Cuntz-Krieger $\Lambda$-family in $C^* ( \Lambda_I )$. To do this we first observe that each $S_f$ is a partial isometry since it is a finite sum of partial isometries with different source projections (cf.\ \cite[Lemma 2.1]{Maloney-Pask-Raeburn}).

Next we check that this family $\{ S_f, S_v\}$ satisfies the relations (ECK1)--(ECK4). 

We begin with (ECK1). The $S_v$'s are non-zero mutually orthogonal projections since they are sums of projections satisfying the same properties. For (ECK3) let $f \in \Lambda^{\varepsilon_i}$, $i=1,2$. 
Since the generators $t_{v^j}, t_{f^j}$ of $C^*(\Lambda_I)$ satisfy (ECK1) and (ECK3),
\begin{align*}
S_f^* S_f &= \left( \sum_{1 \le j \le m(s(f))} t_{f^j} \right)^* \left( \sum_{1 \le j \le m(s(f))} t_{f^j} \right) \\
&= \sum_{1 \le j \le m(s(f))} t_{f^j}^* t_{f^j} \text{ since } s ( f^i ) = s(f)^i \neq s(f)^j = s ( f^j ) \text{ if} \ i \neq j \\
&= \sum_{1 \le j \le m(s(f))} t_{s(f)^j} =  S_{s(f)} \text{ by definition.}
\end{align*}

\noindent
For (ECK2) let $g , g'\in \Lambda^{\varepsilon_1}$ and $h, h' \in \Lambda^{\varepsilon_2}$ be such that $gh \sim  h'g'$ in $\Lambda$. Suppose that $h \in \mathcal E_{s(g)}^\ell$ and $g' \in \mathcal E_{s(h')}^m$. Then the fact that $T := \{ t_x: x \in \Lambda_I^0 \cup \Lambda_I^{\varepsilon_1} \cup \Lambda_I^{\varepsilon_2}\}$ is a family of partial isometries satisfying (ECK1)  allows us to conclude that
\[t_{g^j} t_{h^k} =t_{g^j} t_{s(g)^j} t_{r(h)^\ell} t_{h^k}\]
is zero unless $j= \ell$. Moreover, the definition of $\sim_I$, together with the fact that $T$ satisfies (ECK2), implies that $t_{g^\ell} t_{h^k} = t_{(h')^m} t_{(g')^k}. $ Since $gh \sim h'g'$ implies in particular that $s(h) = s(g')$, 
\begin{align*}
S_g S_h &= \left( \sum_{1 \le j \le m(s(g))} t_{g^j} \right) \left( \sum_{1 \le k \le m(s(h))} t_{h^k} \right) = \sum_{1\leq k \leq m(s(h))} t_{g^\ell} t_{h^k} \\
&= \sum_{1 \le k \le m(s(g'))} t_{(h')^m} t_{(g')^k} = \sum_{1\leq j' \leq m(s(h'))} t_{(h')^{j'}}\sum_{1 \le k \le m(s(g'))} t_{(g')^k} \\
&= S_{h'} S_{g'}. 
\end{align*}

\noindent
For (ECK4), 
recall that every $g \in \Lambda_I^{\varepsilon_i}$ is of the form $g = f^k$ for some $f \in \Lambda^{\varepsilon_i}$, and that $r_{\Lambda_I}(f^k) = r_\Lambda(f)^j$ if $f \in \mathcal E^j_{r(f)}$.  Therefore, for any $v \in \Lambda^0$ and $i=1,2$ we can write 
\[ 
v\Lambda_I^{\varepsilon_i} = \bigcup_{1\leq j \leq m(v)} \bigcup_{f \in \mathcal E^j_{v}} \{ f^k: 1 \leq k \leq m(s_\Lambda(f))\}.\]
Moreover, since $t_{f^k} = t_{f^k} t_{s(f)^k}$ by (ECK3), we have $t_{f^k} t_{f^\ell }^* = \delta_{\ell,k} t_{f^k} t_{f^\ell}^*$ by (ECK1).
It follows that 
\begin{align*}
S_v = \sum_{1 \le j \le m(v)} t_{v^j} &= \sum_{1 \le j \le m(v)} \sum_{f \in \E_v^j} \sum_{1 \leq k \leq m(s(f))} t_{f^k} t_{f^k}^* \text{ by (ECK4) in $C^* ( \Lambda_I )$} \\
&= 
\sum_{f \in v\Lambda^{\varepsilon_i} } \sum_{1\leq k \leq m(s(f))} t_{f^k} t_{f^k}^* = \sum_{f\in v\Lambda^{\varepsilon_i}} \sum_{1 \leq k, \ell \leq m(s(f))} t_{f^k} t_{f^\ell}^* \\
&= \sum_{f\in v\Lambda^{\varepsilon_i}} S_f S_f^*,
\end{align*}
so $\{ S_v, S_f \}$ satisfies (ECK4). By the universal property of $C^* ( \Lambda )$ there is a  $*$-homomorphism $\pi_{S} : C^* ( \Lambda ) \to C^* ( \Lambda_I )$ taking $s_x$ to $S_x$ for all $x \in \Lambda^0 \cup \big( \Lambda^{\varepsilon_1} \sqcup \Lambda^{\varepsilon_2} \big)$. 
Moreover, the definition of $S_x$ ensures that $\pi_S$ is gauge-equivariant.

We claim that $\pi_S$ is onto. To do this we show that the generators of $C^* ( \Lambda_I )$ lie in the image $C^* ( \{ S_x : x \in \Lambda^0 \cup (\Lambda^{\varepsilon_1} \sqcup \Lambda^{\varepsilon_2})\} )$ of $\pi_S$. 

Fix $v^j$, where $1 \le j \le m(v)$. Recall that, for any $i = 1,2$, 
\[ v^j \Lambda_I^{\varepsilon_i} = \{ f^k: d(f) = \varepsilon_i, f \in \E^j_v, 1 \leq k \leq m(s(f))\}.
\]
Using the fact that the generators of $C^*(\Lambda_I)$ satisfy (ECK4), we conclude that for each $i=1,2,$
\[
S_v \left( \sum_{f \in \E^j_v \cap \Lambda^{\varepsilon_i}} S_f S_f^* \right) = \left( \sum_{1 \le k \le m(v)} t_{v^k} \right) \left( \sum_{f \in \E^j_v \cap \Lambda^{\varepsilon_i}} \sum_{1 \le \ell  \le m(s(f))} t_{f^\ell} t_{f^\ell}^* \right) = t_{v^j} \sum_{ g \in v^j \Lambda_I^{\varepsilon_i}} t_{g} t_{g}^*= t_{v^j} ,
\]
so each generator $t_{v^j}$ of $C^*(\Lambda_I)$ lies in the image of $\pi_S$.

Applying (ECK1) now implies that, for any $1\leq j \leq m(s(f))$,
\[ t_{f^j} = t_{f^j} t_{s(f)^j} = \left( \sum_{1\leq k \leq m(s(f))} t_{f^k} \right) t_{s(f)^j} = S_f t_{s(f)^j} \]
also lies in the image of $\pi_S$.  We conclude that $\pi_S: C^*(\Lambda) \to C^*(\Lambda_I)$ is onto.

To see that $\pi_S$ is injective, we invoke the gauge-invariant uniqueness theorem \cite[Theorem 3.4]{kp}. Since $\pi_S$ is gauge-equivariant, it is injective because  $S_v$ is nonzero for any $v \in \Lambda^0$, thanks to the universal property of $C^*(\Lambda_I)$.  That is, $\pi_S$ is a $*$-isomorphism.

Finally, we check that $\pi_S(\mathcal D(\Lambda)) = \mathcal D(\Lambda_I).$  Given $\lambda \in \Lambda$, choose a representative $e_1 \cdots e_n$ of $\Lambda$ with each $e_j \in \Lambda^{\varepsilon_i}$ for some $i=1,2$.  As $t_{e_j^k} t_{e_{j+1}^m} = 0$ unless $e_{j+1} \in \mathcal E_{s(e_j)}^k,$ 
\[ S_{e_1} \cdots S_{e_n} = \sum_{1\leq \ell \leq m(s(e_n))} t_{e_1^{m_1}} \cdots t_{e_{n-1}^{m_{n-1}}} t_{e_n^\ell},\]
where $e_j \in \mathcal E^{m_{j-1}}_{r(e_j)}$ for all $2 \leq j \leq n.$  By (ECK1) and (ECK3), $t_{e_n^\ell} t_{e_n^k}^* = \delta_{\ell, k} t_{e_n^\ell} t_{e_n^\ell}^*$.  Thus,   
\[ S_\lambda S_\lambda^* = (S_{e_1} \cdots S_{e_n})(S_{e_1} \cdots S_{e_n})^* = \sum_{1\leq \ell \leq m(s(e_n))} t_{e_1^{m_1}} \cdots t_{e_{n-1}^{m_{n-1}}} t_{e_n^\ell}(t_{e_1^{m_1}} \cdots t_{e_{n-1}^{m_{n-1}}} t_{e_n^\ell})^*\]
lies in $\mathcal D(\Lambda_I)$, being a finite sum of elements of the form $t_\lambda t_\lambda^*.$  As every $*$-isomorphism is norm-preserving, we conclude that $\pi_S(\mathcal D(\Lambda)) \subseteq \mathcal D(\Lambda_I).$ 

To show that $\pi_S(\mathcal D(\Lambda))= \mathcal D(\Lambda_I)$, since $\pi_S$ is norm-preserving it suffices to show that $t_\lambda t_\lambda^* \in \pi_S(\mathcal D(\Lambda))$ for all $\lambda \in \Lambda_I$.  So, fix $\lambda = e_1^{j_1} e_2^{j_2} \cdots e_n^{j_n} \in \Lambda_I$, and observe that since $s(e_n^{j_n}) = s(e_n)^{j_n},$
\[ s(\lambda) \Lambda_I^{\varepsilon_i} = \{ e^\ell: e \in \Lambda^{\varepsilon_i} \cap \mathcal E_{s(e_n)}^{j_n} , \ 1 \leq \ell \leq m(s(e))\}. \]
Therefore, by (ECK4), 
\[ t_\lambda = \sum_{e \in \Lambda^{\varepsilon_i} \cap \mathcal E_{s(e_n)}^{j_n} } \sum_{1\leq \ell \leq m(s(e))} t_\lambda t_{e^\ell} t_{e^\ell}^* = \sum_{e \in \Lambda^{\varepsilon_i} \cap \mathcal E_{s(e_n)}^{j_n} } t_\lambda S_e S_e^*.\]
As above, recall that for all $1 \leq i \leq n-1$, there is a unique $1 \leq j_i \leq m(s(e_i))$ so that $t_{e_1}^{j_1} \cdots t_{e_n}^{j_n} \not= 0$ [namely, $j_i = m$ iff $e_{i+1} \in \mathcal E^m_{s(e_i)}$].  Consequently, using the fact that $s(e_n^k) = s(e_n)^k,$
\[ \sum_{e \in \Lambda^{\varepsilon_i} \cap \mathcal E_{s(e_n)}^{j_n} } S_{e_1} \cdots S_{e_n} S_e  ( S_{e_1} \cdots S_{e_n} S_e )^* = \sum_{e \in \Lambda^{\varepsilon_i} \cap \mathcal E_{s(e_n)}^{j_n} } t_\lambda S_e S_e^* t_\lambda^* = t_\lambda t_\lambda^* \in \pi_S(\mathcal D(\Lambda)). \]
We conclude that $\pi_S$ is diagonal preserving, as claimed.
\end{proof}

\subsection{What about outsplitting?}
The following definition was introduced in \cite[Definition 1.5.1]{listhartke}.

\begin{definition}[2-graph outsplitting]
Let $\Lambda$ be a row-finite source-free 2-graph. Let $G= (\Lambda^0, 
\Lambda^1, r, s)$ be its 1-skeleton. An \emph{outsplitting partition} of $G$ is a partition $\mathcal{E}= \{\mathcal{E}_w: w\in \Lambda^0\}$ of $\Lambda^{\epsilon_1}\sqcup \Lambda^{\epsilon_2}$, where $\mathcal{E}_w = \{ \E_w^1, \ldots, \E_w^{m(w)}\} $ is a partition of $s^{-1}(w)\cap \Lambda^1$ satisfying the following analogue of the pairing condition: 
\[ \text{ if } f, g \in s^{-1}(w)\cap \Lambda^1 \text{ and } \exists a, b \in G^1 \text{ s.t. } af \sim bg, \text{ then } f \in \mathcal E_w^j \iff g \in \mathcal E_w^j.\]

We then define a 2-colored graph $((\Lambda_O^0, \Lambda_O^1, r_O, s_O), d)$ where 
\begin{itemize}
    \item $\Lambda_O^0= \{v^j: v\in \Lambda^0, 1 \leq j \leq m(v)\},$
    \item $\Lambda_O^{\epsilon_i}= \{f^j: f\in \Lambda^{\epsilon_i}, 1\leq j \leq m(r(f)) \}, $
    \item $s_{\Lambda_O}(f^j)= s(f)^k$ if $f \in \mathcal E_{s(f)}^k,$
   \item $r_{\Lambda_O}(f^j)= r(f)^j$,
   \item $d(f^j) = d(f).$
\end{itemize}    
The commuting squares in $\Lambda_O$ are given by 
\[ f^i g^j \sim_O a^i b^k \iff fg \sim ab , f \in \mathcal E^j_{r(g)}, a \in \mathcal E^k_{r(b)}.\]
\end{definition}
An argument analogous to Theorem \ref{thm:inplitcomplete} (cf.~also \cite[Claim 1.5.1]{listhartke}) will show that $\Lambda_O$ is a 2-graph.

In general, by \cite[Theorem 4.0.1]{listhartke}, 2-graph outsplitting results in a stable isomorphism of $C^*$-algebras but not an on-the-nose isomorphism $C^*(\Lambda) \cong C^*(\Lambda_O)$. 
Therefore, by \cite[Theorem 3.2]{carlsen-rout-kgraph}, outsplitting (even for directed graphs; see \cite[Example 5.1]{bates-pask}) does not give rise to an eventual conjugacy, and in particular does not give rise to a conjugacy.  In particular, we can't expect to describe 2-graph outsplitting in terms of Johnson--Madden outsplitting, or other related moves.

\bibliographystyle{amsalpha}
\bibliography{eagbib}

\newcommand{\etalchar}[1]{$^{#1}$}
\providecommand{\bysame}{\leavevmode\hbox to3em{\hrulefill}\thinspace}
\providecommand{\MR}{\relax\ifhmode\unskip\space\fi MR }
\providecommand{\MRhref}[2]{%
  \href{http://www.ams.org/mathscinet-getitem?mr=#1}{#2}
}
\providecommand{\href}[2]{#2}
\begin{thebibliography}{HRSW13}

\bibitem[ABCE23]{ABCE}
B.~Armstrong, K.~A. Brix, T.~M. Carlsen, and S.~Eilers, \emph{Conjugacy of
  local homeomorphisms via groupoids and {${\rm C}^*$}-algebras}, Ergodic
  Theory Dynam. Systems \textbf{43} (2023), no.~8, 2516--2537.

\bibitem[AER18]{aer-1}
S.~E. Arklint, S.~Eilers, and E.~Ruiz, \emph{A dynamical characterization of
  diagonal-preserving *-isomorphisms of graph {$C^*$}-algebras}, Ergodic Theory
  Dynam. Systems \textbf{38} (2018), no.~7, 2401--2421.

\bibitem[AER22]{aer-2}
\bysame, \emph{Geometric classification of isomorphism of unital graph
  {$C^*$}-algebras}, New York J. Math. \textbf{28} (2022), 927--957.

\bibitem[Aso00]{Aso}
H.~Aso, \emph{Conjugacy of ${\ZZ}^2$-subshifts and textile systems}, Publ.
  RIMS.\ Kyoto Univ. \textbf{36} (2000), 1--18.

\bibitem[BC20a]{brix-carlsen-2}
K.~A. Brix and T.~M. Carlsen, \emph{Cuntz-{K}rieger algebras and one-sided
  conjugacy of shifts of finite type and their groupoids}, J. Aust. Math. Soc.
  \textbf{109} (2020), no.~3, 289--298.

\bibitem[BC20b]{brix-carlsen}
\bysame, \emph{{$\rm C^*$}-algebras, groupoids and covers of shift spaces},
  Trans. Amer. Math. Soc. Ser. B \textbf{7} (2020), 134--185.

\bibitem[BCW17]{brownlowe-carlsen-whittaker}
N.~Brownlowe, T.M. Carlsen, and M.F. Whittaker, \emph{Graph algebras and orbit
  equivalence}, Ergodic Theory and Dynamical Systems \textbf{37} (2017), no.~2,
  389--417.

\bibitem[BHRS02]{bhrs}
T.~Bates, J.-H. Hong, I.~Raeburn, and W.~Szyma\'nski, \emph{The ideal structure
  of the ${C}^*$-algebras of infinite graphs}, Illinois J. Math. \textbf{46}
  (2002), no.~4, 1159--1176.

\bibitem[BP04]{bates-pask}
T.~Bates and D.~Pask, \emph{Flow equivalence of graph algebras}, Ergodic Theory
  Dynam. Systems \textbf{24} (2004), no.~2, 367--382.

\bibitem[Bri22]{brix-balanced}
K.~A. Brix, \emph{Balanced strong shift equivalence, balanced in-splits, and
  eventual conjugacy}, Ergodic Theory Dynam. Systems \textbf{42} (2022), no.~1,
  19--39.

\bibitem[CDOE24]{carlsen-doron-eilers}
T.~M. Carlsen, A.~Dor-On, and S.~Eilers, \emph{Shift equivalences through the
  lens of {C}untz-{K}rieger algebras}, Anal. PDE \textbf{17} (2024), no.~1,
  345--377.

\bibitem[CEOR19]{CEOR}
T.~M. Carlsen, S.~Eilers, E.~Ortega, and G.~Restorff, \emph{Flow equivalence
  and orbit equivalence for shifts of finite type and isomorphism of their
  groupoids}, J. Math. Anal. Appl. \textbf{469} (2019), no.~2, 1088--1110.

\bibitem[CK80]{cuntz-krieger}
J.~Cuntz and W.~Krieger, \emph{A class of {$C\sp{\ast} $}-algebras and
  topological {M}arkov chains}, Invent. Math. \textbf{56} (1980), no.~3,
  251--268.

\bibitem[CR17]{carlsen-rout}
T.~M. Carlsen and J.~Rout, \emph{Diagonal-preserving gauge-invariant
  isomorphisms of graph {$C^\ast$}-algebras}, J. Funct. Anal. \textbf{273}
  (2017), no.~9, 2981--2993.

\bibitem[CR21]{carlsen-rout-kgraph}
\bysame, \emph{Orbit equivalence of higher-rank graphs}, J. Operator Theory
  \textbf{86} (2021), no.~2, 395--424.

\bibitem[Cun81]{cuntz-K-thy}
J.~Cuntz, \emph{A class of ${C}^*$-algebras and topological {M}arkov chains.
  {II}. reducible chains and the {E}xt-functor for ${C}^*$-algebras}, Invent.
  Math. \textbf{63} (1981), no.~1, 25--40.

\bibitem[EFG{\etalchar{+}}22]{EFGGGP}
C.~Eckhardt, K.~Fieldhouse, D.~Gent, E.~Gillaspy, I.~Gonzales, and D.~Pask,
  \emph{Moves on $k$-graphs preserving {M}orita equivalence}, Canad. J. Math.
  \textbf{74} (2022), 655--685.

\bibitem[ERRS17]{errs-cuntz}
S.~Eilers, G.~Restorff, E.~Ruiz, and A.~P.~W. S\o{}rensen, \emph{Invariance of
  the {C}untz splice}, Math. Ann. \textbf{369} (2017), no.~3-4, 1061--1080.

\bibitem[ERRS21]{errs}
S.~Eilers, G.~Restorff, E.~Ruiz, and A.P.~W. S\o{}rensen, \emph{The complete
  classification of unital graph {$C^*$}-algebras: geometric and strong}, Duke
  Math. J. \textbf{170} (2021), no.~11, 2421--2517.

\bibitem[ES12]{evans-sims}
D.~G. Evans and A.~Sims, \emph{When is the {C}untz-{K}rieger algebra of a
  higher-rank graph approximately finite-dimensional?}, J. Funct. Anal.
  \textbf{263} (2012), no.~1, 183--215.

\bibitem[FGKP15]{FGKP}
C.~Farsi, E.~Gillaspy, S.~Kang, and J.~Packer, \emph{Separable representations,
  {KMS} states, and wavelets for higher-rank graphs}, J. Math. Anal. Appl.
  \textbf{434} (2015), no.~1, 241--270.

\bibitem[FGLP21]{FGLP}
C.~Farsi, E.~Gillaspy, N.S. Larsen, and J.~Packer, \emph{Generalized gauge
  actions on $k$-graph ${C}^*$-algebras: {KMS} states and {H}ausdorff
  structure}, Indiana Univ. Math. J. \textbf{70} (2021), 669--709.

\bibitem[GPS95]{giordano-putnam-skau}
T.~Giordano, I.F. Putnam, and C.F. Skau, \emph{Topological orbit equivalence
  and {$C^*$}-crossed products}, J. Reine Angew. Math. \textbf{469} (1995),
  51--111.

\bibitem[HLRS14]{aHLRS14}
A.~an Huef, M.~Laca, I.~Raeburn, and A.~Sims, \emph{{KMS} states on
  {$C^*$}-algebras associated to higher-rank graphs}, J. Funct. Anal.
  \textbf{266} (2014), 265--283.

\bibitem[HLRS15]{aHLRS}
A.~an Huef, M.~Laca, I.~Raeburn, and A.~Sims, \emph{{KMS} states on the
  ${C}^*$-algebra of a higher-rank graph and periodicity in the path space}, J.
  Funct. Anal. \textbf{268} (2015), 1840--1875.

\bibitem[HRSW13]{hazle-raeburn-sims-webster}
R.~Hazlewood, I.~Raeburn, A.~Sims, and S.B.G. Webster, \emph{Remarks on some
  fundamental results about higher-rank graphs and their {$C^*$}-algebras},
  Proc. Edinb. Math. Soc. (2) \textbf{56} (2013), no.~2, 575--597.

\bibitem[HS04]{hong-syman}
J.-H. Hong and W.~Szyma\'nski, \emph{The primitive ideal space of the
  ${C}^*$-algebras of infinite graphs}, Journal of the Mathematical Society of
  Japan \textbf{56} (2004), 45--64.

\bibitem[JM99]{johnson-madden}
A.~S.~A. Johnson and K.~M. Madden, \emph{The decomposition theorem for
  two-dimensional shifts of finite type}, Proc. Amer. Math. Soc. \textbf{127}
  (1999), no.~5, 1533--1543.

\bibitem[Kit98]{Kitchens}
B.P. Kitchens, \emph{Symbolic dynamics: One-sided, two-sided and countable
  state markov shifts}, Universitext, Springer-Verlag, Berlin, 1998.

\bibitem[KP00]{kp}
A.~Kumjian and D.~Pask, \emph{Higher rank graph ${C}^*$-algebras}, New York J.
  Math. \textbf{6} (2000), 1--20.

\bibitem[KP03]{KumjianPask3}
A.~Kumjian and D.~Pask, \emph{Actions of $\mathbb{Z}^k$ asscoated to
  higher-rank graphs}, Ergod.\ Thy.\ \& Dynam.\ Sys. \textbf{23} (2003),
  1153--1172.

\bibitem[KPR98]{kpr}
A.~Kumjian, D.~Pask, and I.~Raeburn, \emph{{C}untz-{K}rieger algebras of
  directed graphs}, Pacific J. Math. \textbf{184} (1998), no.~1, 161--174.

\bibitem[KPRR97]{kprr}
A.~Kumjian, D.~Pask, I.~Raeburn, and J.~Renault, \emph{Graphs, groupoids and
  {C}untz-{K}rieger algebras}, J. Funct. Anal. \textbf{144} (1997), 505--541.

\bibitem[KPS12]{Kumjian-Pask-Sims-Homology}
A.~Kumjian, D.~Pask, and A.~Sims, \emph{Homology for higher-rank graph
  ${C}^*$-algebras and twisted higher-rank graph $c^*$-algebras}, J.\ Funct.\
  Anal. \textbf{263} (2012), 1539--1574.

\bibitem[KPS15]{Kumjian-Pask-Sims-Twisted}
\bysame, \emph{On twisted higher rank graph ${C}^*$-algebras}, Trans.\ Amer.\
  Math.\ Soc. \textbf{367} (2015), 5177--5126.

\bibitem[Lis24]{listhartke}
B.~Listhartke, \emph{${C}^*$-equivalences of $k$-graph and $\mathbb{N}$-graph
  algebras through graph transformations}, Ph.D. thesis, Kansas State
  University, 2024.

\bibitem[LM95]{lind-marcus}
D.~Lind and B.~Marcus, \emph{An introduction to symbolic dynamics and coding},
  Cambridge University Press, Cambridge, 1995.

\bibitem[LP10]{Lewin-Pask}
P.~Lewin and D.~Pask, \emph{Simplicity of $2$-graph algebras associated to
  dynamical systems}, Bull.\ Malaysian Math.\ Sci.\ Soc. \textbf{33} (2010),
  177--196.

\bibitem[Mat10]{matsumoto}
K.~Matsumoto, \emph{Orbit equivalence of topological {M}arkov shifts and
  {C}untz-{K}rieger algebras}, Pacific J. Math. \textbf{246} (2010), no.~1,
  199--225.

\bibitem[ML75]{maclane}
S.~Mac~Lane, \emph{Homology}, Springer-Verlag, 1975.

\bibitem[MM14]{matsumoto-matui}
K.~Matsumoto and H.~Matui, \emph{Continuous orbit equivalence of topological
  {M}arkov shifts and {C}untz-{K}rieger algebras}, Kyoto J. Math. \textbf{54}
  (2014), no.~4, 863--877.

\bibitem[MPR14]{Maloney-Pask-Raeburn}
B.~Maloney, D.~Pask, and I.~Raeburn, \emph{Skew products of higher-rank graphs
  and crossed products by semigroups}, Semigroup Forum \textbf{88} (2014),
  162--176.

\bibitem[MRS92]{mann-raeburn-sutherland}
M.~H. Mann, Iain Raeburn, and C.~E. Sutherland, \emph{Representations of finite
  groups and {C}untz-{K}rieger algebras}, Bull. Austral. Math. Soc. \textbf{46}
  (1992), no.~2, 225--243.

\bibitem[Nas95]{nasu}
M.~Nasu, \emph{Textile systems for endomorphisms and automorphisms of the
  shift}, Mem. Amer. Math. Soc. \textbf{114} (1995), no.~546, viii+215.

\bibitem[PQR04]{pask-quigg-raeburn}
D.~Pask, J.~Quigg, and I.~Raeburn, \emph{Fundamental groupoids of
  {$k$}-graphs}, New York J. Math. \textbf{10} (2004), 195--207.

\bibitem[PRW09]{pask-raeburn-weaver}
D.~Pask, I.~Raeburn, and N.~A. Weaver, \emph{A family of 2-graphs arising from
  two-dimensional subshifts}, Ergodic Theory Dynam. Systems \textbf{29} (2009),
  no.~5, 1613--1639.

\bibitem[Rae05]{raeburn}
I.~Raeburn, \emph{Graph algebras}, CBMS Regional Conference Series in
  Mathematics, vol. 103, Published for the Conference Board of the Mathematical
  Sciences, Washington, DC; by the American Mathematical Society, Providence,
  RI, 2005.

\bibitem[RS04]{raeburn-szyman}
I.~Raeburn and W.~Szyma\'nski, \emph{{C}untz-{K}rieger algebras of infinite
  graphs and matrices}, Trans. Amer. Math. Soc. \textbf{356} (2004), no.~1,
  39--59.

\bibitem[RSY03]{RSY1}
I.~Raeburn, A.~Sims, and T.~Yeend, \emph{Higher-rank graphs and their
  ${C}^*$-algebras}, Proc. Edinburgh Math. Soc. \textbf{46} (2003), 99--115.

\bibitem[RSY04]{rsy2}
\bysame, \emph{The ${C}^*$-algebras of finitely aligned higher-rank graphs}, J.
  Funct. Anal. \textbf{213} (2004), 206--240.

\bibitem[Sch98]{Schmidt}
K.~Schmidt, \emph{Tilings, fundamental cocyles and fundamental groups of
  symbolic ${\ZZ}^d$-actions}, Ergod.\ Th.\ \& Dynam,\ Sys. \textbf{18} (1998),
  1473--1525.

\bibitem[S\o13]{sorensen-first}
A.P.W. S\o{}rensen, \emph{Geometric classification of simple graph algebras},
  Ergodic Theory Dynam. Systems \textbf{33} (2013), no.~4, 1199--1220.

\bibitem[SZ08]{Skalski-Zacharias}
A.~Skalski and J.~Zacharias, \emph{Entropy of shifts on higher-rank graph $c^*$
  -algebras}, Houston J.\ Math. \textbf{34} (2008), 269--282.

\bibitem[Tan13]{Yuxiang-Tang}
Y.~Tang, \emph{Tiling systems and $2$-graphs associated to textile systems},
  Ph.D. thesis, University of Wollongong, 2013,
  \url{http://ro.uow.edu.au/theses/4003}.

\end{thebibliography}
\end{document}